\documentclass[12pt]{article}
\usepackage{amscd,amssymb,amsmath,verbatim,amsthm, graphicx}

\usepackage{color}

\newcommand{\bS}{\mathbb S}

\newcommand{\CC}{\mathbb C}

\newcommand{\NN}{\mathbb N}
\newcommand{\RR}{\mathbb R}
\newcommand{\ZZ}{\mathbb Z}
\newcommand{\del}{\partial}

\newcommand{\phg}{{\mathrm{phg}}}

\newcommand{\be}{\beta}

\newcommand{\calA}{{\mathcal A}}
\newcommand{\calC}{{\mathcal C}}
\newcommand{\calD}{{\mathcal D}}
\newcommand{\calE}{{\mathcal E}}

\newcommand{\calG}{{\mathcal G}}

\newcommand{\calI}{{\mathcal I}}
\newcommand{\calJ}{{\mathcal J}}
\newcommand{\calK}{{\mathcal K}}

\newcommand{\calO}{{\mathcal O}}

\newcommand{\calQ}{{\mathcal Q}}

\newcommand{\calT}{{\mathcal T}}
\newcommand{\calU}{{\mathcal U}}
\newcommand{\calV}{{\mathcal V}}

\newcommand{\frakc}{\mathfrak c}
\newcommand{\frakp}{\mathfrak p}
\newcommand{\frakq}{\mathfrak q}
\newcommand{\frakr}{\mathfrak r}
\newcommand{\frakC}{\mathfrak C}

\newcommand{\conic}{\mathrm{conic}}

\newcommand{\wt}{\widetilde}
\newcommand{\wh}{\widehat}
\newcommand{\reg}{\mathrm{reg}}

\newtheorem{theorem}{Theorem}
\newtheorem*{theorem*}{Theorem}
\newtheorem{proposition}{Proposition}

\newtheorem{lemma}{Lemma}
\theoremstyle{definition}
\newtheorem{definition}{Definition}
\theoremstyle{remark}
\newtheorem*{remark}{Remark}

\title{Conical metrics on Riemann surfaces, I: the compactified configuration space and regularity}
\author{Rafe Mazzeo \\ Stanford University \and Xuwen Zhu
\\ Stanford University}

\date{}

\begin{document}
\maketitle

\begin{abstract}
We introduce a compactification of the space of simple positive divisors on a Riemann surface, as well as a compactification
of the universal family of punctured surfaces above this space. These are real manifolds with corners.  We then study
the space of constant curvature metrics on this Riemann surface with prescribed conical singularities at these divisors.
Our interest here is in the local deformation for these metrics, and in particular the behavior as conic points coalesce.
We prove a sharp regularity theorem for this phenomenon in the regime where these metrics are known to exist.
This setting will be used in a subsequent paper to study the space of spherical conic metrics with large cone angles,
where the existence theory is still incomplete. 
\end{abstract}

\tableofcontents

\section{Introduction}
This paper is a sequel to \cite{MW} by the first author and Weiss concerning the space of metrics with constant curvature and
prescribed conic singularities on a compact Riemann surface $M$.  In that paper a careful analysis was made of the deformation
theory for such metrics provided all the cone angles are less than $2\pi$. This assumption simplifies both the analytic and the
geometric considerations considerably.   Consider the space of tuples $(\frakc, \frakp, \vec{\beta}, K, A)$, where $\frakc$ is a conformal 
structure on $M$, $\frakp$ a collection of $k$ distinct points, $\vec{\beta}$ a $k$-tuple of parameters prescribing the cone angles at 
the points $p_j$, with each $\beta_j \in (0,1)$ (this corresponds to all cone angles lying in $(0,2\pi)$), and constant $K$ specifying 
the Gauss curvature and $A > 0$ specifying the area, all subject to the requirement forced by Gauss-Bonnet that 
\begin{equation}\label{e:gb}
\chi(M, \vec{\beta}) := \chi(M) + \sum_{j=1}^k (\beta_j-1) = \frac{1}{2\pi} K A.
\end{equation}
It is known through the work of several authors that to each such tuple there exists a unique metric on $M$ which has constant curvature $K$,
area $A$, and conical singularities at the points $p_j$ with cone angle $2\pi\beta_j$. There is a caveat when $K > 0$ and $k>2$
which states that in this case an extra condition is needed on the cone angles, namely that they satisfy the so-called Troyanov condition
\begin{equation}\label{e:Troyanov}
\min\{2,2\beta_j\} + k- \chi(M)> \sum_{i=1}^{k} \beta_i, \qquad j = 1, \ldots, k,
\end{equation}
which is trivial when restricted to $\vec\beta \in (0,1)^{k}$ except when $M=S^{2}$.
The main result of \cite{MW} states that the Teichm\"uller space $\calT_{\gamma, k}^{\conic}$ of all such solutions moduli the space of
diffeomorphisms of $M$ isotopic to the identity is a smooth manifold.

It is known that the situation becomes much more complicated when some or all of the $\beta_j$ are greater than $1$, at least
in the case that $K > 0$.  One classical inspiration to study this case is when each $\beta_j \in \mathbb N$, i.e., all cone angles 
are integer multiples of $2\pi$, in which case examples are easily obtained as ramified covers over other compact surfaces with 
metrics of constant curvature.  Existence and uniqueness of spaces with arbitrary cone angles and curvature $K \leq 0$ subject to~\eqref{e:gb} is relatively easy, see \cite{Mc}. Much more recent is the dramatic breakthrough by Mondello and Panov \cite{MP}, 
which establishes through beautiful and purely geometric reasoning necessary and sufficient conditions on the possible set of values 
$\vec{\beta}$ for which there exists a metric with constant curvature $1$ (a spherical metric) on $S^2$ with these prescribed cone angle 
parameters.   

This last-cited paper leaves open some fundamental questions. The one which interests us here is to describe the space of
points $\frakp$ and cone angle parameters $\vec{\beta}$ for which there exist spherical metrics with this data prescribing
the conic singularities.  The answer is complicated and (at least to our understanding) not completely explicit.  An initial
hope might be to show that the space of all solutions (mod diffeomorphisms) is a smooth manifold. From this one might
then further try to apply various techniques from geometric analysis to count solutions. 
Unfortunately, for spherical cone metrics with cone angles greater than $2\pi$, this space fails to be smooth on certain subvarieties.
One of our goals, which will be addressed in a sequel to this paper,  is to understand this failure more precisely.  
Briefly, however, the key observation is that if $g$ is a spherical metric for which the deformation theory is obstructed, it is possible
to consider this solution in a larger moduli space where the deformation theory is unobstructed. This broader setting consists
of letting certain of the cone points $p_j$ split into clusters of cone points with smaller angles. This is an analytic
manifestation of one of the important steps in the geometric arguments of Mondello and Panov \cite{MP}. 

The analysis needed to carry this out turns out to be somewhat complicated and requires the development of some machinery
which will occupy a significant part of this paper.  We regard this machinery of independent interest, and expect that it may 
be a useful tool in studying various other analytic problems involving geometric objects which are singular or otherwise
distinguished at families of points which can cluster. We mention in particular the study of solutions of the two-dimensional
vortex equation on a Riemann surface, as well as the study of analytic constructions related to holomorphic quadratic
differentials in relationship to the Hitchin moduli space.  

\subsection{Outline of results}
This paper has two main parts. In the first part, \S\ref{s:resolution}, we develop these general ideas, which involve the construction of a resolution
via real blow-up of the configuration space of $k$ points on $M$ and of the universal family of marked surfaces over this
blown-up configuration space.  Similar constructions are classical in algebraic geometry if one uses complex blowups, but
our use of real blow-ups and other $\calC^\infty$ methods here lead to spaces which are compact manifolds with corners 
which encode the different modes of clustering of these $k$ points.  This construction is closely related to other
recent work, notably the ongoing work of Kottke and Singer \cite{KS} on the compactification of the 
moduli space of monopoles in $\RR^3$. We describe the construction of the extended configuration space $\calE_{k}$ (the base manifold) in \S\ref{ss:Ek}, and the resolution of the universal family $\calC_{k}$ (the total space) in \S\ref{ss:Ck}. We then give the two simplest examples when $k=2$ or $3$ in~\S\ref{ss:k2} and \S\ref{ss:k3}.  In~\S\ref{ss:bd} we give a description of the combinatorial structure of the boundary faces for the generic $k$-point case, and in particular show that we obtain a b-fibration. 

In the second part we consider the space of metrics with constant curvature and prescribed conic singularities;
in this paper we restrict attention to flat and hyperbolic metrics with no angle constraints, and spherical metrics with cone angles less than $2\pi$.  Our main theorem here is a new regularity result, which we give a sketch below, and refer to Proposition~\ref{flat}, Theorem~\ref{t:hyp} and Theorem~\ref{t:sph} for the precise statements.
\begin{theorem*}
The family of hyperbolic or flat metrics with conic singularities with arbitrary cone angle, or spherical conic surfaces with cone angles less than $2\pi$ and satisfying the
Troyanov constraint, lifts to be 
polyhomogeneous, a natural generalization of smoothness, on the compactified universal family of curves $\calC_{k}$ over this 
extended configuration space $\calE_{k}$. 
\end{theorem*}

In~\S\ref{s:geometry} we set up the geometric process of merging cone points, described both locally and globally. In~\S\ref{s:analysis} we recall some facts on analysis of conic elliptic operators. \S\ref{s:flat}--\S\ref{s:sph} give the proof of the main theorem in the flat, hyperbolic, and spherical cases. In each of the three cases we study solutions to a family of singular elliptic PDEs on the new space $\calC_{k}$ constructed above. For the flat case in~\S\ref{s:flat},  the proof is done by a direct computation and we show that the solutions given by Green's functions are polyhomogeneous. The proof for the nonzero curvature cases are more involved. In~\S\ref{ss:hyp2} we prove the result when two cone points merge, which involves  first constructing approximate solutions to arbitrarily high order, followed by using maximum principle to get the exact solution, and finally using commutator argument to show conormality and polyhomogeneity. In~\S\ref{ss:hypk} the case with more cone points merge is proved by a similar argument but with a more involved process in constructing the approximate solutions. In~\S\ref{s:sph} the spherical case is proved in a similar way, except that maximum principle no longer holds and is replaced by invertibility of the linearized operator.

These results are first steps in our program to understand the entire moduli space of constant
curvature conical metrics on surfaces.  The explanation of the extended configuration family, which
is the setting for this regularity theory, is already of interest, and its definition is vindicated by
our main regularity theorem. 
In a second paper we will employ this machinery to understand features of the moduli space of spherical cone 
metrics where the cone angles are greater than $2\pi$. Our eventual goal is to understand the stratified nature 
of these moduli spaces in sufficient detail that we can produce a count of solutions.  We also hope to 
reach a better correspondence between the classical results and tools used to study these problems and the ones developed here. 

\bigskip

\noindent{\bf Acknowledgements:}  The authors are happy to acknowledge useful conversations with
Misha Kapovich, Richard Melrose and Michael Singer. The first author was supported by the NSF grant DMS-1608223. 

\section{Resolution of point configurations}\label{s:resolution} 
The first part of this paper focuses on a rather intricate geometric construction, which is a resolution via real blow-up of 
the configuration space of $k$ points on a compact Riemann surface $M$, as well as the resolution of the universal family 
over this space.  

To be more specific, let $\calD_k(M)$ denote the space of nonnegative divisors on $M$ of total degree $k$. Thus a point
of $\calD_k(M)$ consists of an ordered $k$-tuple of not necessarily distinct points $p_1, \ldots, p_k \in M$.  Although it 
is more common to study this using algebro-geometric ideas, we take a decidedly real and $\calC^\infty$ approach. 
Away from coincidences where two or more of the $p_j$ are the same, $\calD_k(M)$ is a copy of $M^k$ with all the partial 
diagonals removed. 
Of course, $\Sigma_k$ acts freely on
this open set. Our first goal is to define a {\it real} compactification $\calE_k(M)$ of this open dense set in $\calD_k(M)$, 
which we call the extended configuration space. This compactification is a manifold with corners, which comes equipped
with a blowdown map $\beta: \calE_k \to M^k$. We next consider the product $\calE_k(M) \times M$; this is a trivial bundle 
which has a tautological multi-valued section $\sigma$: if $q \in \calE_k(M)$ and $\frakp = \beta(q)$, then $\sigma(\frakp)$ is 
the divisor $\frakp$ considered as a subset of $M$.  We shall define a resolution of this object, again as a manifold
with corners, using a suitable blowup of the graph of $\sigma$; this is called the extended configuration family 
and denoted $\calC_k(M)$.  This is not quite a fibration over $\calE_k(M)$ since certain fibers are `broken'; instead it is a
slightly more general type of map called a $b$-fibration, a natural extension of the notion of fibrations to the category of 
manifolds with corners. See the appendix for a general discussion about b-fibrations and manifolds with corners. 

\subsection{The extended configuration space $\calE_{k}$}\label{ss:Ek}
Our first goal is to define a good compactification for $\calD_k^s$ the space of all `simple' divisors, defined in~\eqref{e:dks} below. 
We begin with some notation. 
Suppose first that $\calI \subset \{1, \ldots, k\}$ is an index set with $|\calI|\geq 2$. The $\calI^{\mathrm{th}}$ partial
diagonal is the subset 
\[
\Delta_{\calI} = \{ \frakp \in \calD_k:  p_{i} = p_{j}, \forall i,j \in \calI\}.
\]
There is a reverse partial order of diagonals corresponding to the inclusion of index sets, 
$$
\calI \subset \calJ \Leftrightarrow \Delta_{\calI} \supset \Delta_{\calJ}.
$$

The union of the two index sets is defined in the usual sense. If $\calI$ and $\calJ$ have at least one common element, 
one can identify the partial diagonal corresponding to their union as the intersection of their diagonals:
$$
\Delta_{\calI\cup \calJ}=\Delta_{\calI} \cap \Delta_{\calJ}, \text{ if } \calI\cap \calJ \neq \emptyset.
$$
The assumption of nonempty intersection guarantees that the intersection of two diagonals is still a diagonal. 
Otherwise, if $\calI \cap \calJ = \emptyset$, there is a strict inclusion
$$
\Delta_{\calI\cup \calJ}\subsetneq \Delta_{\calI} \cap \Delta_{\calJ}. 
$$

On the other hand, when $|\calI\cap \calJ|\geq 2$, then $\Delta_{\calI \cap \calJ}$ 
is the smallest diagonal containing both $\Delta_{\calI}$ and $\Delta_{\calJ}$.

We also let 
\[
\Delta_{\calI}^0 = \Delta_{\calI} \setminus \left(\operatorname{\cup}_{\calJ \supsetneq \calI} \Delta_{\calJ} \right).
\]
It is then clear that the ensemble $\{\Delta_\calI^0\}$ is a stratification of $M^k = {\calD_k}$; the dense open stratum equals
\begin{equation}\label{e:dks}
{\calD_k^s} = M^k \setminus \left( \operatorname{\cup}_{\calI} \Delta_\calI^0 \right).
\end{equation}

To resolve the point collisions, we resolve all the partial diagonals; this is done by blowing up the diagonals iteratively 
in order of decreasing index set. In other words, if $\calI\subsetneq \calJ$, then $\Delta_{\calJ}$ is blown up before $\Delta_{\calI}$. 
There is still some freedom in the order of blow up, since the inclusion of index sets only provides a partial order. Below we show 
that the final space is well-defined and does not depend on which specific order to blow up.

Let $X$ be a manifold with corners, containing two $p$-submanifolds, $Y_1$ and $Y_2$. (A $p$-submanifold $Y \subset X$ is 
defined to be a submanifold for which some neighborhood $\calU \supset Y$ is diffeomorphic as a manifold with
corners to the normal bundle $NY$.) The iterated blowup $[X; Y_1; Y_2]$ is the manifold with corners obtained as 
follows. First blow up $Y_1$ in $X$ to obtain a space $[X; Y_1]$. Now lift $Y_2 \setminus (Y_1 \cap Y_2)$ to this space
and take its closure. Finally, take the blowup of this lift in $[X; Y_1]$. In general the resulting space depends
on the order in which these blowups are taken; the reverse order may result in a nondiffeomorphic space. There are 
two special situations where the order does not matter: the first is if $Y_1 \subset Y_2$, and the second is if $Y_1$ and $Y_2$
meet transversely so that their normal bundles are disjoint (away from the zero section). 
Similarly, when there are more $p$-submanifolds $\{Y_{i}\}_{i=1}^{k}$, $[X; Y_{1};\dots; Y_{k}]$ is well-defined if the following is true: for any $Y_{i}$ and $Y_{j}$, either $Y_{i}\subset Y_{j}$, or $Y_{i}$ and $Y_{j}$ are transversal.

In the prescription for blowing up the partially ordered sequence of partial diagonals, we are blowing up 
these partial diagonals by inclusion, i.e., we always blow up the `smaller' submanifolds first. However, we must check
that the second criterion about transversality is satisfied.  
\begin{lemma}
Let $\calI$ and $\calJ$ be any two index sets. Suppose that $X_{\calI, \calJ}$ is the manifold with corners obtained by blowing up
all the partial diagonals $\Delta_{\calK}$ in $M^k$ for which $\calK \supset \calI$ and $\calK \supset \calJ$. 
Then the lifts of $\Delta_{\calI}$ and $\Delta_{\calJ}$ are transverse in $X_{\calI, \calJ}$. 
\end{lemma}
\begin{proof}
When $\calI\cap \calJ=\emptyset$, we can choose the complex coordinates $s_1, \ldots, s_k\}$ such that 
$\Delta_{\calI}=\{s_1 = \ldots = s_p =0\}$, $\Delta_{\calJ}=\{s_{p+1} = \ldots = s_q =0\}$. Clearly then $\Delta_{\calI}$ and 
$\Delta_{\calJ}$ intersect transversely.

On the other hand, if $\calI \cap \calJ \neq \emptyset$, then $\calI' = \calI \setminus (\calI\cap\calJ)$ and 
$\calJ' \setminus (\calI \cap \calJ)$ are disjoint.  Choosing coordinates so that 
$\{s_1 = \ldots = s_p = s_{\ell+1} = \ldots = s_k = 0\}$ on $\Delta_\calI$ and $\{s_{p+1} = \ldots = s_q = s_{\ell+1} = \ldots = s_k\}$ 
on $\Delta_{\calJ}$, then it is not hard to check that the lifts of these submanifolds to the blowup around
the set $\{s_1 = \ldots = s_q = s_{\ell+1} = \ldots = s_k = 0\}$ are disjoint, hence transverse by default. 


\end{proof}

Using this Lemma, we may now proceed through this sequence of blowups to obtain the extended (ordered) configuration space
\begin{equation}
{\calE_{k}}=[{\calD_{k}}; \cup_{\calI } \Delta_{\calI} ].
\end{equation} 
One consequence of this operation is that the action of symmetric group $\Sigma_{k}$ on $\calD_{k}$ is resolved.
\begin{proposition}
The symmetric group $\Sigma_k$ acts freely on ${\calE_k}$. 
\end{proposition}
\begin{proof}
The fixed points of $\Sigma_{k}$ on ${\calD_{k}}$ are precisely the partial diagonals, and moreover, the isotropy
group at $\Delta_{\calI}$ is a subgroup of the isotropy group at $\Delta_{\calJ}$ when $\Delta_{\calI} \supset \Delta_{\calJ}$, 
or equivalently when $\calI \subset \calJ$.  Thus our iterative blowup corresponds to the blowup which resolves
this group action, define by blowing up the fixed point sets ordered by reverse isotropy type inclusion, and it is not
hard to check in this case that the isotropy groups of the lifted group action are all trivial. This is a special case of a 
more general iterated blowup considered by Albin and Melrose \cite{AlbinMelrose} which resolves a general Lie group action.
\end{proof}

The space $\calE_k$ appears rather complicated at first glance, but the combinatorial structure of its faces mirrors
the partially ordered set of subsets $\{\calI\}$ of $\{1, \ldots, k\}$.  For each element $\calI$ of this set, there
is a boundary hypersurface $F_{\calI}$ of $\calE_k$ generated by blowing up $\Delta_{\calI}$. We also denote by
$\rho_\calI$ the boundary defining function for this face. Identifying the interior of $\calE_k$ with the nonsingular 
part of $M^k$ away from all the diagonals, we see that $\rho_\calI$ provides a measurement of the radius of a cluster of $|\calI|$
coalescing points.  

\subsection{The extended configuration family $\calC_{k}$}\label{ss:Ck}
We next consider the universal family over $\calE_k$. This is a space $\calC_k$ equipped with a $b$-fibration
\[
\widehat{\beta}: \calC_k \longrightarrow \calE_k,
\]
such that for each $\frakp \in \calE_k^\reg (\simeq \calD_{k}^{s})$, the fiber $\wh{\beta}^{-1}( \frakp)$ is the surface $M$ blown up at
the points of $\frakp$.  We point out that $\wh{\beta}$ is different than the map $\beta$ discussed earlier, which
is the blowdown $\calE_k \to M^k$.   In the following, it is often simpler to refer to points $\frakp$ on $\calE_k$;
these are, however, elements of the compactified configuration space, so the actual divisor, or $k$-tuple of
points on $M$, is really the image under $\beta$ of this point.   In any case, with this understanding, 
if $\frakp$ lies in one of the boundary faces of $\calE_k$, then the fiber $\wh{\beta}^{-1}(\frakp)$ is a union of 
surfaces with boundary which encode the various ways the corresponding cluster of points can come together.

To define this universal family, we begin with the trivial fibration $\calE_{k}\times M \to \calE_k$, and let $z$ be 
a generic point on the fiber, which we may as well assume is a local holomorphic coordinate there. We wish to 
resolve the graph of the canonical `section' $\sigma$ of this bundle: 
$$
\{(\frakp,z) \in \calE_{k}\times M: z \in \sigma(\frakp)\}. 
$$
Since $\sigma$ is multi-valued, we must first blow up the crossing loci, which are contained in the graphs of $\sigma$ over 
the faces of $\calE_k$.  More specifically, if $\frakp$ lies in a face $F_\calI$, we write $\sigma^\calI(\frakp)$ for the 
corresponding `coincidence point' $p_{i_1} = \ldots = p_{i_r}$, $\calI = (i_1, \ldots, i_r)$ (so $\sigma(\frakp)$ 
has $r$ copies of this point and $k-r$ other points), and then define the coincidence set 
\begin{equation}
F_{\calI}^\sigma=\{\rho_{\calI}=0, z= \sigma^\calI(\frakp) \}.
\label{liftedincset}
\end{equation}

The space $\calC_k$ may now be defined by iteratively blowing up this collection of submanifolds with respect to the 
partial order on index sets, culminating at the last step in the blowups of the nonsingular parts of the graph of $\sigma$, i.e.,
the submanifolds $F^\sigma_i = \{ z = p_i\}$, $i = 1, \ldots, k$, where $\frakp$ does not lie in any partial diagonal. Altogether,
\begin{equation}
\calC_k = \big[ \calE_k \times M; \{ F^\sigma_\calI\} \big].
\label{extconfspace}
\end{equation}

The following lemma shows, just as for $\calE_k$, that the blowup is well-defined. 
\begin{lemma}
The lifts of $F^\sigma_\calI$ and $F^\sigma_\calJ$ are transverse after $F_{\calI\cup \calJ}^\sigma$ has
been blown up. In particular, the lifts of $F^\sigma_{i}$ and $F^\sigma_{j}$ do not meet when $i \neq j$. 
\end{lemma}
\begin{proof}
As before, this follows from the fact that
$$
F_{\calI}^\sigma \cap F_{\calJ}^\sigma = F_{\calI\cup \calJ}^\sigma
$$  
when $\calI\cap \calJ\neq \emptyset$, while $F_{\calI}^\sigma$ and $F_{\calJ}^\sigma$ are transverse 
away from $F_{\calI\cup \calJ}^\sigma$ when $\calI\cap\calJ=\emptyset$.  The last assertion is obvious. 
\end{proof}

\subsection{The simplest case, $k=2$} \label{ss:k2}
The description of the boundary faces of $\calE_k$ and $\calC_k$ is somewhat complicated and at first glance confusing, 
so to warm up, we present the cases $k = 2$ and $3$ in some detail since it is possible to see what is going on without
too much work then. 

The space of ordered divisors ${\calD_2}$ is simply $M^2$, and there is a single diagonal $\Delta_{12} = \{p_{1}=p_{2}\}$, hence
\[
{\calE_2} = [ {\calD_2} ; \Delta_{12}].
\]
Here and below we keep the subscript $12$ to foreshadow the general case. From local coordinates $(z_1, z_2)$ near 
$(p_0, p_0) \in \Delta$, we determine the center of mass $\zeta = \frac12(z_1 + z_2)$ and displacement $w = \frac12 (z_1 - z_2)$, 
so that 
\begin{equation}\label{e:zetaw}
z_1 = \zeta + w,\ z_2 = \zeta - w. 
\end{equation}
The blowup amounts to setting $w = \rho_{12} e^{i\theta}$, $\theta \in [0,2\pi]$ and adding the face $\rho_{12} = 0$. 

The front face $F_{12}$  is then a possibly nontrivial circle bundle over the diagonal. Indeed, it is the unit normal bundle of the 
diagonal in $M^2$, and hence has Euler characteristic equal to $\chi(M)$.  In any case, we have coordinates $(\theta, \zeta)$ on $F_{12}$
and a full set of coordinates $(\rho_{12}, \theta, \zeta)$ near this face in the blowup. 

The symmetric group $\Sigma_2$ interchanges the two coordinates $(z_1, z_2)$, and hence sends $\zeta \mapsto \zeta$,
$w \mapsto -w$.  In local coordinates, $(\rho_{12}, \theta, \zeta) \mapsto (\rho_{12}, \pi+\theta, \zeta)$, and it is easy
to see that this is a free action. 

The extended configuration family is now obtained from the product $\calE_2 \times M$ by blowing up
in succession the two submanifolds
\begin{multline*}
F_{12}^\sigma = \{ (\rho_{12}=0, \theta, \zeta, z): z = \zeta\} \subset F_{12} \times M, \ \  \mbox{and} \\ 
F_1^\sigma \cup F^\sigma_2 =  \{ (\frakp, \sigma(\frakp)): \frakp \in \calD_2\}. 
\end{multline*}

For the first of these blowups, introduce spherical coordinates $(R_{12}, \Omega)$ around the codimension three submanifold
$\{\rho_{12} = 0$, $\zeta = z\}$, so $R_{12} \geq 0$ and $\Omega \in \bS^2_+$.  We write 
\begin{equation}\label{S2+}
\Omega= (\rho_{12}, z-\zeta)/R_{12} =  (\sin \omega,  \cos \omega \, e^{i\phi} )
\end{equation} 
where $\omega \in [0,\pi/2]$ and $\phi \in [0,2\pi]$, and so 
\begin{equation}\label{e:Romega}
\rho_{12} =  R_{12} \sin \omega, \ \ z = \zeta + R_{12} \cos \omega e^{i\phi}.
\end{equation}
We also set $z-\zeta = re^{i\phi}$. 

The face created by this blowup, which we call $\frakC_{12}$, is the total space of a fibration $\pi_{12}: \frakC_{12} \to F_{12}$,
with each fiber a copy of $\bS^2_+$. The preimage of a point $(0,\theta, \zeta) \in F_{12}$ is the union of two manifolds with 
boundary: the first is the blowup of $M$ around the point $\zeta$, $[M, \{\zeta\}]$ and the second is $\bS^2_+$. These meet 
along their common boundary, which is a circle.  From~\eqref{e:Romega},
\begin{equation}\label{bdf}
\pi_{12}^{-1} ( \rho_{12} )  = A R_{12} \omega,
\end{equation}
where $A$ is a strictly positive smooth function. The significance of this computation is that the lift of the defining function
for $F_{12}$ equals the product of defining functions for the fiber $M$ blown up at $\zeta$ and the half-sphere, up to a nonvanishing
smooth factor. That is, the boundary defining functions satisfy the b-fibration condition.

\begin{figure}[h]
\includegraphics[width=\textwidth]{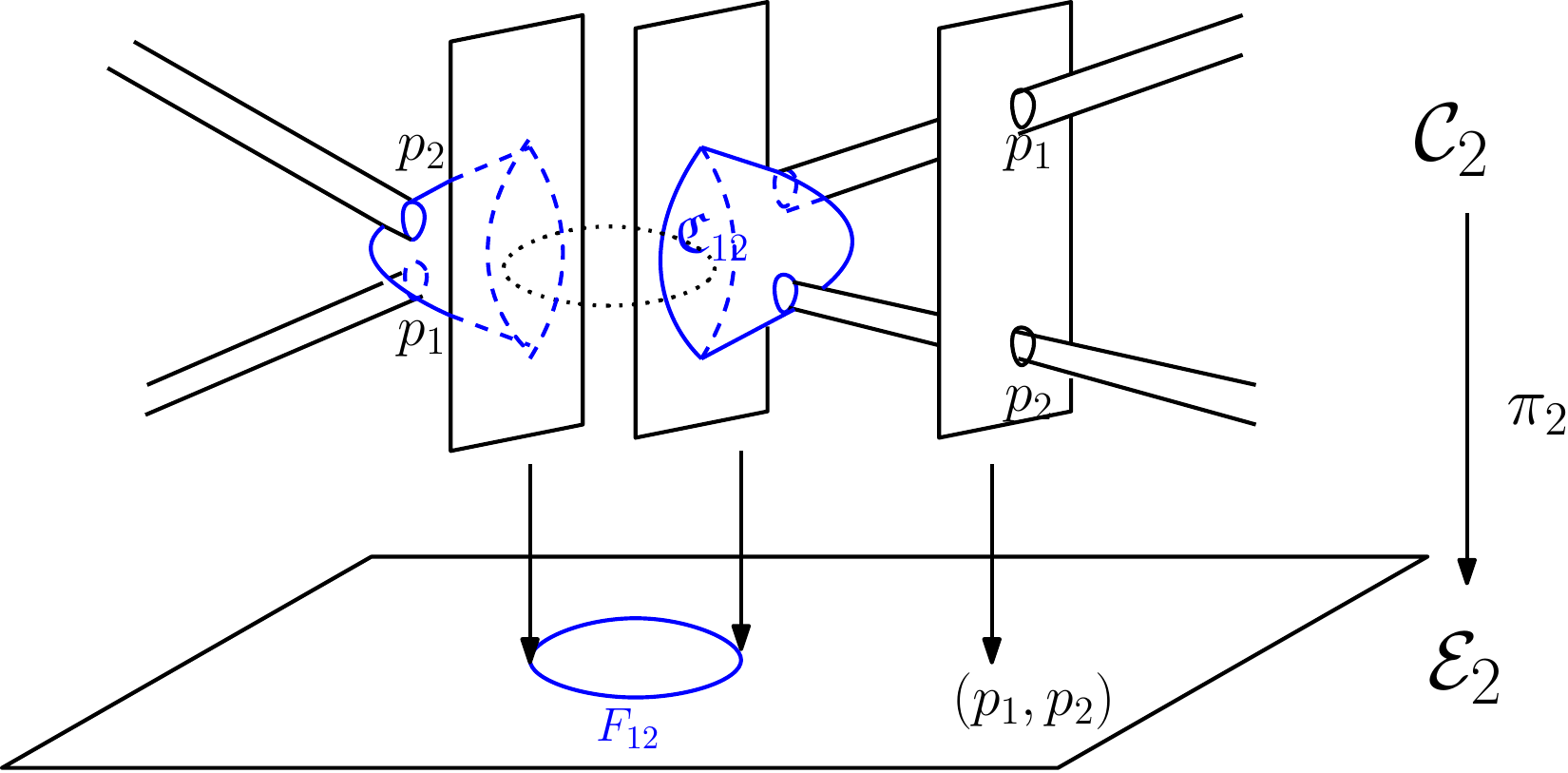}
 \caption{The singular fibration of $\calC_{2}\rightarrow \calE_{2}$. Here we removed the center of mass $\zeta$, and the coordinate in the base is $w$, see~\eqref{e:zetaw}. The boundary face in the base, $F_{12}$, is parametrized by $\theta$ such that $w=\rho_{12}e^{i\theta}$. When $\theta$ goes from $0$ to $\pi$, the two points $p_{1}$ and $p_{2}$ on the fiber interchange. }
 \label{f:C2}
\end{figure}

We now turn to the second blowup. Consider the graph of $\sigma$, 
\begin{equation}\label{intersection}
\begin{array}{l}
F_{1}^{\sigma}=\{z=z_{1}\}=\{z= \zeta + \rho_{12} e^{i\theta}\}=\{\omega=\pi/4, \phi=\theta\} \\
F_{2}^{\sigma}=\{z=z_{2}\}=\{z= \zeta - \rho_{12} e^{i\theta}\}=\{ \omega=\pi/4, \phi=\theta+\pi\}. 
\end{array}
\end{equation}
These two components intersect $\frakC_{12}$ at two disjoint copies of $\mathbb{S}^{1}$. Their intersection with each
$\pi_{12}^{-1}( 0, \theta, \zeta)$ consists of two points on each $\bS^2_+$ fiber. In other words, the boundaries of these 
components are each circles, and there is one point of each of these circles in each $\bS^2_+$ fiber. These intersection
points are given explicitly in~\eqref{intersection}.  Observe that as $\theta$ goes from 0 to $\pi$, these two points 
interchange.  The final configuration space is equal to
\[
{\calC_2} := \big[ {\calE_2} \times M;  F_{12}^\sigma;  F_1^\sigma \cup F_2^\sigma \big].
\]

Note finally that since each $F^{\sigma}_{i}$ projects surjectively to $\calE_{2}$, the relation describing the pullback
of boundary defining functions stays the same as for the space before this last blowup.  

This space is equipped with a $b$-fibration 
\[
\pi: \calC_2 \longrightarrow \calE_2.
\]
Over a regular point $\frakp \not\in \Delta_{12}$, the preimage $\pi^{-1}(\frakp)$ is a copy of $M$ blown up at the two
points of $\frakp$. On the other hand, the preimage $\pi^{-1}( 0, \theta, \zeta)$ of a point on $F_{12}$ is the union
of $M$ blown up at the single point $\zeta$ and the half-sphere $\bS^2_+$ blown up at two points,
$[ \bS^2_+; \{\omega = \pi/4, \phi = \theta\} \cup \{\omega = \pi/4, \phi = \theta + \pi\} ]$. See Figure~\ref{f:C2} for an illustration of the fibration. 

\subsection{The case $k=3$}\label{ss:k3}
The next case illustrates the iterative nature of the general construction. 

As before, we work in local coordinates $(z_1, z_2, z_3) \in M^3$ near a point $(p_0, p_0, p_0) \in \Delta_{123}$.

The initial blowup of the central diagonal $\Delta_{123}$ in ${\calD_{3}}$ results in the space $[{\calD_3}; \Delta_{123}]$.
Just as for ${\calE_2}$, this is a manifold with boundary; its boundary, or front face at this first step, is called 
$F_{123}$ and is a sphere bundle over $\Delta_{123}$, but now with three-dimensional spherical fibers. 
In choosing coordinates, it seems to  be more convenient to break the symmetry by using classical Jacobi coordinates 
(introduced originally to study the $N$-body problem in celestial mechanics). Thus we define a center of mass $\zeta$ 
in the first two variables, as well as two displacement variables
\[
\zeta = \frac12 (z_1 + z_2), \ w_1 = \frac12 (z_1 - z_2),\ w_2 = z_3 - \frac12 (z_1 + z_2). 
\]
In these coordinates, 
\[
\Delta_{123} = \{ \zeta+ w_1 = \zeta - w_1 = \zeta + w_2\} = \{w_1 = w_2 = 0\}.
\]
The resolution is the blowup of the origin in $\CC_{w_1 w_2}$, which is captured by spherical coordinates in the fibers 
of the normal bundle: 
$$
N_\zeta \Delta_{123} \cong \RR^4_{w_1, w_2} \ni \rho_{123} \Theta, \ \ (\rho_{123}, \Theta) \in \RR^{+}\times \bS^{3}.
$$
To proceed, write
\[
\Theta = (e^{i\phi_1}  \cos \theta , e^{i\phi_2} \sin \theta), \ \ (\theta, \phi_{1}, \phi_{2}) \in [0,\pi/2] \times [0,2\pi]^2,
\]
corresponding to the fact that $\bS^3$ is a join of two circles, so that
\begin{equation}\label{S3}
w_1 = \rho_{123} \cos \theta \, e^{i\phi_1}, \ \ w_2 = \rho_{123} \sin \theta\, e^{i\phi_2}.
\end{equation} 
These coordinates are singular at $\theta = 0, \pi/2$.  To remedy this near $\theta = \pi/2$, for example, 
we use instead the projective coordinate $\wt{w}_1 = w_1/\rho_{123}$ along with $\phi_2$. 

In these new coordinates, the lifts of the partial diagonals in a neighborhood of $\Delta_{123}$ have the form
\begin{equation}\label{Delta12}
\begin{array}{l}
\Delta_{12}=\{z_{1}=z_{2}\}=\{w_1=0\}= \{\theta=\pi/2\} = \{\wt{w}_1 = 0\} \\
\Delta_{13}=\{z_{1}=z_{3}\}=\{w_{1}=w_{2}\}=\{\theta=\pi/4, \phi_{1}=\phi_{2} \}\\
\Delta_{23}=\{z_{2}=z_{3}\}=\{w_{1}=-w_{2}\}=\{\theta=\pi/4, \phi_{1}= - \phi_{2}\}.
\end{array}
\end{equation}
The additional expression for the lift of $\Delta_{12}$ is included because $\theta = \pi/2$ is a singular
locus for the $(\theta, \phi_1, \phi_2)$ coordinate system. Each of these lifts is locally a product $\bS^{1}\times \RR^{+}$ 
near $\bS^3 \times \{0\}$,  and their intersections with each fiber of $F_{123}$ are three disjoint circles. We then blow these 
up in any order to obtain
$$
{\calE_{3}}=[{\calD_{3}};\Delta_{123}; \Delta_{12}  \cup \Delta_{13} \cup \Delta_{23}]
$$ 

There are three new front faces, $F_{ij}$, each with a boundary defining function $\rho_{ij}$, and each (locally) 
diffeomorphic to $\bS^{1} \times \bS^{1} \times \RR^{+}$.  The intersection $F_{ij} \cap F_{123}$ is a torus 
$\mathbb{S}^{1}\times \mathbb{S}^{1}$. 

Let us illustrate this geometry near $F_{12}\cap F_{123} $ in coordinates.  From~\eqref{S3} and~\eqref{Delta12}, 
$\phi_{2}\in \bS^{1}$ and $\rho_{123} \in \RR^{+}$ parametrize $\Delta_{12}$, while we set
$\wt{w}_1 := w_1/\rho_{123}$ as a coordinate for the normal bundle.  Blowing up at $\wt{w}_1 = 0$ amounts 
to writing $\wt{w}_1 = \rho_{12} e^{i\theta_{12}}$. This is of course really the same as setting $\rho_{12} = \cos\theta$, 
$\theta_{12} = \phi_1$, but the new monikers have been introduced to conform with the general notation when $k > 3$.
Altogether, we have a full set of local coordinates
\begin{equation}\label{crdnate12}
(\zeta, \rho_{12}, \theta_{12}, \phi_{2}, \rho_{123})\in \RR^2 \times \RR^{+}\times \bS^{1}\times \bS^{1}\times \RR^{+}, 
\end{equation}
These are related to the original coordinates by
\begin{equation}
z_1 = \zeta + \rho_{123}\rho_{12} \, e^{i\theta_{12}},\ z_2 = \zeta - \rho_{123}\rho_{12} \, e^{i\theta_{12}}, \ z_3 = \zeta + \rho_{123} \sqrt{1-\rho_{12}^{2}}\, e^{i\phi_{2}}.
\end{equation} 

Now consider the configuration space ${\calE_{3}}\times M$. Near $(F_{12}\cap F_{123}) \times M$ we extend the 
local coordinates \eqref{crdnate12} to 
\begin{equation}\label{local12}
(\zeta, \rho_{12}, \theta_{12}, \phi_{2}, \rho_{123}, z) \in \RR^2 \times \RR^{+}\times \bS^{1}\times \bS^{1}\times \RR^{+} \times \RR^2,
\end{equation}
where $z$ is a local coordinate on $M$. The coordinate $\zeta$ is sometimes omitted below. The first step is to resolve the 
`central' coincidence set, cf.\ \eqref{liftedincset},
$$
F_{123}^{\sigma}=\{z-\zeta= \rho_{123}=0\},
$$
which we do by introducing spherical coordinates $(z-\zeta, \rho_{123})=R_{123}\Omega_{123}$, where $R_{123} \geq 0$ and 
\begin{equation}\label{S2}
\bS^2_+ \ni \Omega_{123} = (\cos \omega \, e^{i\phi}, \sin \omega ), \ \omega \in [0,\pi/2], \ \phi \in \mathbb{S}^{1} .
\end{equation}
This introduces the front face $\frakC_{123} = \{R_{123} = 0\}$. 
Coordinates at this stage are
$$
(\zeta, R_{123}, \Omega_{123}, \rho_{12}, \theta_{12}, {\phi_{2})\in \RR^{2} \times \RR^{+}\times \bS^{2}_{+}\times \RR^{+}\times \bS^{1}\times \bS^{1}}.
$$ 
Note also that
\begin{equation}\label{generic}
\begin{split}
z_1 - \zeta &=-(z_2 - \zeta) =R_{123}\sin \omega  \,\rho_{12}\, e^{i\theta_{12}},  \\
& z_3 -\zeta =R_{123}\, \sin \omega\, \sqrt{1-\rho_{12}^{2}}\, e^{i\phi_{2}},
z-\zeta=R_{123}\, \cos \omega\,  e^{i\phi}.
\end{split}
\end{equation}

As before, the coordinates $(\omega, \phi)$ become degenerate at $\omega = \pi/2$, so in a neighborhood of
this locus we replace these by the projective coordinate $\wh{z} = (z-\zeta)/R_{123}$. 

Now fix a point $q = (\rho_{12}, \theta_{12}, \phi_2)$ on the front face $F_{123}$ in the  base ${\calE}_3$. The fiber
above $q$ after the blowup above is a union of $\RR^2$ blown up at the origin (or more globally, the surface $M$ blown
up at the point $\zeta$) and a half-sphere, parametrized by the coordinates $(\omega, \phi)$ (or $\wh{z}$) above. These meet along 
the circle $\{\omega = 0\}$.  If $\pi_3$ is the blowdown map, then, analogous to \eqref{bdf}, 
\begin{equation}
\pi_{3}^{*}\rho_{123}=R_{123}\sin \omega=A R_{123} \, \omega
\end{equation}
for some smooth nonvanishing function $A$. 

The partial coincidence sets $F_{ij}^\sigma$ intersect the fiber over $q$ only in the interior of the hemisphere. 
In local coordinates, 
\begin{equation}
\begin{array}{l}
F_{12}^{\sigma}=\{z =z_{1}=z_{2}\}= \{\rho_{12}=0,  \omega=\pi/2 \} = \{ \wt{w}_{1} = 0, \wh{z} = 0\}\\
F_{13}^{\sigma}=\{z =z_{1}=z_{3}\}=\{\rho_{12}= \sqrt{2}/2, \tan \omega=\sqrt{2}, \theta_{12}=\phi_{2}=\phi \}\\
F_{23}^{\sigma}=\{z=z_{2}=z_{3}\}=\{\rho_{12}=\sqrt{2}/2, \tan \omega =\sqrt{2},\theta_{12}+\pi=\phi_{2}=\phi\}
\end{array}
\end{equation}
The final expression for the lift of $F_{12}^{\sigma}$ is included because of the degeneracy of the other coordinate
system at $\rho_{12} = 0$, $\omega = \pi/2$. 

Now blow up $F_{12}^\sigma$.  This introduces a new front face $\frakC_{12}$, the fibers of which (for each $(\zeta, q)$) 
are hemispheres ``on top'' of the previous hemispheres.   Carrying out the analogous blowup for the two other
partial diagonals as well leads to the collection of front faces $\frakC_{ij}$, each diffeomorphic to 
$\RR^{+}\times \mathbb{S}^{1} \times \bS^{1}\times \mathbb{S}^{2}_{+}$. Each $\frakC_{ij}$ fibers over the front face $F_{ij}$
in the base ${\calE_3}$ by projecting off the final $\mathbb{S}^{2}_{+}$. Again, this is very similar to the two-point case. 

\begin{figure}[h]
\centering
\includegraphics[width = 0.5\textwidth]{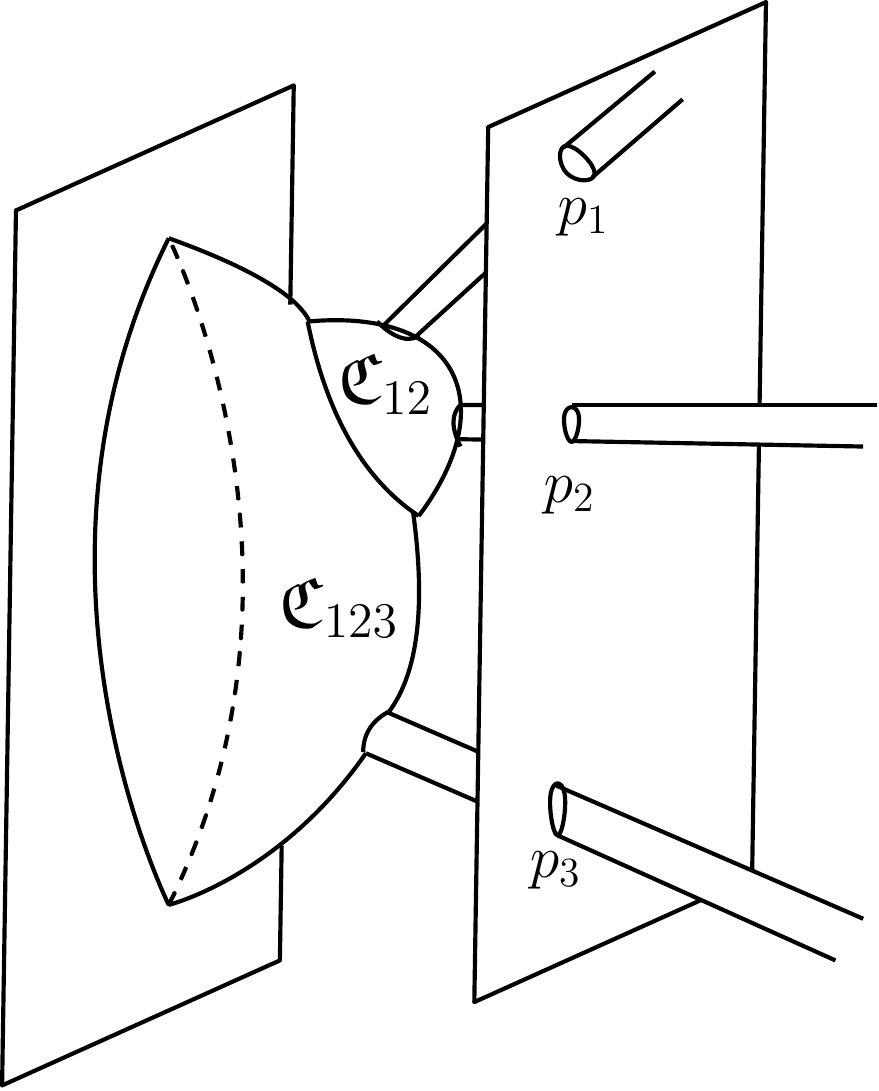}
\caption{One of the singular fibers in $\calC_{3}$, where two of the points collide faster than the third one}
\label{figurek=3}\end{figure}

Let $\Omega_{12} \in \bS^{2}_{+}$ be the variable in the fiber of $\frakC_{12}$ and $R_{12}$ the corresponding boundary defining function.
Thus 
$$
(\cos \omega\,  e^{i\phi}, \rho_{12})=R_{12} \Omega_{12}
$$
and 
\begin{equation}
\Omega_{12}=(\cos \omega_{12}\, e^{i\phi_{12}}, \sin \omega_{12} ), \ \omega_{12} \in [0,\pi/2], \ \phi_{12} \in \mathbb{S}^{1}. 
\end{equation}
Coordinates near $\frakC_{12}$ are then
$$
(\zeta, R_{123}, R_{12}, \Omega_{12}, \theta_{12}, {\phi_{2})  \in \RR^{2}\times \RR^{+}\times \RR^{+} \times \bS^{2}_{+}\times \bS^{1}\times \bS^{1}},
$$
and these translate to the original coordinates by
\begin{multline}\label{blowdown}
z_1-\zeta=-(z_2- \zeta) =R_{123} \sqrt{1-(R_{12}\cos \omega_{12})^{2}}  \, R_{12}\sin \omega_{12}\, e^{i\theta_{12}},\\
z_3 -\zeta=R_{123}\sqrt{1-(R_{12}\cos \omega_{12})^{2}}\sqrt{1-(R_{12}\sin \omega_{12})^{2}}\, e^{i\phi_{2}}, \\
z - \zeta =R_{123}R_{12}\cos \omega_{12} \, e^{i\phi_{12}}.
\end{multline}

The boundary defining function relation near the corner $\frakC_{123} \cap \frakC_{12}$ is given by
\begin{multline}\label{e:3pt}
\pi_{3}^{*}\rho_{123}=R_{123}\sqrt{1-(R_{12}\, \cos \omega_{12})^{2}}  \sim R_{123}, \\
\pi_{3}^{*} \rho_{12}=R_{12} \, \sin \omega_{12} \sim R_{12} \, \omega_{12}.
\end{multline}

We conclude by blowing up the lifts of the three submanifolds $F_{i}^{\sigma}=\{z=z_i\}$. In a generic 
fiber over $F_{123}$, (i.e., away from all the partial diagonals), $F_{i}^{\sigma}$ meets the front face $\frakC_{123}$
on the hemisphere at three distinct points. Using the coordinates~\eqref{generic}, 
\begin{equation}
\begin{array}{l}
F_{1}^{\sigma}=\{\cot \omega =\rho_{12}, \theta_{12}=\phi\}\\
F_{2}^{\sigma}=\{\cot \omega =\rho_{12}, \theta_{12}=\phi+\pi\}\\
F_{3}^{\sigma}=\{\cot \omega =\sqrt{1-\rho_{12}^{2}}, \phi_{2}=\phi\}.
\end{array}
\end{equation}
 
On approach to the resolved partial diagonals, two of the points $z_i$ converge to one another on $M$. However, their
lifts converge to distinct points on a hemispherical fiber of the innermost front face. In other words, 
$F_{i}^{\sigma}$ and $F_{j}^{\sigma}$ both intersect $\frakC_{ij}$ while if $k$ is the third value distinct
from $i, j$, then $F_{k}^{\sigma}$ does not. In terms of the coordinates~\eqref{blowdown}, we have
\begin{equation}
\begin{array}{l}
F_{1}^{\sigma}=\{ \omega_{12}=\pi/4 ,\ \theta_{12}=\phi_{12}\}\\
F_{2}^{\sigma}=\{\omega_{12}=\pi/4, \ \theta_{12}=\phi_{12}+\pi\}\\
F_{3}^{\sigma}=\{\cot \omega=\sqrt{1-\rho_{12}^{2}}, \ \phi_{2}=\phi_{12}\}.
\end{array}
\end{equation}

The roles of $z_1$, $z_2$ and $z_3$ are interchangeable in this whole discussion, even though
the Jacobi coordinates $w_1$ and $w_2$ break this symmetry. it is also not hard to check that the 
symmetric group $\Sigma_3$ acts freely on ${\calC}_3$.

As this case makes clear, the geometry illustrated in Figure 2 requires compound singular coordinate
transformations, which quickly become quite involved.  The optimistic interpretation is that the
compound asymptotics of solutions to natural elliptic operators which appear later in this paper
are captured entirely by the intricate but still quite comprehensible geometry of this iterated blowup.

\subsection{Boundary faces of $\calE_k$ and $\calC_k$}\label{ss:bd}
The previous discussion suggests that while it is possible to write out the iterated polar coordinate
systems corresponding to iterated blowups, these become prohibitively complicated after a few steps.
The special cases above should be used to gain intuition about the general case.   In this section 
we undertake a more systematic study of the boundary faces and corners of $\calE_k$ and $\calC_k$.
The goal is to explain some general features of the geometry of these boundary faces, which is done 
through a combination of invariant and local coordinate reasoning.

\medskip

\noindent {\bf Faces of $\calE_k$:}  As per our initial discussion in \S 2.1, the boundary hypersurfaces $F_{\calI}$ 
of $\calE_k$ are in bijective correspondence with subsets $\calI \subset \{1, \ldots, k\}$ by way of the associated 
partial diagonals $\Delta_\calI$.  The set of all such subsets $\{\calI\}$ is ordered by inclusion; thus $\{1, \ldots, k\}$ is the 
unique maximal element and the minimal elements are the singletons $\{j\}$. 
The associated directed graph $\calG_k$ has a vertex for each $\calI$ and we choose the direction on the edges
to point from $\calI$ to $\calJ$ when $\calI \supset \calJ$. This `flips' the inclusion ordering, so we should
think of $\{1, \ldots, k\}$ as the bottom vertex (or root) of the tree, with the edges flowing upward. 

We first consider the interiors of each face $F_{\calI}$.  Write $\Delta_{\calI}^0$ for the open dense subset of elements in 
$\Delta_\calI$ where {\it only} the points $p_i$, $i \in \calI$, coincide.  Fixing $\calI$, let $M^k_\calI$ be the subset 
of $M^k$ where we remove all partial diagonals $\Delta_\calJ^0$ with $\calJ \cap \calI \neq \emptyset$.  The blowup 
\[
[M^k_\calI;  \Delta_{\calI}^0]
\]
can be regarded as a relatively open subset of $\calE_k$. If $|\calI| = \ell$, then the normal bundle of $\Delta_\calI^0$ in $M^k_\calI$ 
is naturally identified with $(TM)^{\oplus (\ell-1)}$. This means that the new front face coming from the blowup of $\Delta_\calI^0$ 
is a bundle over a dense open subset in $M^{k-\ell+1} \cong \Delta_\calI$ with fiber $\bS^{2\ell-3}$. 

Reorder the indices so that $\calI = \{1, \ldots, \ell\}$ and write $z' = (z_1, \ldots, z_\ell)$, $z'' = (z_{\ell+1}, \ldots , z_k)$;
thus $z'' \in M^{k-\ell }$ and $\Delta_\calI$ is the complete diagonal in the $z'$ subsystem. Subdivide $z'$ further, 
writing it as the sum of a center of mass $\zeta^\calI(z')$ and $w'=(w_{1}, \dots, w_{\ell}) \in \CC^{|\calI|}$, 
where $\sum_{i \in \calI} w_i = 0$, so $w'$ lies in a space of real dimension $2|\calI| - 2$. The blowup affects only the
$w'$ coordinates, and we see in this coordinate description that the new front face is locally a fibration 
over $\RR^{2k-2\ell+2}$ with fiber $\bS^{2\ell-3}$. 

We now turn to the more difficult task of understanding the structure of each boundary hypersurface of $\calE_k$ as a manifold 
with corners. By the iterative nature of these spaces, it is sufficient to focus on the innermost face $F_{1 \ldots k}$, which for 
convenience we denote by $F_{\max}$ below. We also write $\calI_{\max} = \{1, \ldots, k\}$. 
\begin{proposition}
The boundary faces and corners of $F_{\max}$ are in bijective correspondence with the connected trees $T \subset \calG_k$
such that any minimal vertex $\calI \in T$ has $|\calI| \geq 2$. 
\label{cornersoffaces}
\end{proposition}
The correspondence is that if $Z$ is any corner of this principal face, then the vertices of the associated tree $T = T_Z$
are the subsets $\calI$ for which the face created by blowing up $\Delta_{\calI}$ contains $Z$ in its closure.  

To get a feel for this, consider a few examples. First, a boundary hypersurface of $F_{\max}$ is the intersection
of precisely two boundary hypersurfaces of $\calE_k$: $F_{\max}$ and some other $F_{\calI}$. The corresponding 
tree has two vertices, $\calI_{\max}$ and $\calI$, connected by an edge.  Slightly more generally, if
$\Delta_{\max} \subset \Delta_{\calI_1} \subset \Delta_{\calI_2} \subset \ldots \subset \Delta_{\calI_\ell}$,
then $Z = F_{\max} \cap F_{\calI_1} \cap \ldots \cap F_{\calI_\ell}$ is a (nonempty) corner, and $T_Z$ is the
tree $\calI_{\max} \to \calI_1 \ldots \to \calI_\ell$.  On the other hand, if $\calI$ and $\calJ$ are disjoint, then
the lifts of $\Delta_{\calI}$ and $\Delta_{\calJ}$ to $[M^k; \Delta_{\max}]$ intersect transversely, so the blowups
of these two partial diagonals may be taken in either order, and the intersection $F_{\max} \cap F_{\calI}\cap F_{\calJ}$ 
is a corner of codimension two in $F_{\max}$.  The corresponding tree has root $\calI_{\max}$ connected to 
$\calI$ and $\calJ$, but these two vertices are at the same level and not connected to one another. 
There is a similar description for any collection $\calI_1, \ldots, \calI_s$ of disjoint subsets. 

To prove Proposition~\ref{cornersoffaces}, we recapitulate the construction of $\calE_k$, emphasizing
what happens at each stage of the iterative process.   A basic principle here is that the lifts of the partial diagonals 
at each stage are `$p$-submanifolds'. Recall that if $X$ is a manifold with corners, then a submanifold $Y$ 
is called a $p$-submanifold if, near every point $q \in Y$, there is a local coordinate system $(x,y) 
\in (\RR^+)^\ell \times \RR^m$ for $X$ such that $Y$ is given by setting some number of the $x_i$ 
and $y_j$ to $0$. In other words, locally, $Y$ is a product inside of $X$. Denote by $M^k(r)$ the space 
obtained after blowing up all partial diagonals $\Delta_\calI$ with $|\calI| \geq k-r$.  Also, let $F(r)$ 
denote the central front face in $M^k(r)$, so $F(k) = F_{\max}$. 

At the initial step, $M^k(0) = [M^k; \Delta_{\max}]$ and $F(0)$ is an $\bS^{2k-3}$ bundle over $M \cong \Delta_{\max}$. 
The partial diagonals $\Delta_\calI$ with $|\calI| = k-1$ lift to $M^k(0)$ to a disjoint collection of four-dimensional 
submanifolds with boundary; these intersect each $\bS^{2k-3}$ fiber of $F(0)$ in $k$ disjoint copies of $\bS^1$.  The space 
$M^k(1)$ is obtained by blowing up these lifted partial diagonals, and doing so yields a manifold with corners of 
codimension two. The new boundary hypersurface has $\bS^{2k-5}$ fibers over a (disconnected) four-dimensional base. 
The codimension two corner of $M^k(1)$ is the boundary of $F(1)$; it has $k$ components, each of which  is 
a bundle with fiber $\bS^1 \times \bS^{2k-5}$ over $\Delta_{\max}$. 

Continuing on, the partial diagonals $\Delta_{\calI}$ with $|\calI| = k-2$ lift to $p$-submanifolds in $M^k(1)$. 
These lifts intersect each $\bS^{2k-5}$ fiber in the new boundary face (over the $(k-1)$-fold diagonals) 
in copies of $\bS^1$, and intersect the fibers of $F(1)$ in a copy of $\bS^3$.  The important observation 
is that even in low dimension, i.e., $k=4$, these $3$-spheres in $\bS^{2k-3}$ do not intersect in $M^k(1)$ 
by virtue of the fact that we have already blown up their intersection loci, the union of $\bS^1$.
When we blow up this set of lifted partial diagonals, the new boundary hypersurface of $F(2)$ is 
fibered by copies of $ [\bS^3; \sqcup \, \bS^1] \times \bS^{2k-7}$, and the corners of this boundary face, 
which are now corners of codimension $3$ in $M^k(2)$, has fibers equal to $\bS^1 \times \bS^1 \times \bS^{2k-7}$. 

As we proceed further in this construction, the thing to note is that at each stage, the lift of each $\Delta_\calI$,
$|\calI| = k - r  - 1$, to $M^k(r)$ is a $p$-submanifold. Furthermore, if $\calI$ and $\calJ$ are two subsets of
size $k-r-1$, then either $\calI \cap \calJ$ share an element in common, or the intersection is empty.  In the 
first case, it is straightforward to check that their lifts to $M(r)$ are in fact disjoint.  In the second case, these two
lifted diagonals intersect transversely, and hence it is not necessary to blow up their intersection since we can 
blow these up in either order to obtain the same result. 


Now let us return to Proposition \ref{cornersoffaces}.  As already described, given a corner $Z$ of $F_{\max}$, we
may associate to it a subgraph $T_Z \subset \calG_k$. Clearly $T$ is connected since all branches lead
to the root $\calI_{\max}$ (since $Z \subset F_{\max}$ in particular).    Next, suppose that $\calI \cap \calJ \neq \emptyset$,
but neither one contains the other.  Then the lifts of $\Delta_{\calI}$ and $\Delta_{\calJ}$ would already be separated 
in $M(r)$, $r = k - |\calI \cup \calJ|$, which is where $\Delta_{\calI \cup \calJ}$ is blown up, hence $F_\calI \cap F_\calJ = \emptyset$.  
This shows that it is impossible for there to exist $\calI_1, \calI_2 \supset \calK$ with  $\calI_1 \not\supset \calI_2$ and 
$\calI_2 \not\supset \calI_1$.  This proves that $T_Z$ is a tree.

Conversely, if $T$ is any connected tree in $\calG_k$ emanating from $\calI_{\max}$ (and which does not terminate
at a node with $|\calI| = 1$), then we must show that
\[
Z = \bigcap_{\calI \in T} F_{\calI} 
\]
is nonempty.  Since $T$ is a tree, if $\calI$ and $\calJ$ lie on the same branch, then one of these is a proper 
subset of the other, while if they lie on different branches, then they are disjoint. We can divide $T$ into 
branches $B_i$, and it follows from the earlier discussion that the intersection of the $F_{\calI}$ along a 
branch $B_i$ is a nonempty corner $Z_i$; on the other hand, if the branch $B_2$ is rooted at some
vertex of another branch $B_1$, then we can reduce back to the case of one node splitting in two as in
the previous paragraph to see that the corresponding corners $Z_1$ and $Z_2$ intersect transversely 
in a nontrivial subset. We conclude using induction on the number of branches.   This proves the result. 

\medskip

\noindent{\bf Faces of $\calC_k$: }  We wish to carry out a similar analysis of the faces of $\calC_k$. As before, we 
proceed inductively, so it suffices to analyze the structure of the central face $\frakC_{\max}$. 


Recall that the construction of $\calC_k$ involves iteratively blowing up the coincidence sets $F_{\calI}^\sigma$ in 
$\calE_k \times M$ defined in~\eqref{liftedincset}. Let $\calC_k(r)$ denote the space obtained by blowing up all such coincidence sets with 
$|\calI| = j$, $k \geq j \geq k-r$, in order of decreasing cardinality, and also write $\frakC(r)$ for
the central front face of $\calC_k(r)$. 

If $\rho_{\max}$ is a defining function for $F_{\max}$ in $\calE_k$, then $\calC_k(0)$ is the blowup of the set
$\{\rho_{\max} = 0,\ z = \sigma^{\max}(\frakp)\}$. The front face $\frakC(0)$ is fibered over $F_{\max}$
with fiber a half-sphere $\bS^2_+$. The lifts of the coincidence sets $F_{\calI}^\sigma$  intersect
$\frakC(0)$ as (usually mutually intersecting) $p$-submanifolds.   We denote the faces created by 
blowing up $F_\calI^\sigma$ by $\frakC_\calI$; with a slight abuse of notation, we often do not distinguish 
between this face at some intermediate step of the construction and at the final stage in $\calC_k$.

\begin{figure}[h]
\centering
\includegraphics[width = 0.65\textwidth]{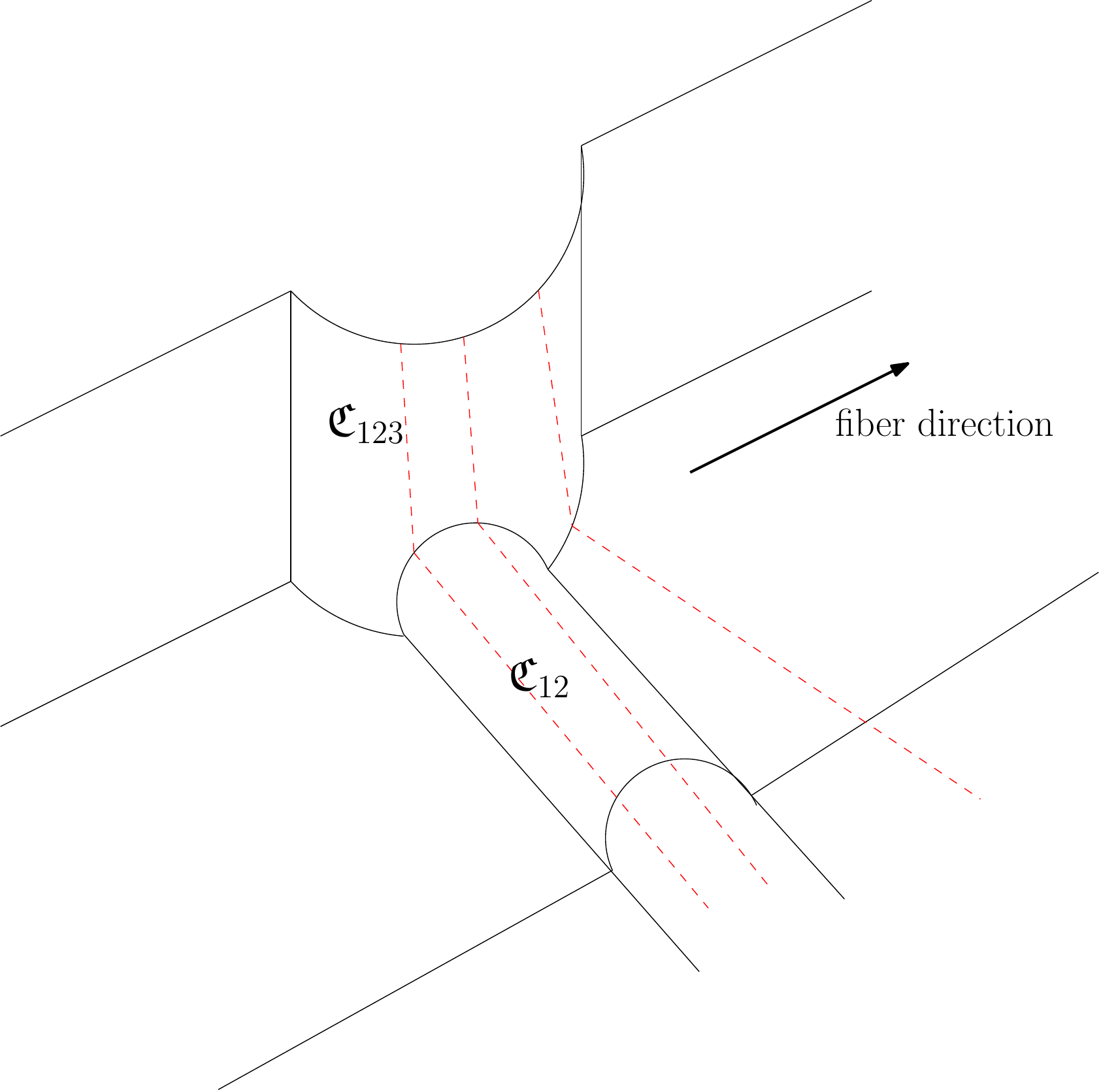}
\caption{The intersection of front faces $\frakC_{123}$ and $\frakC_{12}$, where the red lines are where the singleton coincidence sets $F_{i}^\sigma, i=1,2,3$ meet the front faces}
\label{f:blowup}
\end{figure}

For expository purposes we jump immediately to the final step and consider the 'generic' region, i.e., the preimage of the 
interior of $\calE_k$. The only coincidence sets in this region are those with $|\calI| = 1$, which means
that the blowups in these region are unaffected by any of other blowups. These (singleton) coincidence sets $F_i^\sigma$ 
are of codimension two in $\calE_k \times M$, so the faces $\frakC_i$ are fibered by copies of $\bS^1$.  In particular, 
the intersection of $\frakC_i$ with $\frakC_{\max}$ intersects each $\bS^2_+$ fiber in a $\bS^1$.  In other words,
the portion of $\frakC_{\max}$ over the interior of $F_{\max}$ (i.e., the principal front face of
$\calC_k$ away from all the faces lying over partial diagonals) is fibered by copies of $\bS^2_+$ 
blown up at $k$ distinct points, the locations of which are determined by the corresponding point in $F_{\max}$
(the `directions of approach' of the coalescing cluster of $k$  points).

Now return to the construction in the proper order, and consider the passage from $\calC_k(0)$ to $\calC_k(1)$. 
This involves blowing up the coincidence sets $F_{\calI}^\sigma$ with $|\calI|  = k-1$. Over the interior of the faces 
$F_\calI \times M$ in $\calE_k \times M$, the picture is analogous to the blowup at the principal front face in 
$\calE_{k-1}\times M$: indeed, the corresponding point in $M^k$ is, up to reordering, of the form $(p, \ldots, p, p_k)$. 
As $\rho_{\max} \to 0$, the points $p$ and $p_k$ coalesce. The coincidence set $F_{\calI}^\sigma$ is a $p$-submanifold 
which intersects the codimension two corner $(F_{\max} \cap F_{\calI}) \times M$. When it is blown up, the new face 
is a bundle with $\bS^2_+$ fibers over $F_{\calI}$, uniformly to this intersection.  

The new feature is that the fibers of $\calC_k(1) \to \calE_k$ over the corner $F_{\max} \cap F_{\calI}$ are each a `tower' 
of hemispheres of height two, i.e., two copies of $\bS^2_+$, the second one attached along its boundary to the circle 
created by blowing up a point in the first $\bS^2_+$. The submanifold $F_{k}^\sigma$ (corresponding to the singleton 
set $\{k\} = \{1, \ldots, k\}\setminus \calI$) intersects each of these fibers 
at a point of $\bS^2_+$ away from the this second hemisphere.  On the other hand, the other $F_{i}^\sigma$ (away 
from the boundaries of the corner $F_{\max} \cap F_{\calI}$) intersect the $\bS^2_+$ fibers of the face over $F_{\calI}$ 
in $(k-1)$ distinct points, so after blowing these up, the fibers are hemispheres blown up at $k-1$ points.  At 
the intersection with $F_{\max}$, the fibers are each a tower of two hemispheres, the inner one blown up at
$k-1$ distinct points and the outer one at one additional point.   All of this has been illustrated earlier in
Figure~\ref{f:blowup} for the case $k=3$. 

The rest of the construction follows the same pattern.  When we perform a blow up in $\calC_k(r)$ of a coincidence set
$F_{\calI}^\sigma$ for some $\calI$ with $|\calI| = k - r -1$, having previously blown up all coincidence
sets $F_{\calJ}^\sigma$ with $|\calJ| > k-r-1$, then the interior of this face, i.e., the portion lying over the
interior of $F_{\calI}$, is again fibered by copies of $\bS^2_+$, and in this region the blowups of the sets 
$F_i^\sigma$, $i \in \calI$, produce hemisphere fibers blown up at $|\calI|$ distinct points. 
These fibers intersect the fibers of the previous faces $\frakC_{\calJ}$ in similar ways, creating a new level in 
the tower of hemispheres over those corners. 

We now state the final result, which is a description of all boundary faces and corners of $\frakC_{\max}$.
We have already described this face over the interior of $F_{\max}$: it is a fibration with each fiber a copy 
of $\bS^2_+$ blown up at $k$ distinct points.  More generally, at any corner $Z$ of $F_{\max}$, consider 
the preimage $\frakC_{\max}(Z)$, i.e., the portion of the boundary of $\frakC_{\max}$ lying over $Z$.  
This is a tower of hemispheres, each blown up at a set of points, so that altogether $k$ points 
in this entire tower are blown up.  Figure~\ref{treetower} illustrates a typical scenario when $k=5$. 

\begin{figure}[h]
\centering
\includegraphics[width = 0.7\textwidth]{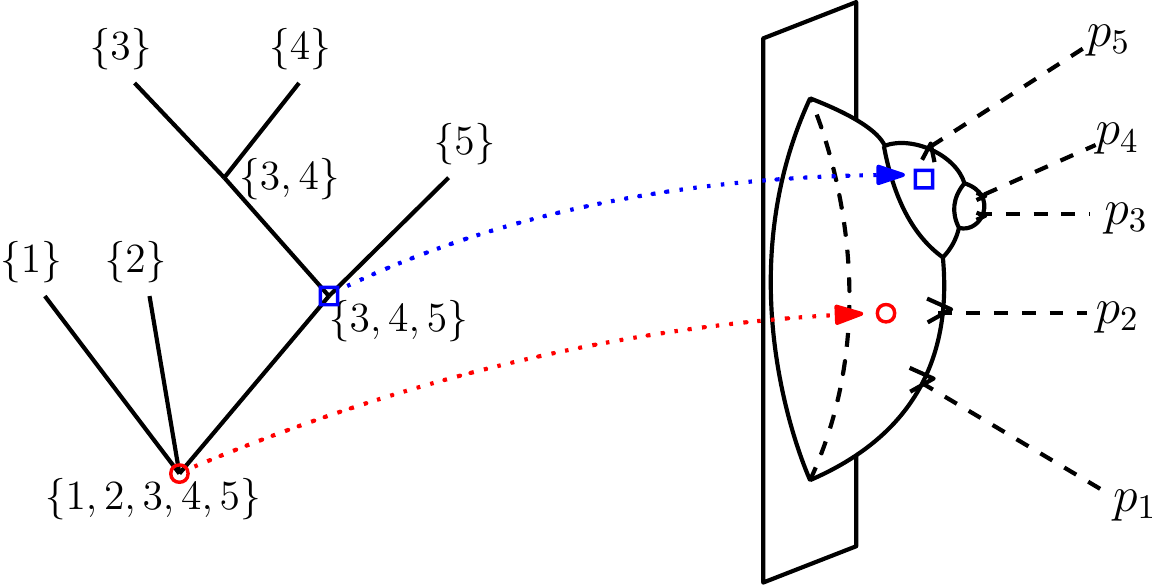}
\caption{A tree structure encodes the clustering of bubbles}
\label{treetower}
\end{figure}

We can classify the towers of hemispheres which arise in this way.  Recall first that associated to the corner $Z$ is a 
tree $T$ in the power set $\calG_k$.  We claim that the half-spheres in this tower correspond precisely to the nodes 
in $T$.   Indeed, $T$ consists of all multi-indices $\calI$ with $|\calI| > 1$ such that $F_{\calI} \supset Z$. 
We use induction on the height of the tree. The case of height $0$ and $1$ were described earlier. The 
same argument applies when we pass from the space obtained up to height $r$ as we take the blowups 
corresponding to the nodes in $T$ at height $r+1$. This shows that there is a half-sphere corresponding to 
each node of $T$.  Conversely, the blowups required to construct this tower of half-spheres corresponds exactly 
to the sequence of blowups in this induction.    In summary, we have the
\begin{proposition}
The boundary faces and corners of $\frakC_{\max}$ in $\calC_k$ are in bijective correspondence
with pairs $(T, \calI)$, where to each corner $\mathfrak Z$ of $\frakC_{\max}$, $T$ is the tree associated 
to the corner $Z$ of $F_{\max}$ under $\mathfrak Z$ and $\calI$ is a node of $T$. 
\end{proposition}

We also consider {\em augmented trees}, which are simply trees as before, but now allowing the terminal nodes
to consist of single-element sets.  These nodes correspond to the faces $\mathfrak C_j$ of $\calC_k$, but
in particular to the boundary components of the hemispheres in the penultimate faces. 

\medskip

\noindent {\bf Cluster decompositions:}  It is helpful to understand both $\calE_k$ and $\calC_k$ using the intuition that 
each point $\frakp \in M^k$ and $(\frakp, z) \in M^k \times M$ has a neighborhood $\calU$ in which there is a well-defined 
decomposition of the points into clusters $\calQ_1, \ldots, \calQ_{k'}$ for some $k' \leq k$.  Each cluster $\calQ_i$ has a 
center of mass $\zeta_i$, and the $k'$-tuple $(\zeta_1, \ldots, \zeta_{k'})$ determines a divisor, as do the points in each $\calQ_i$. 
This cluster decomposition picture changes as the points move around, and the corners of $\calE_k$ and $\calC_k$ quantify
precisely where the clusters trade points.   

We say a bit more about this now. Recall the blowdown map $\beta: \calE_k \to M^k$.  Fix $\epsilon>0$ and 
a collection of disjoint index sets $\calI_* = \{\calI_j\}$, $j = 1, \ldots, \ell$ such that $\cup \calI_j = \{1, \ldots, k\}$. 
Now define the open sets 
\begin{multline}\label{e:cluster}
\calU_{\calI_{*}, \epsilon} = \beta^* \{ \frakp \in M^{k}: \ d(p_{i'},p_{j'})>\epsilon, \text{ if } i'\in \calI_{i}, j'\in \calI_{j}, \forall i\neq j\}. \hfill 
\end{multline}
(The precise distance function used here is not important.) 
Thus in each $\calU_{\calI_*, \epsilon}$, coalescing occurs only within each cluster. Furthermore, any point $q \in \calE_k \setminus F_{\max}$
lies in one of these open sets.  Indeed, $\beta(q) = \frakp$ lies in some corner $Z$ of $\calE_k$. Take the corresponding tree $T$.
Note that since $q \not\in F_{\max}$, the lowest element of $T$ is some $\calI \neq \calI_{\max}$.  Denote by
$\calI_{1}, \dots, \calI_{\ell'}$ the vertices at the other extreme, i.e., the highest elements of $T$.  If $\cup_{j=1}^{\ell'} \calI_{j} 
\neq \calI_{\max}$, set $\ell=\ell'+1$ and define $\calI_{\ell}=\calI_{\max} \setminus \cup_{j=1}^{\ell'} \calI_{j}$. These correspond
to the `free' points which are not in any larger cluster. Otherwise, let $\ell=\ell'$.  
\begin{lemma}\label{l:tree}
There exists $\epsilon>0$ such that $q \in U_{\calI_{*},\epsilon}$.
\end{lemma}
\begin{proof}
By the definition of $T$, $\beta(q) = \frakp$ lies in the intersection of diagonals $\{\Delta_{\calI_{j}}\}_{j=1}^{\ell-1}$.  We must
first choose $\epsilon$ so that the $\epsilon$-ball around $\frakp$ in $M^k$ does not intersect any other partial diagonals.
Within this ball, clustering only happens amongst the points $p_i$ with $i$ lying in a single index set $\calI_{j}$. 
Since $q \notin F_{\max}$, there are at least two clusters, i.e., $\ell\geq 2$, so we can suppose that $2\epsilon$ is smaller
than the minimal distance between the various clusters. It is then easy to see that $\frakp$ is contained in the open 
set  $U_{\calI_{*},\epsilon}$. 
\end{proof}

These cluster decompositions allow us to describe some further useful structure of the spaces $\calE_{k}$ and $\calC_{k}$.
\begin{lemma}\label{l:tree}
\begin{enumerate}
\item[(1)] For any $q \in \calE_k \setminus F_{\max}$ and associated neighborhood 
$\calU_{\calI_{*}, \epsilon}$, there exists a smaller neighborhood $\calU \subset U_{\calI_{*}, \epsilon}$ of $q$ 
which is a product of neighborhoods $\calU_i$ in $\calE_{k(i)}$ for smaller values of $k(i)$. More specifically,
$$
\calU = \Pi_{j=1}^{\ell} \calU_j\ \ \mbox{where}\quad \calU_j \subset \calE_{|\calI_{j}|}.
$$
\item[(2)]  For $q \in \calE_k$ and $\calU$ a product neighborhood of $q$ as in (1), there exist $\ell$ open sets 
$\calV_{i} \subset \calC_{k}$, $\cup_{i=1}^{\ell} \calV_{i}=\wh{\beta}^{-1}(\calU)$, so that in each $\calV_{i}$, the fibration $\wh{\beta}$ is a product form
$$
\left( \wh{\beta}_{|\calI_i|}:  \tilde \calV_{i} \to \calU_{i}\right)\times \Pi_{j\neq i} \calU_{j},
$$ 
where $\tilde \calV_{i}\subset \calC_{|\calI_{i}|}$.
\item[(3)]  Consider any point $q \in \calC_k$, and let $(T, \calJ)$ be the data encoding the hemisphere tower
on which it sits. We assume that $q$ lies in the interior of the maximal hemisphere of this tower. 
Let $\calJ'=\calJ\setminus \cup_{\calJ \supset \calI \in T}\calI$ be the set of free points in $\calJ$. If $\calJ' \neq \emptyset$,  
then there exists a neighborhood of $z$ with product decomposition 
$$
\wh{\beta}^{-1} (\calU_{\calI_{*}, \epsilon}) \supset  \calV_{\calC} \times \calV_{\calE}, \ \ \calV_{\calC}\subset \calC_{|\calJ'|}, \ 
\calV_{\calE} \subset \calE_{k-|\calJ'|}. 
$$
\item[(4)] If there are no free points, i.e., $\calJ' = \emptyset$, then there is a decomposition
$$
q \in  \calV_{\calC} \times \calV_{\calE},\ \ \calV_{\calC}\subset M, \ \ \calV_{\calE} \subset \calE_{k}.
$$
\end{enumerate}
\end{lemma}
\begin{proof}
For (1), the decomposition separates points into independent clusters.  If $q \in U_{\calI_{*}, \epsilon}$, then it has 
a neighborhood which does not intersect any $F_{\calJ}$ except when $F_{\calJ}, \calJ \subset \calI_{i}$ for some $i$. 
All the possible blowups in this neighborhood occur within each cluster, so one can write $q=(q_{1}, \dots, q_{\ell})$ 
where $q_{j} \in E_{\calI_{j}}$. The conclusion is now clear.

For (2), the clusters in $\calU_{*,\epsilon}$ are separated by distance at least $\epsilon$ so this decomposition exists. 
Restricting to each $\calV_{j}$, then only $p_{i}, i\in \calI_{j}$ are included, so locally the map $\calC_k \to \calE_k$
splits in such a way that the lift of projection $(\frakp, z)\to \frakp$ restricts to the lift of $\{(\frakp_{\calI_{j}},z) \to 
\frakp_{\calI_{j}}\} \times \frakp_{\{1,\dots, k\}\setminus \calI_{j}}$. 

It is possible to refine the decomposition in (2) further when moving deeper into the tree. To prove (3) and (4), 
take a point in $\calC_k$ lying above $(\frakp,z)$ and let $\frakC_{\calJ}$ be the boundary face, some point on 
the interior of which projects to $(\frakp,z)$. By the definition of $T$, the only boundary faces intersecting 
$\frakC_{\calJ}$ in this fiber correspond to the vertices $\calI \supset \calJ$. Since $z$ lies in the interior of 
a boundary face, it has a neighborhood which does not intersect any of these $\frakC_{\calI}$. 
If $\calJ'\neq \emptyset$, there are $|\calJ'|$ free points contained in this region, hence only $\frakC_{i}, i\in \calJ'$ 
intersect this neighborhood. There are no other free points in this neighborhood,  and it contains no other boundary faces. 
In the first case, write $\frakp=(\frakp_{\calJ'}, \frakp_{\calJ''})$ with $\calJ''=\{1,\dots, k\}\setminus \calJ'$. 
Only the coincidence set $F_{\calJ'}^\sigma$, i.e., where $z=p_{i}, i \in \calJ'$, intersects this neighborhood. 
Thus it is given by a neighborhood of $\calC_{|\calJ'|}$ which does not intersect any partial diagonals. The
other $\frakp_{\calJ''} \in \calE_{k-|\calJ'|}$ fill out the remaining base variables. For (4), there are no free points 
in this neighborhood, hence the fiber is just a neighborhood of $M$. Using the same reasoning as above
yields the product decomposition.
\end{proof}


\medskip

\noindent {\bf $b$-fibrations:}
The final and crucial fact we need is the following result.
\begin{proposition} 
The natural projection $\wh{\beta}: \calC_k\rightarrow  \calE_k$  is a $b$-fibration.
\end{proposition} 
We review the definition of $b$-fibrations in the Appendix, and their central importance in the
theory of manifolds with corners.  In particular, this Proposition will be crucial in proving the fine
regularity results for the families of fiberwise constant curvature metrics later in this paper.
\begin{proof}
The proof is by induction on $k$. The result is obvious when $k=1$, since in that case $\calE_1 = M$ and
$\calC_1 = [M \times M;\Delta_{12}]$, and a map where the range is a manifold without boundary or corners
is trivially a $b$-fibration.

Note that when $k=2$, we have already written the explicit relationship between boundary defining functions, cf.\ \eqref{bdf}, 
which proves the result in that case as well.

Now suppose that the assertion is true for any $\ell < k$.   The first goal is to show that $\wh{\beta}$ is a $b$-map,
which means that the pullback of any boundary defining function $\rho_{\calI}$ for a face $F_\calI \subset \calE_k$ is 
a product of boundary defining functions in $\calC_k$ (up to a nonvanishing smooth factor). If $q \in \calC_k$
lies on the interior of some boundary face, then by (3) and (4) of the previous lemma, we can replace $\wh{\beta}$ by
the product of maps
\begin{equation}\label{e:prod}
(\wh{\beta}_{|\calJ'|}: \calC_{|\calJ'|} \rightarrow \calE_{|\calJ'|})\times \calE_{k-|\calJ'|}.
\end{equation}
When $|\calJ'|=k$, the tree $T$ associated to $(\frakp,z)$ has precisely one node, and only the face $F_{\max}$ contains 
$\wh{\beta}(q)$.  Furthermore, in this case, $q$ lies in the interior of the boundary face $\frakC_{\max}$,
so just as for the cases $k=2$ and $3$, we have 
$$
\wh{\beta}^*\rho_{\max}=AR_{\max} w,
$$
where $R_{\max}$ is the boundary defining function for $\frakC_{\max}$ and $w$ is a defining function for the the boundaries
of $M$ blown up at $k$ points. If $|\calJ'|<k$, then $\wh{\beta}_{|\calJ'|}$ in~\eqref{e:prod} is a $b$-fibration by the
inductive hypothesis, so the boundary defining functions $\rho_{\calI}$ with $\calI\subset \calJ'$ pull back to products 
of boundary defining functions in $\calC_{|\calJ'|}$ (up to a nonvanishing smooth function), while other boundary defining 
functions $\rho_{\calI}$ with $\calI \not\subset \calJ' $ pull back trivially.

On the other hand, suppose $q$ lies on the boundary of a boundary face, so there is an associated tree $T$ and node $\calJ$,
and let $\calI \supset \calJ$ be the node directly over $\calJ$. (If there is no node above $\calJ$, then in the following
we regard $R_{\calI}=w$ be a boundary defining function for the surface $M$ blown up at the appropriate number of
points.)  By (1) of the previous lemma, there is a local product decomposition such that the points in $\calI$ are 
separated from all others,  and locally $\wh{\beta}$ splits as 
$$
(\wh{\beta}_{|\calI|}: \calC_{|\calI|} \rightarrow \calE_{|\calI|})\times \calE_{k-|\calI|}.
$$ 
Furthermore, 
\begin{equation}
\wh{\beta}^{*}\rho_{\calJ}=AR_{\calJ}w_{\calI}, \ \wh{\beta}^{*}\rho_{\calJ'}=R_{\calJ'}, \ \ \calJ'\neq \calJ,
\end{equation}
where $R_{\calJ}$ is the boundary defining function of $\frakC_{\calJ}$ and $w_{\calI}$ is the boundary defining function of 
$\frakC_{\calI}$ in that corner.  All other boundary defining functions pull back trivially by virtue of this product decomposition.

The remaining issue is to show that $\wh{\beta}$ does not map any one of the boundary hypersurfaces $\frakC_{\calI}$ 
of $\calC_k$ to corners of $\calE_k$.  This is clear from the construction since $\frakC_{\calI}$ maps to the boundary
hypersurface $F_{\calI}$. 

These facts together prove that $\wh{\beta}$ is a $b$-fibration.
\end{proof}

\section{Geometry of merging cone points}\label{s:geometry}

\subsection{Review of existing results}\label{ss:literature}
We study constant curvature metrics with conical singularities, which is defined by the following: a smooth metric on $M\setminus \frakp$ with constant curvature, and near the punctures the metric is asymptotically conical, that is, there exist local coordinates such that the metric is given by
$$
e^{\phi} |z|^{2(\beta-1)}|dz|^{2}
$$ 
with $\phi$ being smooth. There is also a geodesic coordinate description given by
\[
g=\left\{
\begin{array}{ll}
d\frakr^{2}+\beta^{2}\frakr^{2}d\phi^{2}, & K=0\\
d\frakr^{2}+\beta^{2}\sinh^{2} \frakr d\phi^{2}, & K=-1\\
d\frakr^{2}+\beta^{2}\sin^{2} \frakr d\phi^{2}, & K=1
\end{array}
\right.
\]
 In particular, in each case with curvature of different signs, asymptotically it is always given by the flat metric.

The central problem to study is: given $(\frakc, \frakp,\vec \beta, K, A)$ as the conical data, does there exist a constant curvature conical metrics of this data? Is the metric unique? 
The study of this singular uniformization problem has a long history and has been very active recently. When the curvature $K$ is nonpositive, the conclusion is relatively straight-forward. By the results of McOwen~\cite{Mc} and Luo--Tian~\cite{LT}, for any fixed $(\frakc, \frakp, \vec \beta, K, A)$ that satisfy the Gauss--Bonnet formula and $K\leq 0$, there exists a unique constant curvature conical metric prescribed by the tuple.  When $K>0$, the situation is complicated depending on the cone angles. When all the cone angles are smaller than $2\pi$, by the results of Troyanov~\cite{Tr} and Luo--Tian loc. cit., there is a unique constant curvature metric when the cone angles are in the ``Troyanov region''~\eqref{e:Troyanov}. And when all the cone angles are less than $2\pi$, there is a moduli space structure, obtained by the first author and Weiss~\cite{MW}.

When some of the cone angles are bigger than $2\pi$,  unlike the cases above, there is no uniform result.  When $M$ is a sphere, there are some results depending on the number of cone points. When $k=2$,  Troyanov~\cite{Tr2} gave the results. When $k=3$, the characterization via complex analytical methods was given by Eremenko~\cite{Ere2} and Umehara--Yamada~\cite{UY}. When $k=4$ with symmetry, complex analysis techniques can also be applied, see Eremenko--Gabrielov--Tarasov~\cite{EGT}. When the genus of $M$ is greater  than 0, there are some general existing results by Carlotto--Malchiodi~\cite{CM1, CM2}, Bartolucci--De Marchis--Malchiodi~\cite{BMM}. In~\cite{SCLX} some new families of metrics are constructed by relating to Strebel differentials.
We also mention here some related problems, including Toda systems and mean field equations, see Chen--Lin~\cite{CL1, CL2} and references therein. In particular the recent result by Lee--Lin--Yang--Zhang~\cite{LLYZ} considering the singularity formation of two-points collision is morally related to the cone points merging behavior studied in this paper.

Recently, the breakthrough of Mondello--Panov~\cite{MP} shows the necessary condition of existence on a sphere by the following holonomy condition:
\begin{equation}\label{e:MP}
d_{1}(\vec \beta-\vec 1, \ZZ_{odd})\geq 1.
\end{equation}
They also showed that when the strict inequality holds (``non-coaxial'' situation) there exists at least one such metric. 
The recent results by Dey~\cite{Dey}, Kapovich~\cite{Ka}, and Eremenko~\cite{Ere1} determined the necessary condition of existence when the equality holds in~\eqref{e:MP}.

When trying to extend the result of~\cite{MW} to get a smooth manifold structure in this case, we found that there are obstructions, reflected in the fact that the linearized operator fails to be surjective on some subvarieties.
A key component in the construction of Mondello and Panov is the splitting of cone points, which turns out to be the key to resolve the analytic obstruction. We are going to describe the geometry of this process (splitting, or equivalently, 
merging of cone points) below.

\subsection{Local geometry of merging cone points}\label{ss:local}
In the next three sections, we consider the behavior of constant curvature metrics with some of the cone points merging together (or equivalently, when a cone splits into several cones). The cases we are going to study in this paper include all hyperbolic and flat metrics, and spherical metrics with angles less than $2\pi$. 

We first describe this process locally by looking at the following family of metrics parametrized by $t\in [0,\epsilon)$
\begin{equation}
g(t)=|z-p_{1}(t)|^{2(\beta_{1}-1)}|z-p_{2}(t)|^{2(\beta_{2}-1)}|dz|^{2}.
\end{equation}
where $p_{1}(t)$ and $p_{2}(t)$ are smooth coordinates that parametrize two moving points on $M$. We also require $p_{1}(0)=p_{2}(0)$. 
When $t\neq 0$, it gives a metric with two separate cone points with angles $2\pi\beta_{1}$ and $2\pi\beta_{2}$. And the distance between the two cone points decreases as $t$ goes to 0. Eventually in the limit $t=0$, the metric is given by 
$$
|z|^{2(\beta_{1}+\beta_{2}-2)}|dz|^{2}
$$
which is a conic metric with a single cone angle $2\pi(\beta_{1}+\beta_{2}-1)$ if $\beta_{1}+\beta_{2}-1>0$. 

This process of merging two cone points can be generalized to multiple points. After merging $j$ points each with angle $\{2\pi \beta_{i}\}_{i=1}^{j}$, the angle we get is given by the following defect formula:
\begin{equation}\label{e:beta0}
2\pi\beta_{0}:=2\pi \left (\sum_{i=1}^{j} \beta_{i}- (j-1)\right).
\end{equation}

One thing to notice from the defect formula is that not all conic points can be merged; it only happens when the ``admissible condition'' below is satisfied,
otherwise there is an obstruction to produce a new conic points.
When $\beta<0$, $|z|^{2(\beta-1)}|dz|^{2}$ is no longer conical, and we get some open ends, the form of which depend on the curvature constant. Therefore, we define the following condition for merging:
\begin{definition}\label{d:AdmAn}
We say a set of cone angles $\{\beta_{i}\}_{i\in \calI}$ is \emph{admissible} 
if
\begin{equation}\label{e:beta}
\sum_{i\in \calI}\beta_{i}>|\calI|-1.
\end{equation}
\end{definition}
In particular, this implies that when $k=2$, the two cone points need to satisfy 
$
\beta_{1}+\beta_{2}>1.
$

\subsection{Global geometry}\label{ss:glb}
The metrics we are going to consider in this paper are the following:
\begin{equation}\label{e:AdmAn}
\begin{array}{l}
\cdot \text{Flat or hyperbolic conical metrics with }\vec\beta \in \RR_{+}^{k}
 \text{ such that the }\\ \text {Gauss--Bonnet formula~\eqref{e:gb} is satisfied; or }\\
\cdot \text{Spherical conical metrics with } \vec\beta \in (0,1)^{k} \text{ satisfying~\eqref{e:Troyanov} if }k\geq 3 \\ \text{ or genus }(M)>0 \text{; or }\beta_{1}=\beta_{2} \text{ if }k=2 \text{ and } M=\bS^{2}.
\end{array}
\end{equation}

From~\cite{Tr, LT, Mc}, there is a unique constant curvature metric for each of the configuration $(\frakp, \vec \beta)$ with $\vec \beta$ satisfying one of the conditions above. Now considering the family of metrics with varying cone points, 
We would like to understand the uniform behavior of those metrics when some of cone points merge together.  
Using the defect formula we can see that
$$
\beta_{0}-1=\sum_{i=1}^{j} (\beta_{i}-1).
$$
Together with the Gauss--Bonnet formula, this implies that the curvature remains the same constant in this merging process.

Because of the restriction of cone angles defined above, not all cones can be merged to produce new cones. Moreover, the Troyanov constraint for spherical metrics is not preserved in this merging process. Therefore, for a given set of cone angles $\vec\beta$, the fiber conical metrics might potentially only be defined in a subset of $\calC_{k}$. And we define the following admissible region.
\begin{definition}
For fixed $\vec\beta=\{\beta_{i}\}_{i=1}^{k}$ satisfying one of the conditions in~\eqref{e:AdmAn}, the admissible extended configuration space $\calE_{k,\vec \beta}$ is 
defined to be the union of configurations in $\calE_{k}$ with which there exists a constant conical curvature metric on $M$, i.e. 
\begin{multline}
\calE_{k,\vec \beta} \ :=\calE_{k}^{0} \ \bigcup \bigcup_{\sum_{i\in \calI_{j}}\beta_{i}>|\calI_{j}|-1, \forall j}\big\{\frakp \in \cap_{j=1}^{l}F_{\calI_{j}}: \ \ 
\text{ there exists a constant } \\ \text{curvature metric with configuration }\ \big(\frakp, \{\sum_{i\in \calI_{j}} (\beta_{i}-1)+1\}_{j=1}^{l}\big) \big\}.  
\end{multline}
\end{definition}
The admissible extended configuration family $\calC_{k,\beta}$ is defined in a similar way by only considering the admissible combinations, or equivalently $\calC_{k}\supset \calC_{k,\vec \beta}=\pi_{k}^{-1} (\calE_{k,\vec \beta}).$

\medskip

\noindent \textbf{Flat and hyperbolic cases}

The metrics we consider will be any $\vec\beta \in \RR_{+}^{k}$ such that the Gauss-Bonnet formula~\eqref{e:gb} is satisfied. In particular, while the Gauss-Bonnet constraint
$$
\sum_{i=1}^{k}(\beta_{i}-1)\leq \chi(M)
$$ 
gives an upper bound of $\beta_{i}$, 
we do not require the cone angles to be uniformly small.

The admissible extended configuration space is relatively easy in this case. As long as the merging cone angles are admissible in the sense of definition~\ref{d:AdmAn}, there exists a flat (or hyperbolic, depending on the angle combination) conical metric after merging. In particular, in this case we have
$$
\calE_{k, \vec\beta}=\calE_{k}^{0} \bigcup \bigcup_{\sum_{i\in \calI}\beta_{i}>|\calI|-1}F_{\calI},
$$
and
$$
\calC_{k, \vec\beta}=\calC_{k}^{0} \bigcup \bigcup_{\sum_{i\in \calI}\beta_{i}>|\calI|-1}\frakC_{\calI},
$$

When $k=2$ and $M=\bS^{2}$, there is neither flat nor hyperbolic conical metrics on $M$. When $k\geq 3$ or the genus of $M$ is greater than 0, $\calC_{k, \vec\beta}\setminus \calC_{k}^{0}$ is nonempty in general. In particular, when the genus of $M$ and the cone angles are all sufficiently large, all directions of merging will be allowed and in that situation $\calC_{k, \vec\beta}=\calC_{k}$. 

\medskip

\noindent\textbf{The spherical case} 

In this case all the cone angles are assumed to be less than $2\pi$, hence $\vec\beta \in (0,1)^{k}$. Notice that by the relation~\eqref{e:beta0}, the cone angle obtained after merging, denoted by $2\pi\beta_{0}$, is also less than $2\pi$. Therefore during this merging process we stay inside the class defined in~\eqref{e:AdmAn}.

The extra rigidity in the spherical case of~\eqref{e:AdmAn} and the fact that merging does not preserve this constraints make  $\calE_{k, \vec\beta}$ and $\calC_{k, \vec\beta}$ much more complicated to describe, compared to the previous case. We illustrate a few cases here, keeping in mind that the football case ($M=\bS^{2}, k=2, \beta_{1}=\beta_{2}$) is special for the reason that will be made clear in Section~\S\ref{s:sph}. 

We start with $M=\bS^{2}$. From~\cite{Tr2}, there is no ``tear-drop'' metric on $\bS^{2}$ with only one conical point. Hence there is no admissible merging on $\calC_{2}$,  which implies $\calC_{2,\vec\beta}=\calC_{2}^{0}$.
On a sphere with 3 conical points and assuming $0<\beta_{1}\leq \beta_{2}\leq \beta_{3}<1$, the Troyanov condition is given by 
\begin{equation}\label{e:3pt}
2\beta_{1}+1>\sum_{i=1}^{3}\beta_{i}, \text{ or } \beta_{1}>\beta_{2}+\beta_{3}-1.
\end{equation}
We show that it cannot merge to a football. Since all cone angles are less than $2\pi$, the merging process decreases the angle strictly, i.e. $\beta_{i}+\beta_{j}-1<\min\{\beta_{i}, \beta_{j}\}$. Hence the only feasible way to achieve a football would be merging the two bigger angles, and this gives
$$
\beta_{2}+\beta_{3}-1=\beta_{1}
$$
which contradicts~\eqref{e:3pt}. Since the three points cannot be simultaneously merged into one point either, we have $\calC_{3,\vec\beta}=\calC_{3}^{0}$.

When $k\geq 4$, depending on different $\vec \beta$, the situation can be very different. Assuming $0<\beta_{1}\leq \beta_{2}\leq \beta_{3}\leq \beta_{4}<1$ satisfying the Troyanov condition
$$
2\beta_{1}+2>\sum_{i=1}^{4}\beta_{i},
$$ 
if $1<\beta_{1}+\beta_{4}=\beta_{2}+\beta_{3}$ (which can be achieved for example by taking $\vec \beta=(\frac{1}{2}, \frac{2}{3},\frac{2}{3},\frac{5}{6})$ ), then by merging the two groups $\{\beta_{1},\beta_{4}\}$ and $\{\beta_{2}, \beta_{3}\}$ simultaneously, we get a football; 
however since we cannot split a football to get an admissible 3-point configuration, one cannot merge $\{\beta_{1},\beta_{4}\}$ nor $\{\beta_{2},\beta_{3}\}$ without the other group. 
That is, $\frakC_{14}\cap \frakC_{23}\subset \calC_{4,\vec\beta}$ but $\frakC_{14}^{0}\cup \frakC_{23}^{0} \not\subset \calC_{4,\vec\beta}$. 
However, it is possible to merge $\{\beta_{2},\beta_{4}\}$ or $\{\beta_{3},\beta_{4}\}$ in most of those cases (as in the example given above), hence $\frakC_{24}^{0}\cup \frakC_{34}^{0} \subset \calC_{4,\vec \beta}$. 
In contrast, if $\beta_{1}+\beta_{4}\neq\beta_{2}+\beta_{3}$ then one cannot get a football but merging into a 3-point configuration is still possible. 

When the genus of $M$ is greater than 1, the description of $\calC_{k,\vec \beta}$ still depends on $\vec \beta$ and the Troyanov condition, however since we will not get any football configuration in this case, it is analytically same to the flat or hyperbolic cases.

\section{Preliminary analysis}\label{s:analysis}
Our approach to the study of the geometric problems described in the last section involves the analysis of conic elliptic
operators on spaces with isolated conic singularities, as employed already in \cite{MW}. The new feature here is that
we study families of such operators on spaces with coalescing conic singularities. 

\subsection{$b$-vector fields on $M$ and conic elliptic operators}\label{ss:bvector}
Let $M$ be a manifold with isolated conic singularities at the collection of points $\frakp = \{p_1, \ldots , p_k\}$, and denote
by $\wh{M} = [M; \frakp]$ the blowup of $M$ at these points. Thus $\wh{M}$ is a manifold with $k$ boundary components;
when $\dim M = 2$, each boundary component is diffeomorphic to a circle. Choose local polar coordinates $(r,\theta)$
near any $p_j$, so $r$ is a boundary defining function for the boundary face created by blowing up $p_j$ and 
$\theta$ is a set of local coordinates on that face, e.g.\ $\theta$ is the angular coordinate if $\del \wh{M}$ is a union of circles. 
We recall the space of $b$-vector fields, which consists of all smooth vector fields on $\wh{M}$ which are tangent 
to the boundary. In these local coordinates, 
\[
\calV_b = \mbox{$\calC^\infty$-span}\, \{ r\del_r, \del_\theta\}
\]

A differential operator is called a $b$-operator if it can be written locally as a finite sum of products of elements of $\calV_b$, 
\[
L = \sum_{j + |\alpha| \leq m}  a_{j \alpha}(r,\theta) (r\del_r)^j \del_\theta^\alpha
\]
(here we continue to think of $\theta$ as possibly multi-dimensional and $\alpha$ a multi-index).  This operator
is called $b$-elliptic if its $b$-symbol is nonvanishing, 
\[
{}^b\sigma_m(L) = \sum_{j+|\alpha| = m} a_{j \alpha}(r,\theta) \rho^j \eta^\alpha \neq 0\ \ \mbox{for}\ \ (\rho,\eta) \neq (0,0).
\]
Finally, we say that $L$ is an elliptic conic operator if $L = r^{-m}A$ where $A$ is an elliptic $b$-operator of order $m$.
We are primarily concerned with elliptic conic operators of order $2$, in particular the Laplacian on a surface
with isolated conic singularities. 

Now suppose that $X$ is a more general manifold with corners, i.e., any point $q \in X$ has a neighborhood
$\calU$ diffeomorphic to a neighborhood of the origin in an orthant $\RR^k_+ \times \RR^{n-k}$ (in which
case $q$ lies on a codimension $k$ corner). As before we define the space of $b$-vector fields on $X$ to
consist of the smooth vector fields which are tangent to all boundary faces. In local coordinates $(x_1, \ldots, x_k, y_1, \ldots, y_{n-k})$
near a codimension $k$ corner, with all $x_j \geq 0$, we have that
\[
\calV_b(X) = \mbox{$\calC^\infty$-span}\, \{ x_j \del_{x_j},  \del_{y_\ell}\}. 
\]

Our main examples of manifolds with corners here, of course, are the extended configuration spaces $\calE_k$ and
the extended configuration families $\calC_k$.  We do not consider here $b$-differential operators on general
manifolds with corners. Instead, as motivated by our problem, the fibers $\pi^{-1}(q) \subset \calC_k$ are 
unions of two dimensional surfaces with boundary, possibly `tied' along their boundaries, and we consider 
the families of elliptic conic operators on these fibers, parametrized by $q \in \calE_k$.    Nonetheless, it
still turns out to be important to consider $b$-vector fields on the entire space $\calC_k$.  In doing so, it is 
convenient to organize the boundary hypersurfaces into three types, corresponding to the boundary faces $\frakC_{\calI}$
which resolve point coincidences, the faces $\frakC_i$ corresponding to individual conic points, and the faces corresponding 
to the blowups of each fiber $M_{\frakp}$. We denote the boundary defining functions for these by $\rho_{\calI}$, $\rho_i$ and $\rho$, 
respectively. 

\subsection{$b$-H\"older spaces on $M$ and mapping properties of conic elliptic operators}\label{ss:holder}
Conic elliptic operators act naturally between weighted Sobolev and H\"older spaces defined relative to differentiations
by the vector fields in $\calV_b(M)$. Our ultimate problem is nonlinear, so we define only the $b$-H\"older spaces
and state the key mapping properties on these.  For simplicity, we restrict attention to the $2$-dimensional case.

\begin{definition}[Weighted $b$-H\"older spaces]
For any function defined on $M_{\frakp}$, define the seminorm $[u]_{b; \delta}$ in the usual way away
from a neighborhood of the boundary faces, while in each such neighborhood, in local polar coordinates, we set
$$
[u]_{b; \delta} := \sup_{0< \bar{r} <r_0} \sup_{ \bar{r} \leq r,r' \leq 2\bar{r}} \frac{|u(r,\theta)-u(r',\theta')|\bar{r}^{\,\delta}}{|(r,\theta)-(r',\theta')|^{\delta}}
$$ 
Then $\calC_{b}^{0,\delta}(M)$ consists of the functions $u$ which are bounded and for which $[u]_{b; \delta} < \infty$. 

Next, for any $\ell \in \mathbb N$, define $\calC^{\ell,\delta}_b(M)$ to consist of all functions $u$ such that 
$V_1 \ldots V_j u \in \calC^{0,\delta}_b(M)$ for every $j \leq \ell$ and $V_i \in \calV_b(M)$.   

Finally, $r^\mu \calC^{k,\delta}_b(M) = \{u = r^\mu v:  v \in \calC^{\ell,\delta}_b(M)\}$. 
\end{definition}

It is immediate from these definitions that if $L$ is a second order conic elliptic operator, then for every $\ell \geq 2$, 
\begin{equation}
L:  r^{\mu} \calC^{\ell,\delta}_b(M) \longrightarrow r^{\mu-2} \calC^{\ell-2,\delta}_b(M)
\label{conicmap}
\end{equation}
is bounded.  

It can happen that this map does not have closed range for certain values of $\mu$.  Indeed, $\mu$ is called
an indicial root of $L$ if there exists some function $\phi(\theta)$ such that $L( r^\mu \phi(\theta)) = \calO( r^{\mu-1})$. 
The expected order of decay or blow-up is $r^{\mu-2}$, so $\mu$ is an indicial root only if there is some leading order cancellation. 
It is not hard to check that if $\mu$ is an indicial root, then an appropriate sequence of cutoffs of $r^\mu \phi(\theta)$ can
be constructed to show that \eqref{conicmap} does not have closed range. This is explained at length in \cite{MW}. 
The following is the basic Fredholm result,  proved in \cite{M-edge} but see also \cite{MW}, 
\begin{proposition}
If $\mu$ is not an indicial root, then \eqref{conicmap} is Fredholm. 
\end{proposition}

If $L = \Delta + V$ where $V \in \calC^{0,\delta}_b$ for example (or even just $V \in r^{-2 + \epsilon} \calC^{0,\delta}_b$ for any 
$\epsilon > 0$), then at a conic point $p$ with cone angle $2\pi \beta$, the indicial roots consist of the set of values 
$j/\beta$, $j \in \mathbb Z$. The coefficient $\phi(\theta)$ corresponding to the indicial root $j/\beta$ can be any 
linear combination of $\sin j \theta$ and $\cos j \theta$. 

We also consider $L = \Delta + V$,  when $V \in \calC^{\ell,\delta}_b$, as an unbounded operator 
\begin{equation}
L: \calC^{\ell+2,\delta}_b(M) \longrightarrow \calC^{\ell,\delta}_b(M).
\label{samespace}
\end{equation}
We then seek to characterize its domain, i.e., the (nonclosed) subspace 
\[
\calD^{\ell,\delta}_b(L) = \{u \in \calC^{\ell+2,\delta}_b(M):  Lu \in \calC^{\ell,\delta}_b(M)\}.
\]
This is called the H\"older Friedrichs domain for $L$.  
\begin{proposition}[\cite{MW}] 
The space $\calD^{\ell,\delta}_b(L)$ consists of functions $u \in \calC^{\ell+2,\delta}_b(M)$ such that near 
each conic point $p$, 
\[
u = a_0 + \sum_{j=1}^{N(\beta)} (a_{j1} \cos j\theta + a_{j2} \sin j\theta) r^{j/\beta} + \wt{u},
\]
for some constants $a_{j1}$, $a_{j2}$, where $N(\beta)$ is the largest value $N$ such that $N/\beta < 2$, 
and $\wt{u} \in r^2 \calC^{\ell+2,\delta}_b(M)$.
\end{proposition}

\subsection{Families of conic elliptic operators}\label{ss:conicop}
The previous subsection reviews a few standard results about conic elliptic operators on surfaces.
Our interest is in families of such operators, particularly as the conic points coalesce.  In particular,
suppose $L_\frakp$ is such a family where the cone points are located at some simple divisor $\frakp \in \calD^s_k$. 
A key difficulty in extending the theory for individual operators to families is that the function space on 
which $L_\frakp$ acts vary with $\frakp$.  We use the geometric machinery developed above to handle this.

More specifically, we first consider weighted H\"older spaces on $\calC_k$ and the restrictions of these spaces to the fibers
$\pi^{-1}(q)$, $q \in \calE_k$, then define the appropriate families of weighted H\"older spaces on which we 
may describe extensions of the mapping properties. 

In the following, let $\vec{\nu}$ be a weight vector, with components indexed by the hypersurfaces of $\calC_k$.  We then
define in the obvious way the weighted H\"older spaces $r^{\vec{\nu}} \calC^{\ell,\delta}_b(\calC_k)$.  To make sense of the
restrictions of these spaces to each fiber, we need an easy result.
\begin{lemma}
The restriction of $r^{\vec{\nu}}\calC_{b}^{\ell,\delta}(\calC_{k})$ to each fiber $M_{\frakp}=\pi^{-1}(\frakp)$ is precisely the
weighted space $r^{\vec{\nu}} \calC_{b}^{l,\delta}(M_{\frakp})$, where (abusing notation slightly), the weight vector $\vec{\nu}$ 
here has components indexed by the boundary components of $M_\frakp$. 
\end{lemma}  
\begin{proof}
Observe first that the boundary faces of $M_{\frakp}$ are the components of the intersection 
$\left(\cup_{\calI}\frakC_{\calI}\cup \cup_{i}\frakC_{i}\right) \cap M_{\frakp}$. 

The fact that the restriction of $\calC_{b}^{0,\delta}(\calC_{k})$ to $M_{\frakp}$ equals $\calC_{b}^{0,\delta}(M_{\frakp})$ is straightforward from
the definitions since the boundary defining functions for the faces of $\calC_k$ restrict to the boundary defining functions
for the faces of each fiber, and there are coordinates tangent to the faces of $\calC_k$ which also restrict to the
the $\theta$ coordinates on each fiber. 

Next, the weight functions restrict naturally as well. Thus we must finally show that $\calC_{b}^{\ell,\delta}(\calC_{k})$ 
restricts to $\calC_{b}^{\ell,\delta}(M_{\frakp})$.  For this, note that if $V \in \calV_{b}(M_{\frakp})$, then there is an extension of 
$V$ to $\wh{V} \in \calV_b(\calC_k)$. Thus if $u \in \calC_{b}^{\ell,\delta}(\calC_{k})$ and $V_j \in \calV_b(M_\frakp)$, $i \leq \ell$,
and if $\wh{V}_i$ are the lifts, then $\wh{V}_1 \ldots \wh{V}_\ell u \in \calC^{0,\delta}_b(\calC_{k})$, and the restriction of this
expression is just $V_1 \ldots V_\ell u$, which by the first step lies in $\calC^{0,\delta}_b(M_\frakp)$. 
\end{proof}

Next, for any fiber $M_{\frakp}$ in $\calC_{k}$, if $\vec{\beta}$ is the set of cone angles, then we construct the Friedrichs-H\"older
domain by including at each $p_j$ the terms with local expressions $r^{j/\beta_i} \phi_j(\theta)$, $0 \leq j <2\beta_{i}$.

\begin{definition}
If $\vec \beta$ is fixed, then the fiberwise H\"older-Friedrichs domain associated to a family of conic metrics $g_{\frakp}$ 
is given by 
$$
\calD^{\ell,\delta}_{\mathrm{Fr}}(\calC_{k})=\{u\in \calC^{\ell,\delta}_{b}(\calC_{k}): \ \Delta_{g_{\frakp}}(u|_{M_{\frakp}}) \in \calC^{\ell,\delta}_{b}
(M_{\frakp}), \ \  \frakp \in \calE_{k}\}.
$$
\end{definition}

It is clear that $\calD^{\ell,\delta}_{\mathrm{Fr}}(\calC_{k})$ varies smoothly with $\frakp \in \calE_k$ over the regular fibers (where
all the $p_i$ are distinct). We may proceed just as in \cite{MW} to obtain that 
\begin{equation}\label{e:fr}
u = a_{0}+\sum_{j=1}^{[2\beta]}r^{\frac{k}{\beta}}\phi_j(\theta) + \tilde u, \ \tilde u \in r^{2}\calC^{\ell+2,\delta}_b,
\end{equation}
where as before $\phi_j = a_{j1} \cos j\theta + a_{j2} \sin j\theta$.  In this free region, smoothness follows from the smoothness 
of the boundary defining functions with respect to $\frakp$.

When $\frakp$ approaches a face $F_{\calI}$ of $\calE_k$, then the aggregate cone angle is $\beta=\sum_{i\in \calI} (\beta_{i}-1)+1$,
and functions in $\calD^{\ell,\delta}_{\mathrm{Fr}}(\calC_{k})$ have the form 
\begin{equation}\label{e:frII}
u=a_{0} f_{0}(w) +\sum_{j=1}^{[2\beta]}\rho^{\frac{j}{\beta}}f_{j}(w) + \tilde u, \ \tilde u \in \rho^{2}\calC_{b}^{\ell+2,\delta}(\calC_{k});
\end{equation}
here $\rho=\rho_{\calI}$ is the boundary defining function for the half sphere $\frakC_{\calI}$ and $f_{j}(w)$, $j=0, \ldots, [2\beta]$, 
are functions on $\frakC_{\calI}$ such that each $\rho^{\frac{j}{\beta}}f_{j}(w)$ is (formally) annihilated by the rescaled operator 
$\rho^{2\beta}\Delta_{g_{\frakp}}$ at $\frakC_{\calI}$. On this front face, the conic points are all separated. Therefore, functions in the 
Friedrichs-H\"older domain annihilated by $\rho^{2\beta}\Delta_{g_{\frakp}}$ are as in \eqref{e:fr}, where each cone points on this face
has the obvious cone angle extended from the interior of $\calC_k$. This means that functions on fibers $M_\frakp$ near this
face extend smoothly to this face. 

\section{Flat conical metrics}\label{s:flat}
We now begin our analysis of the space of constant curvature conic metrics by studying the simplest case, when the problem is linear.
Thus we fix closed surface $M$ and a set of cone angles $\vec \beta$ such that 
\begin{equation}
\chi(M, \vec{\beta}) := \chi(M) + \sum_{j=1}^k (\beta_j-1)=0.
\label{zeroEuler}
\end{equation} 
It is standard that if \eqref{zeroEuler} is satisfied, then for each marked conformal structure $(\frakc, \frakp)$ there exists a 
unique flat conic metric with area 1 and cone angle $2\pi \beta_i$ at $p_i$.  Our goal in this section is to show that this family
of flat conic metrics is polyhomogeneous on $\calC_k$. 

This is a local regularity theorem, so we fix a smooth family $g_0(\frakc)$ of smooth constant curvature metrics on $M$ representing
a neighborhood in the space of (unmarked) conformal structures.    For each $\frakp \in \calD_k^s$, consider the linear problem 
\[
\Delta_{g_0(\frakc)}  G =  2\pi \sum_{i=0}^k (\beta_i - 1)\delta_{p_i}.
\]
Then 
\[
- \int_M K = -2\pi \chi(M) = \int_M 2\pi (\sum (\beta_i - 1) \delta_{p_i},
\]
we see that the Liouville equation 
\[
\Delta_{g_0(\frakc)} u + K_{g_0(\frakc)} = 0
\]
has a solution $u = G$ which is unique if we require that $\int_M G = 0$.   This solution $G$ is essentially the Green function
for $\Delta_{g_0}$. It clearly depends smoothly on $\frakc$,  $\frakp \in \calD_k^s$ and $z \in M \setminus \frakp$, and near
each $p_i$ has the form
\[
G \sim (\beta_i - 1)\log |z| + \tilde{G}_i
\]
where each $\tilde{G}_i$ is $\calC^\infty$ in a neighborhood of $p_i$. We then define
\begin{equation}\label{e:flat}
g_{0}( \frakc, \frakp, \vec \beta)= e^{2G} g_0(\frakc). 
\end{equation}
Each of these metrics is flat and, because of the asymptotic structure of $G$, has a conic singularity with cone angle
$2\pi \beta_i$ at $p_i$.  This family of metrics is smooth when all the points $p_i$ are distinct.
 
\begin{theorem}\label{flat}
Fix $\vec \beta$ satisfying \eqref{zeroEuler}, then the family of flat metrics $g_0(\frakc, \frakp, \vec\beta)$ 
extends to a polyhomogeneous family of fiber metrics on $\calC_k$.
\end{theorem}

Because $g_0(\frakc)$ is smooth, it suffices to show that the scalar function $G$ extends to be polyhomogeneous on $\calC_k$. 
Note that by the remarks above, $G$ is $\calC^\infty$ on the interior of the extended configuration family, so our task is
to examine its behavior near each of the boundary faces and corners of $\calC_k$.  In other words, we must prove that there
exists an index family $\{ E_\calI\}$ such that 
$$
G \sim \sum_{(j, \ell) \in \mathcal{E}_{\calI}} (\rho_{\calI})^{j} (\log \rho_{\calI})^{\ell} a_{j\ell} (w_\calI),
$$
where $w_\calI$ are variables in the interior of each $\frakC_\calI$ and each $a_{j \ell}$ is polyhomogeneous with index 
family $\{E_{\calJ}\}_{\calJ \neq \calI}$.  Note that polyhomogeneity of $G$ near  the simplest faces $\frakC_i$ is obvious. 
We also suppress the smooth dependence of $G$ on $\frakc$. 

\subsection{The case of two cone points}\label{ss:flt2}
The proof of Proposition~\ref{flat} is by induction on $k$.  We begin with the proof when $k=2$. 
\begin{proposition}\label{flat2}
When $k=2$, $G(z,\frakp)$ is polyhomogeneous on $\calC_{2}$. 
\end{proposition}
\begin{proof}
Suppose that the two points $p_1$ and $p_2$ converge at the point $p_{12}$, which we may as well assume is fixed and
is the center of mass of these two points. Referring to the local coordinates in \S 2.3, we 
may as well restrict to a slice where $\zeta = \zeta_0 = 0$.
Then
\begin{equation}\label{e:conf2}
\begin{aligned}
G(z, \frakp) & =(\beta_{1}-1)\log|z-w| +(\beta_{2}-1)\log|z+w|  \\ & 
= (\beta_1 - 1) \log | re^{i\phi} - \rho e^{i\theta}| + (\beta_2-1) \log | re^{i\phi} + \rho e^{i\theta}|. 
\end{aligned}
\end{equation}
By \eqref{S2+}, $r=R_{12}\cos \omega,  \rho=R_{12}\sin \omega$, so
\begin{equation}\label{e:conf21}
\begin{aligned}
G(z, \frakp) = (\beta_{12}-1) \log R_{12} & + (\beta_1 - 1) \log |\cos \omega e^{i(\phi-\theta)} - \sin \omega| \\
& + (\beta_2-1) \log | \cos \omega e^{i(\phi-\theta)} + \sin \omega|.
\end{aligned}
\end{equation}
Here, and later in this paper, we set
\begin{equation}\label{defb0}
\beta_{12} = \beta_1 + \beta_2 - 1,
\end{equation}
i.e., $2\pi \beta_{12}$ is the cone angle which results when two cone points with cone angles $2\pi \beta_1$ and $2\pi \beta_2$
merge. The expression in \eqref{e:conf21} is certainly polyhomogeneous as $R_{12} \to 0$ away from $\omega = 0$ 
(the corner, where $\frakC_{12}$ meets $[M; \{p_{12}\}]$) and the points where $e^{i(\phi-\theta)} = \pm 1$.    To understand behaviour 
near the corner, write $s = \rho/r = \tan \omega$, so that when $\omega < \pi/4$, say, and recalling that $r = R_{12}\cos \omega$, 
we have 
\begin{equation}
\label{GnearC}
\begin{aligned}
G(z,\frakp) = (\beta_{12} -1) \log r + & \frac12 (\beta_1-1) \log (1 - 2s \cos (\theta-\phi) + s^2) \\
+ & \frac12 (\beta_2 - 1) \log (1+2s\cos(\theta-\phi) + s^2),
\end{aligned}
\end{equation}
The second and third terms on the right are smooth and vanish $s=0$. Note that there is an apparent asymmetry in
the indices $1$ and $2$ here; however, when the points $p_1$ and $p_2$ are switched, the angle $\theta$ changes
to $\theta + \pi$, so this expression is actually symmetric after all. Finally, near $\omega = \pi/4$ 
and $\theta = \phi$, for example, write $\tan \omega = 1 + \sigma$, so that $\cos \omega = 1/\sqrt{2 + 2\sigma+\sigma^2}$. Then
\begin{equation*}
\begin{aligned}
G(z,\frakp) = (\beta_{12} -1) \log (R_{12}/\sqrt{2+2\sigma+\sigma^2})  & + (\beta_1 - 1) \log | e^{i(\theta-\phi)} - 1 - \sigma| \\
& + (\beta_2-1) \log |e^{i(\theta - \phi)} + 1 + \sigma|,
\end{aligned}
\end{equation*}
and this is obviously polyhomogeneous around the face $\frakC_1$ created by blowing up $\sigma = 0$. The argument
is the same near $\frakC_2$. 

The assertion about polyhomogeneity of $G$ on $\calC_2$ is now proved.  
\end{proof}

\subsection{Inductive proof for higher $k$}\label{ss:fltk}
\begin{proof}[Proof of Proposition~\ref{flat}] Suppose that the result has been proven for $\calC_{j}$ with any $j < k$. Without loss of generality, we can restrict to the slice with the fixed center of mass $\zeta=\zeta_{0}=0$.

We first consider the case that is away from the central diagonal $\frakC_{1\dots k}$, that is, at most $k-1$ points can merge together. This is the case for example when the configuration $\vec \beta\in \RR^{k}$ is such that $\sum_{i=1}^{k}(\beta_{i}-1)\leq 1$. Then we can cover $\calE_{k,\vec \beta}$ by open sets $\{U_{\calI_{*}, \epsilon}\}$ defined in~\eqref{e:cluster}.   That is,  the only possible merging happens within the sub-clusters, and the distance between any clusters is bounded away from 0. From Lemma~\ref{l:tree}, $U_{\calI_{*}, \epsilon}$ locally has a product structure, identified with an open subset $\Pi_{j=1}^{\ell}U_{\calI_{j}}\subset \Pi_{j=1}^{\ell}\calE_{|\calI_{j}|}$. The total space fibers over $U_{\calI_{*}, \epsilon}$, and is given locally by a product of fibrations.
In this case, the conformal factor can be written as a sum
$$
G= \sum_{j=1}^{\ell}(\sum_{i\in \calI_{j}} (\beta_{i}-1)\log|z-z_{i}|).
$$ 
Since $\{z_{i}\}_{i\in \calI_{j}}$ is bounded away from any other clusters $\frakp_{\calI_{j'}}, j'\neq j$, the term $\sum_{i\in \calI_{j}} (\beta_{i}-1)\log|z-z_{i}|$ is only singular near $V_{i}$ as defined in Lemma~\ref{l:tree}.
By induction, this term is polyhomogeneous on $\calC_{|\calI_{j}|}$, hence is polyhomogeneous on  $U_{\calI^{*}, \epsilon}$. Same argument can be applied to other terms. By considering all the open covers, we get polyhomogeneity of $G$ on $\calC_{k}$ in this case.

Now we consider the behavior near the central face, and all the $k$ points can merge together. 
We now prove that away from all sub-diagonals, $G$ is polyhomogeneous near $\frakC_{1,\dots, k}^{0}$.   In this region we write
$p_i = z_i + \zeta$ and assume that the center of mass $\zeta = 0$. Then writing $(z, z_1, \ldots, z_k) = R_{1 \ldots k} \Omega$,
$\Omega = (\Omega_0, \ldots , \Omega_k)$, we have 
\[
G = \sum_{i=1}^k  (\beta_i - 1) \log |z - z_i| = (\sum_{i=1}^{k} \beta_i - k)\log R_{1\ldots k} + \sum_{i=1}^{k} (\beta_i - 1) \log |\Omega_0 - \Omega_i|
\]
and since $z$ remains bounded away from the sub-diagonals, only the term $\log R_{1 \ldots k}$ is singular here and this is obviously
polyhomogeneous on the interior of $\frakC_{1 \ldots k}$. And each term $(\beta_i - 1) \log |\Omega_0 - \Omega_i|$ is singular only near $\frakC_{i}$ and is polyhomogeneous.

Near the outer boundary of $\frakC_{1 \ldots k}$, set $z = r e^{i\phi}$ and $w_i = z_i/r$. Then
\[
G = (\sum_{i=1}^k \beta_i - k) \log r + \sum_{i=1}^{k} (\beta_i - 1) \log |e^{i\phi} - w_i|.
\]
Notice that all the faces $\frakC_i$ occur along the submanifolds $\{z = z_i\} \subset \{ |w_i| = 1\}$, so provided
we stay away from these submanifolds, only the first term $(\sum_{i=1}^k \beta_i - k) \log r$ is singular, and it is polyhomogeneous.  
At the principal diagonal $\frakC_i$, however, $w_i = e^{i\theta_i}$ and the additional singular term is $\log |e^{i \phi} - e^{i\theta_i}|$, 
which is polyhomogeneous there. 


Finally, if $\frakp$ is near any one of the partial diagonals, including near their intersection with $\frakC_{1\ldots k}$, 
then it is in a neighborhood of some intersection of front faces $\{\frakC_{\calI_{j}}\}_{j=1}^{\ell}$, where each $\calI_{j}$ is
a proper subset of $\{1,\dots, k\}$, and the $\calI_{j}$ have no elements in common. The resolution ensures
that the faces $\frakC_{I_i}$ and $\frakC_{I_j}$ are disjoint, so we can once again factor out the defining function
$R_{1 \ldots k}$ and separate out the indices $i$ which do not lie in any of the $\calI_j$, and write
\begin{equation}\label{e:flatconformal}
G = (\sum_{i=1}^{k} \beta_{i}-k ) \log R_{1,\dots, k} +  \sum_{i \notin \cup_{j}\calI_{j} } (\beta_{i}-1) \log |w -w_{i}| + \sum_{j=1}^{\ell} f_{j},
\end{equation}
where $w = z/R_{1 \ldots k}$, $w_i = z_i / R_{1 \ldots k}$.   
Here $f_{j}$ is the rescaled factor
$$
f_{j}=\sum_{i\in \calI_{j}} (\beta_{i}-1)\log |w-w_{i}|= \sum_{i\in \calI_{j}} (\beta_{i}-1)\log R_{\calI_{j}} + \sum_{i\in \calI_{j}} (\beta_{i}-1)\log |\Omega^{\calI_{j}}_{0}-\Omega^{\calI_{j}}_{i}|
$$
where $R_{\calI_{j}}$ is the boundary defining function for $\frakC_{\calI_{j}}$ and the coordinates over this face is given by $(w, w_{i})_{i\in \calI_{j}}=R_{\calI_{j}}(\Omega^{\calI_{j}}_{0}, \Omega^{\calI_{j}}_{i})$.
By induction, each rescaled factor $f_j$ is up to a smooth summand
the Green function near $\calI_j$ and hence is polyhomogeneous near the collection of faces which constitute
the resolution near this cluster. 
This behavior is uniform as $R_{1 \ldots k} \to 0$. 

It is perhaps wise to illustrate this induction for $\calC_{3}$. In this case, near $\frakC_{12}\cap \frakC_{123}$ we can write
$$
G=(\beta_{1}-1)\log |z-\epsilon_{1}(1+\epsilon_{2})| + (\beta_{2}-1) \log |z-\epsilon_{1}(1-\epsilon_{2})| +(\beta_{3}-1)\log|z+\epsilon_{1}|
$$
We wish to examine the region $R_{123} \to 0$ and $R_{12} \to 0$, and in this region, $\epsilon_{1}\sim R_{123}$, $\epsilon_{2}\sim R_{12}$.
Therefore, using $w_{12}=\frac{w_{3}-1}{\epsilon_{2}}$ as a coordinate on $\frakC_{12}$, 
\begin{multline}
G=(\sum_{i=1}^{3}\beta_{i}-3)\log \epsilon_{1} + (\sum_{i=1}^{2} \beta_{i}-2) \log \epsilon_{2}\\
+(\beta_{1}-1) \log |w_{12}-1|+(\beta_{2}-1) \log |w_{12}+1|\\
+ (\beta_{3}-1)\log |w_{3}+1|.
\end{multline}
Since $w_{3} \sim 1$ here, this is polyhomogeneous.
\end{proof}

\section{Hyperbolic conic metrics}\label{s:hyp}
We next turn to the analytic description of the space of hyperbolic cone metrics which, as explained earlier, exist whenever
\[
\chi(M) + \sum_{i=1}^k (\beta_i - 1) < 0. 
\]
The problem is now genuinely nonlinear and the proof of polyhomogeneity correspondingly more difficult.  Indeed, the
proof is directly inductive on the number of cone points.  We now explain the strategy, which requires several steps.   

For the case $k=2$, we construct a family of background metrics which is hyperbolic away from the merging points and flat near 
these points, with a transitional region in between.  Let $\rho$ be the degeneration parameter which measures the distance to
the fiber where the points coincide. We then solve for the expansion of the conformal factor iteratively on $M_{\frakp'}$ then on $\frakC_{12}$. This way we construct approximate solutions to arbitrarily high order of $\rho$. We then solve away the error in the exact curvature equation on each fiber using maximum principle. Finally we use commutator argument to show conormality and polyhomogeneity. Once the
theorem has been established for $k=2$, we follow an inductive procedure to construct families of background metrics in
the general case with the same properties, and once again solve away the error terms and show polyhomogeneity.

The case $k=2$ already contains essentially all of the substantial difficulties, so this case is presented in careful
detail. 

\subsection{The case of two merging cone points} \label{ss:hyp2}
Consider a family of simple divisors $\frakp$ which converge to a point $\frakq \in F_{12} \subset \calE_k$.  We may as well
assume that $p_3, \ldots, p_k$ remain fixed, but $p_1$ and $p_2$ merge at a point $p_{12}$ which, for simplicity, we assume 
is the center of mass of $p_1$ and $p_2$ and also remains fixed.  We write $\frakp'$ for the (simple) $(k-1)$-tuple 
$(p_{12}, p_3, \ldots, p_k)$. 
We are working locally near $\frakq$, and this point is far from any of the other partial diagonals, so we use the local 
coordinates on $\calE_2$ and $\calC_2$: $\rho = \rho_{12}$ and $R = R_{12}$, and for simplicity we set $\theta_{12} = 0$, 
which amounts to fixing the direction through which $p_1$ and $p_2$ approach one another. If $\beta_i$ are the cone 
parameters at $p_i$, then as noted earlier, $\beta_1$ and $\beta_2$ determine the limiting cone parameter $\beta_{12} 
= \beta_1 + \beta_2 - 1$ at $p_{12}$. In order for the two points to merge, it is necessary that
$$
\beta_{1}+\beta_{2}>1  \Leftrightarrow \beta_{12} > 0.
$$

The fiber $\pi^{-1}(\frakq) \subset \calC_k$ consists of two surfaces with boundary,  $M_{\frakp'} = [M; \{\frakp'\}]$ 
(the surface $M$ blown up at the points in $\frakp'$) and the face $\frakC_{12}$, and these meet along a common circle.  

\medskip

\noindent{\bf The initial metric}

We now construct a family of metrics on $M$ with $k$ conic singularities at the family of divisors $\frakp$ above, 
which extends as a smooth family of fiberwise metrics on $\calC_k$.  This family is obtained locally near the fiber
$\pi^{-1}(\frakq)$ by gluing the fixed hyperbolic metric $h_{0,\frakp'}$ with conic singularities at $\frakp'$, with
cone parameters $\beta_{12}, \beta_3, \ldots, \beta_k$, to the degenerating family of flat metrics $g_{0,\frakp}^{\mathrm{fl}}$ in \eqref{e:flat}. 
To do this, define 
\begin{equation}\label{e:model}
g_{0,\frakp}=\chi g^{\mathrm{fl}}_{0,\frakp} + (1-\chi) h_{0,\frakp'},
\end{equation}
where 
\[
\chi(z,\frakp)= 
\begin{cases} 
1\ \mbox{if} \ \rho <\bar{\rho}\ \mbox{and} \ |z|<2\bar{\rho} \\ 0\  \mbox{if}\ \rho > 2\bar{\rho}\ \mbox{or} \ |z|>4\bar{\rho}
\end{cases}
\]
is a smooth nonnegative cutoff function for some small $\bar{\rho} > 0$. We usually drop $\frakp$ and $\frakp'$ from the subscripts 
for simplicity, and also write 
\[
K_{0,\rho} = \begin{cases}
0\  &\mbox{if}\ \rho < \bar{\rho} \text{ and } |z|<2\bar{\rho} \\
-1 &\mbox{if}\  \rho >2\bar{\rho} \text{ or } |z|>4\bar{\rho}. 
\end{cases}
\]
for the curvature of the metrics in this family. 

Our goal is to obtain precise analytic control of the solution to the conformal curvature equation
\begin{equation}\label{e:hyp}
\Delta_{g_{0,\frakp}} u+e^{2u}+K_{0,\rho}=0
\end{equation}
as $\rho \to 0$.  In the neighborhood where $\chi = 1$, \eqref{e:hyp} becomes
\begin{equation}
\Delta_{g_{0,\rho}}u+e^{2u}=0.
\label{e:hyp0}
\end{equation}
Here $\Delta_{g_{0, \frakp}}$ is the Laplace--Beltrami operator with nonnegative spectrum.

\begin{figure}[h]
\centering
 \includegraphics[width=0.5\textwidth]{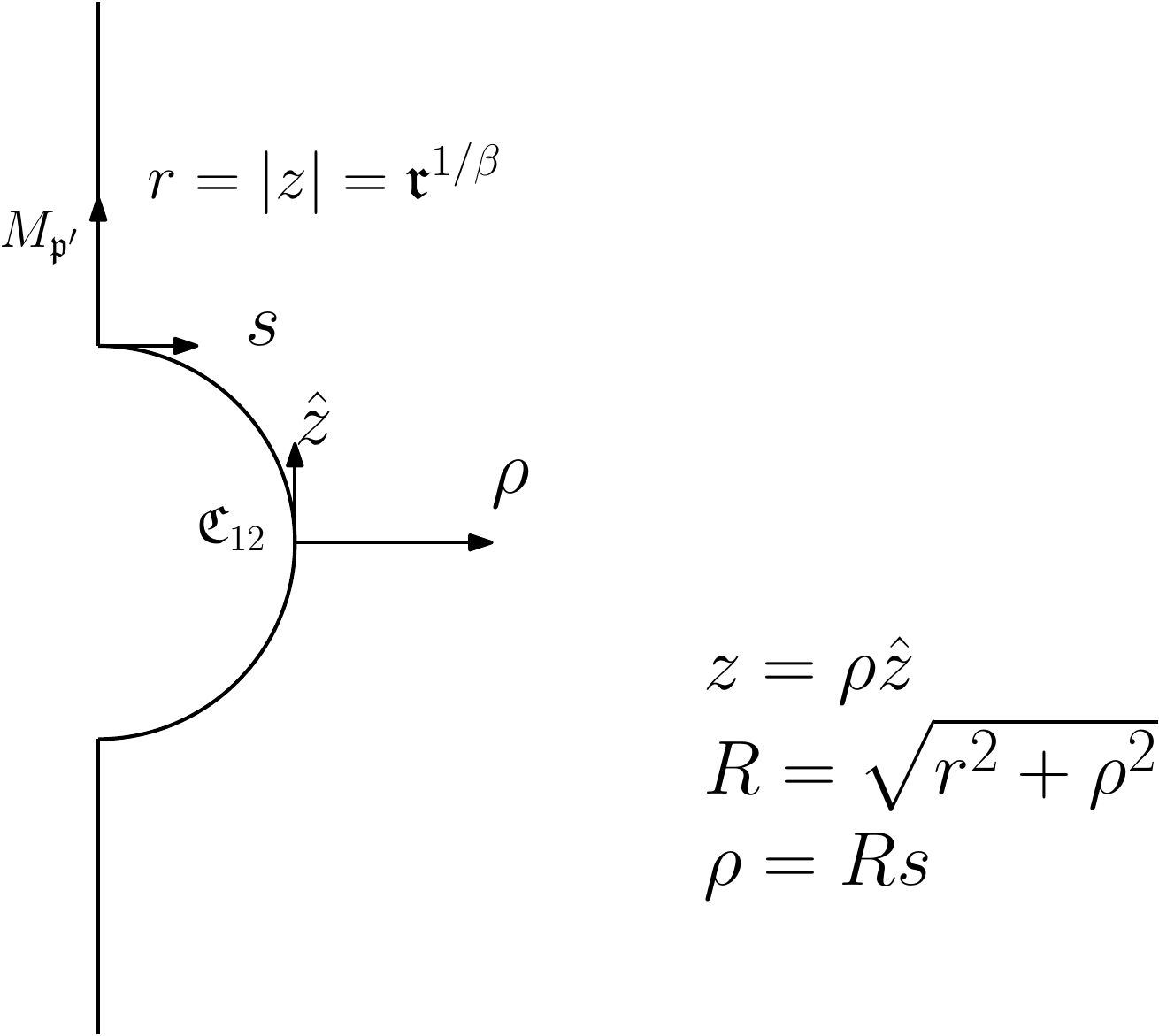}
 \caption{coordinates used in the computation of expansion near the singular fiber $\frakC_{12} \bigcup M_{\frakp'}$}
 \label{f:coordinates}
\end{figure}

Our goal is to prove that the solution $u$ to \eqref{e:hyp} is polyhomogeneous on $\calE_k$ near the point $\frakp' \in \calC_k$,
$z = 0$.  To do this we construct the solution anew (even though its existence is guaranteed by standard barrier arguments)
by first constructing an approximate solution which satisfies \eqref{e:hyp} to any fixed arbitrarily high order as $\rho \to 0$, and 
then correcting this to the exact solution using an analytic construction which guarantees that this additional correction term
also vanishes to that same high order.  

The specifics of the first part of this are that we construct the entire Taylor series for $u$ along each of the two faces $M_{\frakp'}$ 
and $\frakC_{12}$.  These series expansions are related to one another and must satisfy a set of matching conditions along
the corner where these faces intersect.  The fact that we can correct any finite part of this Taylor series to an exact solution 
with a term which vanishes to that order means that these series represent the true expansions for this exact solution. 

\medskip

\noindent{\bf Expansion at $M_{\frakp'}\cap \frakC_{12}$ within $M_{\frakp'}$} 

Recall the coordinates $z = r e^{i\phi}$ and $\rho$ near $\frakC_{12}$ in $\calE_k$ (as before, an angular coordinate
is suppressed), and set
$$
r=|z|, \ \mbox{and}\ s= \rho/r.
$$
Thus $s=0$ defines the surface $M_{\frakp'}$ while $r=0$ defines $\frakC_{12}$.  There is a freedom in the conformal coordinate 
$z$ on $M_{p'}$ by  holomorphic reparametrization, and we fix this below in Lemma~\ref{l:faceI}. We also define the coordinate 
$$
\frakr=\frac{1}{\beta_{12}}|z|^{\beta_{12}}
$$
on $M_{\frakp'}$, which is the radial distance function for the background flat conic metric. 
It follows from \eqref{GnearC} that near the corner $\frakr=s=0$, 
$$
g_{0,\frakp} =\alpha(s, \frakr, \phi, \theta)\, (d\frakr^{2}+\beta_{12}^{2}\frakr^{2}d\phi^{2}),  
$$
where $\alpha(s, \frakr)$ is polyhomogeneous with $\alpha(0, \frakr) = 1$ near $\frakr = 0$.  Our first result is
that by choosing the coordinate $z$ carefully, the expansion for $\alpha$ has a particularly simple form; it simultaneously
also gives the first term in the expansion of $u$ at $M_{\frakp'}$. 
\begin{lemma}
\label{l:faceI}
There is a unique bounded solution $u_0' \in \bigcap_{m \geq 0} \calC^{m,\delta}_b(M_{\frakp'})$ to the restriction of \eqref{e:hyp} 
to $M_{\frakp'}$.  This solution is polyhomogeneous as $\frakr \to 0$, and if this defining function is chosen appropriately, then 
$$
u_{0}'\sim \sum_{j \in \NN_{0}}a_{j}\frakr^{2j}.
$$
\end{lemma}
\begin{proof}
All of this except the last assertion, i.e., the existence and polyhomogeneous regularity, is contained in \cite{Mc} and \cite{MW}.
For simplicity set $\beta=\beta_{12}$. Existence and uniqueness of the solution is proved in \cite{Mc} by constructing bounded 
sub- and supersolutions, and this method also leads to  the uniqueness of $u_0'$ amongst bounded solutions. The regularity theorem
is proved in \cite{MW}. First, (scale-invariant) local elliptic regularity shows that $u_0' \in \calC^{m,\delta}_b$ for every $m \geq 0$.  
The refined regularity theorem in \S 3 of that paper states that $u_0'$ is polyhomogeneous, with
$u_0' \sim \sum_{\ell, j \geq 0} a_{\ell j}(\phi) \frakr^{\ell/\beta + j}$ as $\frakr \rightarrow 0$, where $a_{\ell  j}$ is linear in 
$\cos \ell \phi$, $\sin \ell \phi$. 

It remains to prove the last assertion, that all $a_{\ell j} = 0$ when $\ell \neq 0$ in some choice of coordinates.
To this end, recall that near a cone point, there are geodesic coordinates $(\tilde r, \theta)$ in terms of which the hyperbolic metric
takes the canonical polar form
\begin{equation}\label{e:sinh}
g=d\tilde r^{2} + \beta^{2} \sinh^{2} \tilde r d\phi^{2}.
\end{equation}
On the other hand, there is a local holomorphic coordinate $z$ centered at the conical point for which, in the region near $p_{12}$ in
which it is flat, $g_{0,\frakp} = |z|^{2(\beta-1)}|dz|^2$, so that 
$$
g=e^{2u_{0}'} |z|^{2(\beta-1)} |dz|^{2}
$$
there. 
Since $\frac{1}{\beta}|z|^{\beta}=\frakr$, we have
$$
d\tilde r^{2} + \beta^{2} \sinh^{2} \tilde r d\phi^{2}=e^{2u_{0}'(z)} |z|^{2(\beta-1)} |dz|^{2}=e^{2u_{0}'(z)} (d\frakr^{2}+\beta^{2} \frakr^{2}d\phi^{2}).
$$
This gives
$$
\frac{d\tilde r}{d\frakr }=e^{u_{0}'}, \ \sinh\tilde r=e^{u_{0}'} \frakr,
$$
or equivalently 
$$
\frac{d\tilde r}{\sinh \tilde r} = \frac{d\frakr}{\frakr}, 
$$
and hence 
$$
\tanh \frac{\tilde r}{2} = c\, \frakr. 
$$
We finally scale $z$ so that $c = 1/2$. 

We have now shown that
$$
\tilde r \sim \frakr\left(1 + \sum_{j=1}^{\infty}\tilde a_{j}\frakr^{2j}\right), 
$$
i.e., $\tilde r$ is an odd function of $\frakr$, so that $e^{u_{0}'}= \frakr^{-1} \sinh \tilde r$ is even in $\frakr$
and equals $1$ when $\frakr=0$. We conclude that $u_{0}'$ is even and vanishes at $\frakr=0$, hence
\begin{equation}
u_{0}' \sim \sum_{j=1}^{\infty} \tilde a_{0j} \frakr^{2j} = \sum_{j\geq 1} a_{0j} r^{2j\beta}
\label{e:u0}
\end{equation}
where each $a_{0j}$ is constant. 
\end{proof}

\medskip

\noindent{\bf Expansion at $M_{\frakp'}$}
We next turn to the complete expansion of $u$ at the face $M_{p'}$.   The defining function for this face is $s$, and
we know a priori that $u$ is smooth in $s$ at $s=0$, hence 
\begin{equation}
u \sim \sum_{j=0}^\infty \tilde{u}'_j(r,\phi) s^j.
\label{mps}
\end{equation}
Our goal is to compute these coefficients $\tilde{u}'_j$, and more specifically, to understand their expansions
as $r \to 0$.  In doing this, it is more convenient to write \eqref{mps} as an expansion in $\rho$ since
the Laplacian on the fibers commutes with $\rho$. Recall that near the corner $M_{\frakp'} \cap \frakC_{12}$, 
we have $\rho = R \sin \omega$, so if we set $s = \sin \omega$, then $s = \rho/R$. Furthermore, 
along $s=0$, we can take $R = r$. Therefore, 
\[
u \sim \sum_{j=0}^\infty \rho^j u_j', \ \ \mbox{where}\ u_j' = r^{-j} \tilde{u}'_j.
\]

For simplicity of notation, we assume here $2k\beta\notin \NN$ for any $k\in \NN$. The other case has no essential 
difference except the notation. In particular, the expansions below are the same, and for some values $\ell\in \NN$, 
the coefficients of $r^{\ell}$ appear in more than one place because of the coincidence $\ell'+2k\beta=\ell'$. 
\begin{proposition}\label{p:Mp}
As $s \to 0$, there is an expansion 
\begin{equation}\label{e:uj}
u\sim \tilde{u}_{0}'+\sum_{j=1}^{\infty} s^j \tilde{u}_j', 
\end{equation}
where, in terms of the functions $u_j'$ for $j\geq 1$, 
\begin{equation}\label{e:uj1}
u_{j}'\sim \sum_{\ell\in \NN} r^{\ell} a_{j\ell 0}(\phi) + \sum_{\ell, k \in \NN, \ell\geq 0, k\geq 1} r^{\ell+2k \beta} a_{j\ell k}(\phi).
\end{equation}
or equivalently,
\begin{equation}\label{e:uj11}
\tilde{u}_{j}'\sim \sum_{\ell\in \NN} r^{j+\ell} a_{j\ell 0}(\phi) + \sum_{\ell, k \in \NN, \ell\geq 0, k\geq 1} r^{j+\ell+2k \beta} a_{j\ell k}(\phi)
\end{equation}
Here $a_{j\ell 0}$ are trigonometric polynomials of pure degree $\ell$, i.e., linear combinations of 
$\cos(\ell\phi)$ and $\sin(\ell\phi)$, while $a_{j\ell k}$ $(\ell\geq 0, k\geq 1)$ are trigonometric polynomials of degree at most $\ell$.
In particular, 
\begin{equation}\label{e:u1}
\tilde{u}_{1}'\sim \sum_{\ell\in \NN} r^{1+\ell} a_{1\ell 0}(\phi) + \sum_{\ell, k \in \NN, \ell\geq 0, k\geq 1} r^{1+\ell+2k \beta} a_{1\ell k}(\phi).
\end{equation}
And $a_{1\ell k}$ is a linear combinations of 
$\cos(\ell\phi)$ and $\sin(\ell\phi)$ for any $\ell\geq 0, k\geq 0$.
\end{proposition}
\begin{proof}
Clearly $u$ is $\calC^\infty$ up to $M_{\frakp'}$ away from $z=0$, and thus has an expansion in nonnegative integer powers of $\rho$.
Expand $K_{g_{0,\frakp}}=\sum_{j=0}^{\infty} \rho^{j}K_{j}$ in \eqref{e:hyp}
and insert a formal series expansion for $u$, as in the statement of this theorem.  Notice here $K_{j}\equiv 0$ near $\frakC_{12}$. Since
$\Delta_{g_{0,\frakp}}$ commutes with $\rho$ away from $\frakC_{12}$, we obtain a recursive set of equations 
which successively determine all of the $u_j'$. The first of these is the curvature equation
$$
\Delta_{g_{0,\frakp}}u_{0}'+e^{2u_{0}'}+K_{0}=0,
$$
on $M_{\frakp'}$. By the previous lemma, $u_0'$ has an expansion involving only the powers $r^{2k\beta}, k\in \NN$. 

The equation for $u_j'$, $j \geq 1$,  is 
\begin{equation}\label{e:ind}
e^{-2u_{0}'}\Delta_{g_{0,\frakp}} u_{j}'+ 2 u_{j}' =  - e^{-2u_{0'}}\left(K_{j} + \left(e^{2u^{(j-1)}}-1-2 u^{(j-1)}\right)_{\!  j}\right), 
\end{equation}
where $u^{(j-1)} = \sum_{i=0}^{j-1}\rho^{i}u_{i}'$ and the notation $(w)_{j}$ means that we take the coefficient of $\rho^{j}$ in the 
expansion of $w$.  This is the Laplacian with positive spectrum, so this equation always admits a solution on $M_{\frakp'}$.
Using the special coordinates $(r,\phi)$ from Lemma~\ref{l:faceI} near $r=0$, this equation reduces in a neighborhood of $r=0$ to 
\begin{equation}
(\Delta_{2\beta} + 2e^{2u_0'} )u_j' = -\left(e^{2u^{(j-1)}}-1-2 u^{(j-1)}\right)_{\!  j}
\label{e:indnext}
\end{equation}
where $\Delta_{2\beta}$ is the Laplacian for the conic metric $|z|^{2(\beta-1)}|dz|^2$ where $z=re^{i\phi}$ (which gives the local form of $g_{0,\frakp}$). After multiplying by $r^{2\beta}$, the equation above can be rewritten as
\begin{equation}\label{e:shift}
\left( (r\partial_{r})^{2}+\partial_{\phi}^{2} \right)u_{j}'+2r^{2\beta}e^{2u_{0}'}u_{j}'=-r^{2\beta} \left(e^{2u^{(j-1)}}-1-2 u^{(j-1)}\right)_{\!  j}
\end{equation}
The bounded formal solutions (i.e., solutions 
to leading order) of 
$$\left( (r\partial_{r})^{2}+\partial_{\phi}^{2} \right)u_{j}'+2r^{2\beta}e^{2u_{0}'}u_{j}'=0$$
are $r^{\ell} q_\ell(\phi)$, $\ell \in \NN$, where 
$q$ is a trigonometric polynomial of pure degree $\ell$. The nonnegative indicial roots are $\ell = 0, 1, 2, \ldots $. 

For $j=1$, \eqref{e:indnext} becomes
\begin{equation}
\left( (r\partial_{r})^{2}+\partial_{\phi}^{2} \right)u_{1}'+2r^{2\beta}e^{2u_{0}'}u_{1}'=-r^{2\beta}K_1,
\end{equation} 
note here  $(e^{2u^{(0)}}-1-2 u^{(0)})_{\!  1}= (e^{2u_0'} - 1 - 2u_0')_{\!  1}=0$.
Noting that $K_1$ vanishes near $\frakC_{12}$ and $e^{2u_{0}'}\sim 1+\sum_{k\geq 1} a_{k}r^{2k\beta}$, the solution of this equation has an expansion of the form
$$
u_{1}'\sim \sum_{\ell\in \NN} r^{\ell} a_{1\ell 0}(\phi) + \sum_{\ell, k \in \NN, \ell\geq 0, k\geq 1} r^{\ell+2k \beta} a_{1\ell k}(\phi).
$$
Here the terms $a_{1\ell 0}$ come from indicial roots, and are formally undetermined near $r=0$, but of course are fixed because $u_1'$ solves
a global equation on $M_{\frakp'}$. And each $a_{1\ell 0}$ is a linear combination of $\cos(\ell\phi)$ and $\sin(\ell\phi)$.  All the other terms arise by matching coefficients on the two sides
of this equation. In particular, because of the multiplication by $r^{2\beta}$,  the leading term in the second sum is given by $r^{2\beta}$ and there are no log terms. And each $a_{1\ell k}(\phi)$ is of pure degree $\ell$.

We now prove by induction that the expansion of $u_j'$, $j > 1$, is as in~\eqref{e:uj1}.   The guiding principle in all of
this is that the right hand side of~\eqref{e:shift} does not contain any indicial term of the linear operator. In particular, the right hand side has an expansion where  terms are given by $r^{2\beta+\ell+2k\beta}a_{\ell k}, \ell\geq 0, k\geq 0$ and its coefficient $a_{\ell k}$ is a trigonometric polynomial of degree at most $\ell$. 

The equation for $u_{2}'$ is 
\begin{equation}
\left( (r\partial_{r})^{2}+\partial_{\phi}^{2} \right)u_{2}'+2r^{2\beta}e^{2u_{0}'}u_{2}'=-r^{2\beta}\left(K_2+2(u_{1}')^{2}\right),
\label{u2p}
\end{equation}
and this right hand side has an expansion 
\begin{align}\label{e:u2rhs}
r^{2\beta}\big(\sum_{\ell\geq 0}\sum_{\ell'=0}^{\ell} (c_{\ell \ell'} \cos (\ell' \phi) + d_{\ell  \ell'} \sin(\ell' \phi) )r^{\ell}\\+\sum_{\ell\geq 0, k \geq 1} \sum_{\ell'=0}^\ell (c_{\ell k \ell'} \cos (\ell' \phi) + d_{\ell k \ell'} \sin(\ell' \phi) ) r^{\ell + 2k\beta}\big)
\end{align}
Indeed, the expansion of $r^{2\beta}(u_{1}')^{2}$ contains terms
\begin{equation}\label{e:ujrhs}
r^{\ell+2\beta}q_{\ell}(\phi), \ r^{\ell+2(k+1)\beta}p_{\ell k}(\phi), k\geq 0.
\end{equation}
The coefficient $q_{\ell}, p_{\ell k}$ are finite sum of finite products of trigonometric polynomials $\prod q_j$, with 
$\deg q_j = \ell_j$ and $\sum \ell_j = \ell$. Even if each $q_j$ is pure, this product usually includes all lower degrees as well. Therefore all $q_{\ell}, p_{\ell k}$ are trigonometric polynomials of degree at most $\ell$.  
Since $K_{2}$ vanishes identically near $r=0$, we can see that the right hand side does not contain any indicial terms $r^{\ell}e^{i\ell\phi}$ because of the $r^{2\beta}$ shift, which implies that the solution $u_{2}'$ does not contain any $\log r$ terms. Solving \eqref{u2p} term-by-term gives
\begin{equation}\label{e:u2}
u_{2}'\sim \sum_{\ell\in \NN} r^{\ell} a_{2\ell 0}(\phi) + \sum_{\ell, k \in \NN, \ell\geq 0, k\geq 1} r^{\ell+2k \beta} a_{2\ell k}(\phi).
\end{equation}
Here the first term contains indicial roots where $a_{2\ell 0}(\phi)$ is a linear combination of $\cos(\ell \phi)$ and $\sin (\ell \phi)$, 
while the second term comes from matching coefficients on two sides and $a_{2\ell k}$ is of degree at most $\ell$. 

Now suppose~\eqref{e:uj1} is true for $u_i'$, $i < j$. Then the terms on the right hand side of \eqref{e:ind} are linear combinations of terms
$
\prod_{\sum i=j} u_{i}'.
$
By tracking the terms in~\eqref{e:uj1},  we obtain that the right hand side 
is a linear combination of terms in~\eqref{e:ujrhs}.
Applying the same guiding
principle, we see that $u_j'$ has an expansion as in \eqref{e:uj1}.
By induction, this concludes the proof of the proposition. 
\end{proof}

\medskip

\noindent{\bf Expansion at $\frakC_{12}$}
We next consider the expansion at $\frakC_{12}$.  Unlike the preceding construction at $M_{\frakp'}$, some
terms in this expansion can only be determined once we take into account their compatibility with
the previous expansion. Write 
\begin{equation}
u\sim \sum_{\alpha \in \calE} R^{\alpha}  \tilde{u}''_{\alpha}(s,\phi)
\label{expC12}
\end{equation}
near this face, where $\calE$ is an index set which is determined in the course of the argument below, see \eqref{index}. 
The double prime indicates that the terms are coefficients in the expansion near $\frakC_{12}$. 
As usual, $R = \sqrt{\rho^{2}+r^{2}}$, $R \sin \omega = \rho$, and as before, we set $s = \sin \omega$.
(This is a good coordinate away from the pole of this hemisphere.)  For many purposes it is simpler to use the 
projective coordinates $\wh{z}=z/\rho$ and $\rho$, which are valid on the interior of $\frakC_{12}$; 
$|\wh z| \to \infty$ at the outer boundary of this face, so $s \sim 1/|\wh z|$, and $\rho$ is only a defining function 
for this face away from its outer boundary. See Figure~\ref{f:coordinates} for an illustration of the coordinates.

In these projective coordinates, still writing $\beta = \beta_{12}$,
\begin{equation}\label{rescnear12}
g_{0,\frakp} = \rho^{2\beta} e^{2\hat{u}} |d\wh z|^2,\ \ \hat{u} = (\beta_1-1) \log |\wh z-1| + (\beta_2-1) \log |\wh z+1| + \tilde{u},
\end{equation}
where $\tilde{u}$ is harmonic as a function of $\wh z$, so in particular is smooth across the singular points $\wh z = \pm 1$ 
in this face. It is also bounded in a neighborhood of $\frakC_{12}$, so in fact its restriction to $\frakC_{12}$ must be constant. 
In other words, $\rho^{-2\beta} g_{0,\frakp}$ restricts to a flat metric $\hat{g}$ on the interior of $\frakC_{12}$ 
which has two conic singularities at $w = \pm 1$ and is asymptotic to the large end of a cone with cone angle 
$2\pi \beta$ as $\wh z \to \infty$. 

Now write $R^{\alpha}\tilde u''_{\alpha}=\rho^{\alpha}s^{-\alpha}\tilde u''_{\alpha}=\rho^{\alpha}u''_{\alpha}$ where $u''_{\alpha}(s,\phi)
=s^{-\alpha}\tilde u''_{\alpha}$. Using that the fiber Laplacian commutes with fiber variable $\rho$, i.e. $[\Delta_{g_{0,\frakp}}, \rho]=0$,
the curvature equation thus leads to equations for each of these coefficients: 
\begin{equation}
\wh{\Delta} u_{\alpha}'' = - 2u_{\alpha - 2\beta}'' - ( e^{2 u^{(\alpha-2\beta)}} -1 - 2u^{(\alpha-2\beta)})_{\alpha - 2\beta},
\label{cce}
\end{equation}
where here and below, we denote by $\wh{\Delta}$ the Laplacian for the flat conic metric $\hat{g}=\rho^{-2\beta}g_{0,\frakp}$. 
Also, analogous to our previous notation, $u^{(\alpha-2\beta)}$ denotes the sum of all terms $\rho^{\alpha'}u_{\alpha'}''$ with 
$\alpha'<\alpha-2\beta$, and  $(\cdot )_{\alpha - 2\beta}$ indicates the coefficient of $\rho^{\alpha - 2\beta}$ in the expansion in parentheses.  
The downward shift by $2\beta$ occurs because of the factor $\rho^{-2\beta}$ on $\Delta_{g_{0, \frakp}}$. 

We now analyze these coefficients.  The shift of exponents here motivates the fact that we carry this
out for $\alpha$ in the succession of ranges
\[
2\ell \beta < \alpha < 2(\ell+1) \beta,  \ \ \ell = 0, 1, 2, \ldots;
\]
the endpoints $2\ell \beta$ are handled separately. 

\medskip

\noindent{\bf The case $\alpha \in (0, 2\beta )$:} 

\begin{lemma}
The  only terms in \eqref{expC12} with $\alpha\in (0,2\beta)$ are those for which $\alpha \in \{ 1, 2, \ldots, [2\beta] \}$.
And the term $u_{\alpha}''$ is determined by $\{a_{(\alpha-j)j0}(\phi), 0\leq j\leq \alpha-1\}$ in~\eqref{e:uj1}.
\end{lemma}
\begin{proof}
Write the fiberwise Laplacian $\Delta_{g_{0,\frakp}}$ locally as $\rho^{-2\beta}\wh{ \Delta}$,
where $\wh{\Delta}$ is the Laplacian for the flat conical metric $\hat g=e^{2\hat u}|d\hat{z}|^2$.  
Inserting \eqref{expC12} into the equation,
and recalling that $\wh{\Delta}$ commutes with $\rho$, we obtain that $\wh{\Delta} u_\alpha'' = 0$ for $\alpha < 2\beta$. 
so $u_{\alpha}''$ is harmonic with respect to $e^{2\hat{u}}|d\wh z|^2$, and hence also with respect to $|d\wh z|^2$. 
 
The fact that the original coefficient $\tilde{u}_{\alpha}''$ is bounded as $s \to 0$, i.e., $\wh z \to \infty$, means that  
$u_{\alpha}'' = \tilde{u}''_{\alpha} s^{-\alpha}$ grows at most like $s^{-\alpha} \sim |\wh z|^\alpha$. This means that it
is a harmonic polynomial $p_{\alpha}(\wh z)$ of degree less than or equal to $\alpha$. 
Now if $\alpha$ were not an integer, then $s^{\alpha} p_{\alpha}(\wh z)$ would be a sum of terms, each 
vanishing at a nonintegral rate as $s \to 0$.  This is impossible since $u$ is smooth in $s$ at $s=0$.
Hence the only allowable exponents $\alpha < 2\beta$ are nonnegative integers.  When $\alpha=0$, the only possibility for $u_{\alpha}''$ would be a degree 0 harmonic polynomial hence a constant, which leads to the constant term $c\rho^{0}r^{0}$ in the expansion, and by the choice of $u_{0}'$ we know that this constant must vanish. Therefore we can assume $1\leq \alpha<2\beta$.

Write $u''_{\alpha}=p_{\alpha}=\sum_{j=0}^{\alpha} p_{\alpha j}$ where each $p_{\alpha j}(\hat z)$ is a harmonic polynomial of degree $j$. The term $p_{\alpha j}$ corresponds to a term of growth $\rho^{\alpha} s^{-j}a_{j}$ where $a_{j}$ is a linear combination of $\cos \ell \phi$ and $\sin \ell \phi$.  Compatibility at the corner means that this must match the coefficient of $s^{-j+\alpha} r^{\alpha}$, that is $a_{(\alpha-j)j0}(\phi)$ in~\eqref{e:uj1}. And as we observed there, these coefficients indeed have pure degree $j$. In particular, when $j=\alpha$, this corresponds to a term $\rho^{0}r^{\alpha}$ which vanishes from the expansion of $u_{0}'$. Therefore $0\leq j\leq\alpha-1$, and there is a unique homogeneous harmonic polynomial $p_{\alpha j}$ which satisfies
this boundary condition.   
This determines $u_\alpha$ for any integer $ \alpha< 2\beta$.
\end{proof}

\noindent{\bf The case $\alpha = 2\beta$} 

\begin{lemma}
When $\alpha=2\beta$, $u_{\alpha}''$ is determined by $a_{01}$ in~\eqref{e:u0}.
\end{lemma}
\begin{proof}
Essentially the same calculation as above yields that 
\[
\wh{\Delta} u_{2\beta}'' + 1 = 0,  \quad u_{2\beta}'' = \tilde{u}''_{2\beta} s^{-2\beta},
\]
Here 
\[
\wh {\Delta} = e^{-2U}\Delta_{\wh z}, \ e^{2U} = e^{(2 ((\beta_1-1) \log |\wh z - 1| + (\beta_2-1)\log |\wh z + 1|))},
\]
where $\Delta_{\wh z}$ is the Laplacian for $|dz|^{2}$. Noting that $e^{2U} \sim c|\wh z|^{2\beta - 2}$ as $|\wh z| \to \infty$ for some constant $c \neq 0$, 
it follows from well-known existence theory for the Laplacian on asymptotically conic manifolds that 
there exists a solution $u_{2\beta}''$ to this equation which asymptotic to $A |\wh z|^{2\beta}$ for some constant $A$.
This solution is unique up to harmonic polynomials of degree strictly less than $2\beta$. However, as before, 
any such harmonic polynomial would lead to a term in the expansion of $\tilde{u}_{2\beta}''=u_{2\beta}'' s^{2\beta}$ which is not 
smooth at $s=0$, and this is impossible.  In addition, $R^{2\beta}u_{2\beta}'' s^{2\beta} \to A R^{2\beta}$ as $s \to 0$, so $A$ 
must equal the constant $a_{01}$ in~\eqref{e:u0}.  This determines $u_{2\beta}''$ uniquely.  
\end{proof}

\medskip

We have now explained all coefficients for the initial part of our index set: 
$$
\calE_{2\beta} := \{ \alpha \in \calE, \ \alpha\leq 2\beta\}=\{ 1,2, \dots, [2\beta], 2\beta\}.
$$

\medskip

\noindent{\bf The case $2\beta < \alpha < 4\beta$}. 

Now we consider the cases for $\alpha\in (2\beta, 4\beta)$.
\begin{lemma}
When $\alpha\in (2\beta, 4\beta)$, the index set $\calE \bigcap (2\beta, 4\beta)$ is given by 
\begin{align}
\{(2\beta<\alpha<4\beta: \ \alpha-2\beta\in \NN \mbox{ or } \alpha\in \NN\}.
\end{align}
When $\alpha=\ell+2\beta$ for some $\ell\in \NN$, $u_{\alpha}''$ is determined by $\{a_{(\ell-j)j1}:0\leq j\leq \ell-1\}$ in~\eqref{e:u1}. When $\alpha\in \NN$, $u_{\alpha}''$ is determined by $\{a_{(\alpha-j)j0}(\phi): 0\leq j \leq \alpha-1\}$. 
\end{lemma}  
\begin{proof}
When $\alpha<4\beta$, based on whether the right hand side of~\eqref{cce} is trivial, there are two cases. The first is when 
$\alpha-2\beta= \ell \in \{1, 2, \ldots [2\beta]\}$.  
In this case we get the inhomogeneous equation 
\begin{equation}\label{e:inhg}
\wh \Delta u_{\alpha}''  = -2 u''_{\ell} - \left(e^{2u^{(\ell)}}-1-2u^{(\ell)}\right)_{\ell}
\end{equation}
where as before, $(\cdot)_{\ell}$ is the coefficient of $\rho^\ell$ in the expansion of the expression
in parentheses. The right hand side is a linear combination of terms
\begin{equation}\label{e:gamma}
u_{j_1}'' \ldots u_{j_k}''
\end{equation}
where the sum is over all partitions $(j_1, \dots , j_k)$ with $\sum_{i=1}^{k} j_i=\ell$. Recall that each $u_{j_{i}}''$ is a sum of harmonic polynomials in $\wh z$ of degrees strictly less than $j_{i}$.
When there is only one term in the sum of~\eqref{e:gamma}, i.e. $k=1$, then $u_{j_{1}}''=u''_{\ell}$ is a sum of terms  $p_{\ell r}, \ell-1\geq r \geq 0$ where each $p_{
\ell r}$ is a harmonic polynomial of pure degree $r$, and in particular near the ``infinity'' $s=0$ the angular coefficient of $s^{-(\ell-1)}$ is a linear combination of $\cos((\ell-1)\phi)$ and $\sin((\ell-1)\phi)$. 
On the other hand, when there are at least factors in the summand, i.e. $k\geq 2$, then  
each of these products is a homogeneous polynomial of
degree $j\leq \ell-2$, and near infinity the angular coefficients is a trigonometric polynomial of degree no more than  $j$. To combine these two situations, the right hand side is given by 
$$
s^{-(\ell-1)}c_{\ell-1}(\phi) + \sum_{0\leq i \leq \ell-2}s^{-i}d_{i}(\phi),
$$
where $c_{\ell-1}$ is of pure degree $\ell-1$ and $d_{i}$ is mixed of degree at most $i$.

The solution $u_{\alpha}''$ is the sum of an inhomogeneous term $\sum_{i=0}^{\ell-1} q_{i}$ and potential homogeneous terms.
Here each inhomogeneous term $q_i(\hat z)$ solves away the $s^{-i}$ term in the above expansion, hence $q_{i}\sim A_{i}(\phi)s^{-i-2\beta}$. For the top degree $i=\ell-1$, $A_{\ell-1}$ is a linear combination of $\cos ((\ell-1) \phi)$ and $\sin ((\ell-1) \phi)$. And for the lower degrees $i<\ell-1$, $A_{i}$ is a trigonometric polynomial of degree at most $i$. Those coefficients are matched at the corner, since the term $q_{i}$ would lead to $\rho^{\ell-i} r^{i+2\beta}A_{i}$, so $A_{i}$ is given by $a_{(\ell-i)i1}$ in~\eqref{e:uj1}. In particular, when $i=\ell-1$, $A_{\ell-1}$ is matched by $a_{1(\ell-1)1}$ which is indeed of pure degree $\ell-1$, while for other $i\leq \ell-2$, $A_{i}$ is given by $a_{(\ell-i)i1}$ which is a trigonometric polynomial of degree at most $i$ from Proposition~\ref{p:Mp}.
 
Regarding the potential homogeneous terms, $u_{\alpha}''$ is unique up to addition by harmonic polynomials of degree strictly less than $\alpha-2\beta$. However this would give a term in $u_{\alpha}=u_{\alpha}''s^{\alpha}$ which is not smooth. By the same reasoning as in the case $\alpha=2\beta$, we have shown that  $u_{\alpha}''$ is uniquely determined by  coefficients listed above.

The other case for the equation~\eqref{e:inhg} is given by the homogeneous equation
$\wh \Delta u_{\alpha}''=0$.
By the same reasoning as before, $\alpha$ must be an integer, and $u_{\alpha}''=\sum p_{\alpha j}(\hat z)$ where each $p_{\alpha j}$ is a harmonic polynomial of degree $j \leq \alpha-1$. And for each $j$, the boundary asymptotic of the term $\rho^{\alpha}p_{\alpha j}(\hat z) $ is given by $\rho^{\alpha - j} R^{j} A_{\alpha j}(\phi)$ which is a linear combination of $\sin j\phi$ and $\cos j\phi$, hence is matched by the coefficient $a_{(\alpha-j)j0}(\phi)$.
\end{proof}

\medskip

\noindent{\bf The case $ \alpha = 4\beta$}. 

\begin{lemma}
When $\alpha=4\beta$, $u_{\alpha}''$ is determined by $a_{02}$ in~\eqref{e:u0}.
\end{lemma}
\begin{proof}
When $\alpha=4\beta$, the term $R^{4\beta}u_{4\beta}$ solves $\wh \Delta u''_{4\beta}=2u''_{2\beta}\sim A|\wh z|^{2\beta}$. Using the same argument as for $u''_{2\beta}$, $u''_{4\beta}$ is unique and asymptotic to $B|\wh z|^{4\beta}$ where $B$ is given by the constant $a_{02}$ in~\eqref{e:u0}.
\end{proof}

Hence we have shown
\begin{equation}
\calE_{4\beta}
=\{(2\beta<\alpha<4\beta: \ \alpha-2\beta\in \NN \mbox{ or } \alpha\in \NN\}\bigcup \{4\beta\}.
\end{equation} 

\medskip

\noindent{\bf The case $2(n-1)\beta < \alpha \leq 2n\beta$}. 

Iteratively we can repeat the argument for $\alpha \in (2(n-1)\beta, 2n\beta]$. 
\begin{lemma}
The index set $\calE_{2n\beta}$ is given by
$$
\{\alpha\in (2(n-1)\beta, 2n\beta]: \alpha=j+2k\beta\}. 
$$
For $\alpha=j+2k\beta$, $j,k\geq 1$, $u_{\alpha}''$ is determined by $\{a_{(j-\ell)\ell k}: 0\leq \ell\leq j-1\}$ in~\eqref{e:uj1}. When $\alpha=2n\beta$, $u_{\alpha}''$ is determined by $a_{0n}$ in~\eqref{e:u0}. When $\alpha\in \NN$, $u_{\alpha}''$ is determined by $\{a_{(\alpha-j)j0}: 0\leq j \leq \alpha-1\}$ in~\eqref{e:uj1}.
\end{lemma}
\begin{proof}
As before there are two cases: $\alpha-2\beta=\sum \alpha'$ for some $\alpha'\in \calE$, which by induction means $\alpha=j+2k\beta$ with $n\geq k\geq 1$, $j\geq 0$, and $k=n$ if and only if $j=0$; or $\alpha \in \NN$. Note that the endpoint $\alpha=2n\beta$ is included in the first case. 

In the first case, assuming $j>0$ hence $k\leq n-1$, then $u_{\alpha}''$ solves an inhomogeneous equation~\eqref{e:inhg} where the right hand side is a sum of terms 
\begin{equation}\label{e:gamman}
u_{\alpha_{1}}''\dots u''_{\alpha_{m}}, \ \sum_{i} \alpha_{i}=\alpha-2\beta=j+2(k-1)\beta.
\end{equation}
By induction, each term $u''_{\alpha_{i}}$, where $\alpha_{i}=j_{i}+2k_{i}\beta$, is a sum of terms $\sum_{\ell=0}^{j_{i}-1} A_{\ell}s^{-\ell-2k_{i}\beta}$ such that $A_{j_{i}-1}(\phi)$ is of pure degree $j_{i}-1$ and other $A_{\ell}$ is mixed of degree at most $\ell$.  Therefore, as discussed before, the terms in~\eqref{e:gamman} are characterized in two categories: (a) $m=1$, i.e. $\alpha_{1}=j+2(k-1)\beta$, then the first term is given by $A_{j-1}(\phi)s^{-(j-1)-2(k-1)\beta}$ where $A_{j-1}$ is of pure degree $j-1$, while the rest of the terms would combine with (b); (b) $m\geq 2$, then the product contains terms $A_{\ell}s^{-\ell-2(k-1)\beta}$ where $\ell\leq j-2$, and each $A_{\ell}$ is mixed of degree at most $\ell$. 
For the same reason as before, $u_{\alpha}''$ has a unique solution which is asymptotically given by $\sum_{\ell=0}^{j-1} A_{\ell}s^{-\ell-2k\beta}$, and each term leads $A_{\ell}\rho^{j-\ell}R^{\ell+2k\beta}$, and each $A_{\ell}(\phi)$ is determined by coefficients $a_{(j-\ell)\ell k}$ in~\eqref{e:uj1}. In particular, only when $\ell=j-1$,  $a_{(j-\ell)\ell k}$ is of pure degree $\ell$, while for other $\ell\leq j-2$, $a_{(j-\ell)\ell k}$ is mixed of degree $\ell$. So it is matched.

On the other hand, if in the first case $j=0$ and $k=n$, then $\alpha=2n\beta$. Then $u_{\alpha}''$ satisfies the same equation~\eqref{e:inhg} with the special requirement that all the $\alpha_{i}$ in~\eqref{e:gamman} are of the form $2j\beta$. Then by induction the right hand side is given by $A|\hat z|^{2(n-1)\beta}$, hence $u_{2n\beta}''$ is unique and asymptotic to $B|\hat z|^{2n\beta}$ which is matched by coefficient $a_{0n}$.

In the second case ($\alpha\in \NN$), $u_{\alpha}''$ solves the homogeneous equation $\wh \Delta u_{\alpha}''=0$ and is a combination of harmonic polynomials of degree $j< \alpha$, and by the same argument as before, each term is determined by $\{a_{(\alpha-j)j0}: 0\leq j\leq \alpha-1\}$.
\end{proof}

With the discussion above, we have 
\begin{proposition}
The index set $\calE$ is given by
\begin{equation}\label{index}
\calE=\{j+2k\beta: j,k \in \NN \}.
\end{equation}
\end{proposition}

For each term $R^{\alpha}\tilde u''_{\alpha}$ with $\alpha=j+2k\beta$,  $\tilde u''_{\alpha}(s,\phi)$ is smooth up to $s=0$ 
and asymptotically given by a sum of terms with growth $\{s^{\ell}: \ell\in \NN, \ell\leq j\}$. That is, for the solution near 
the corner there is a product-type expansion
$$
u\sim \sum_{\alpha=j+2k\beta} \sum_{\ell = 0}^{j-1} R^{\alpha}s^{\ell} u_{\alpha \ell}(\phi).
$$

\medskip

\noindent{\bf The approximate solution}

\begin{lemma}\label{l:faceII}
There exists a polyhomogeneous function 
\[
\tilde u\sim \sum_{\alpha=j+2k\beta \in \calE, \ell\leq j}  R^{\alpha}s^{\ell}u_{\alpha \ell}(\phi) + \calO(\rho^{N+\epsilon})
\]
which satisfies
\begin{equation}
\Delta_{g_{0,\frakp}} \tilde{u}+ e^{2\tilde{u}}+K_{g_{0,\rho}} = \calO( \rho^{N} )
\label{appsolnN}
\end{equation}
for any $N \geq 0$ as $\rho = R s \to 0$.
\end{lemma}
\begin{proof}
We take $\tilde{u}$ to be a Borel sum of the formal polyhomogeneous series constructed above; \eqref{appsolnN} is
then obvious.
\end{proof}

We now write 
\begin{equation}
\tilde g_{0}=e^{2\tilde{u}}g_{0,\frakp}.
\label{appsol}
\end{equation}

\medskip

\noindent{\bf Correction to an exact solution}
The final step is to correct the approximate solution to an exact one by solving 
\[
\Delta_{\tilde g_{0}}v+e^{2v}+K_{\tilde g_{0}}=0
\]
for each $\rho$, or equivalently, 
\begin{equation}\label{e:hypnl}
\Delta_{\tilde g_{0}}v+2v = -(K_{\tilde g_{0}}+1)-(e^{2v}-1-2v) 
\end{equation} 
Of course, this solution is already known to exist and be unique, but the method here will show that it is polyhomogeneous on
$\calC_k$.   Indeed, we find a solution to \eqref{e:hypnl} satisfying $|v| \leq C \rho^N$ for any fixed $N$; these are all the
same by uniqueness of course. However, in this way we can estimate the $b$-derivatives of this solution up to that order,
and hence, since $N$ is arbitrary, to all orders. 

For convenience, write $f = - (K_{\tilde{g}_0} + 1)$; we have arranged that $f$ is smooth and vanishes to all orders at $\rho=0$. 
We also set $Q(v) = - (e^{2v}-1-2v)$. 

\begin{proposition}\label{p:maxPr}
For each $N>0$ and $0 < \rho < \epsilon$, there is a unique bounded solution $v$ to \eqref{e:hypnl} such that $|v|\leq C_{0}\rho^{N}$ for some constant $C_{0}$.
\end{proposition}
\begin{proof}
By the maximum principle, if $(\Delta + 2) w = h$, then $\sup |w| \leq \tfrac12 \sup|h|$. For a given choice of $N$,
there exists a constant $C_0$ such that $|f| \leq C_0 \rho^N$ for all sufficiently small $\rho$. There is also a constant
$C_1$ so that $|Q(v)| \leq C_1 |v|^2$ when $|v| \leq 1$.  Now define the sequence $v_j$ by $v_0 = 0$ and
\[
(\Delta_{\tilde{g}_0} + 2) v_{j+1} = f + Q(v_j).
\]
By the remarks above,  $\sup |v_{j+1}| \leq  \tfrac12 C_0 \rho^N + \tfrac12 C_1 \sup|v_j|^2$. 
Assume inductively that $\sup |v_j| \leq A_j = C_0 \rho^N$. Then 
\[
\sup|v_{j+1}| \leq \tfrac12 C_0 \rho^N + \tfrac12 C_1 A_j^2 \leq \tfrac12 ( C_0 \rho^N + C_1 C_0^2 \rho^{2N}) 
= C_0 \rho^N ( \tfrac12 +  \tfrac12 C_1 C_0 \rho^N),
\]
which we can make less than $C_0 \rho^N$ by choosing $\rho < (C_1 C_0)^{-1/N}$. 

It now follows by standard theory that the $v_j$ converge to a solution $v$ which satisfies $|v| \leq C_0 \rho^N$, and
the maximum principle shows that this solution is unique, and in particular independent of $N$.

\end{proof}

\medskip

\noindent{\bf Polyhomogeneity of the solution}  
In order to prove that the solution $u = \tilde{u} + v$ is polyhomogeneous, it suffices to prove that
$v$ is conormal of order $N$, i.e., that $W_1 \ldots W_\ell v = \calO(\rho^N)$ for any $W_j \in \calV_b(\calC_k)$ and for any $\ell$.

Observe that the problem localizes near $\frakC_{12}$ since the polyhomogeneity of $u$ in all regions where
conic points are not coalescing was proved in \cite{MW}.   Now apply a single $b$-vector field $W$ to \eqref{e:hypnl}. This gives
\[
(\Delta_{\tilde{g}_0} + 2) Wv = Wf + 2 W\tilde{u} \Delta_{\tilde{g}_0} v + W Q(v).
\]
This uses that $\Delta_{\tilde{g}_0} = e^{-2\tilde{u}}\Delta_{g_{0,\frakp}}$.  Now $|Wf| \leq C \rho^N$ and $|\Delta_{\tilde{g}_0} v| \leq |-2v + 
f + Q(v)| \leq C \rho^N$, while $|W \tilde{u}| \leq C$. Finally, we can write $W Q(v) = \lambda Wv$
where $\lambda$ is a smooth function which is also bounded by $C \rho^N$. Absorbing this last term on the left hand side
perturbs the constant $2$ by a small amount, so we can bound $Wv$ by $C' \rho^N$ by the maximum principle. 

This argument can obviously be iterated any number of times, which shows that $v$ is indeed conormal of order $N$. 

We have proved the 

\begin{proposition}\label{p:poly}
When $k=2$, or more generally near the locus in $\calC_k$ where precisely two conic points collide, the solution $u$ 
to \eqref{e:hyp} is polyhomogeneous. 
\end{proposition}

\subsection{Hyperbolic metrics with an arbitrary number of merging cone points}\label{ss:hypk}
The construction for the case of two merging cone points initiates and provides the pattern for an inductive
argument to prove the corresponding regularity result for the solution $u$ when an arbitrary number of
points coalesce. The construction in that simpler case was given in sufficient detail that the steps below 
for $k > 2$ are straightforward generalizations.

\begin{theorem}\label{t:hyp}
Fix $k \geq 2$ and $\vec \beta$. Then the family of fiberwise conic hyperbolic metrics $g_{\frakp,\beta}$ 
is polyhomogeneous on $\calC_{k}$. 
\end{theorem}
\begin{proof}
We begin as before, writing $g_{\frakp,\beta}=e^{2u}g_{0,\frakp}$, where $g_{0,\frakp}$ is the metric which is flat in a neighborhood
of the coalescing points. We have already proved that this flat metric is polyhomogeneous, so it suffices to prove that
$u$ is polyhomogeneous on $\calC_{k}$.

As in \S 2.5,  for any $q \in \calC_{k}$, there is a (non-augmented) tree $T$ and a terminal node on that tree encoding the 
chain of faces leading to $q$, where if $\calI$ is the node corresponding to $q$, then $q$ lies in the interior of $\frakC_{\calI}$. 
The argument below is an induction on the depth $N$ of the node, i.e., the height of the tree from the root up to this node. 

We have essentially already given the construction of an approximate solution when $N=1$.  More precisely, in the
extension of the approximate solution from the fiber $M_{\frakp'}$ to the face $\frakC_{12}$, that face is blown up at 
the two points where the incidence sets $F^\sigma_1$ and $F^\sigma_2$ meet $\frakC_{12}$.  When $k > 2$ (but we
are still considering the case $N=1$, the corresponding construction involves blowing up $\frakC_{12\ldots k}$ at 
the $k$ distinct points where the $F^\sigma_j$ meet this face.  The induced metric on $\frakC_{12\ldots k}$
is flat, with $k$ conic singularities at these intersection points, and a complete conic structure near the
outer boundary. The solvability of the sequence of equations to determine the expansion at this face
proceeds exactly as before. 

Suppose now that we have described how to carry the construction out for all trees of height strictly less than $N$.
Let $T$ be a tree with height $N$ and $\calI$ a terminal node (so that only singleton incident sets $F^\sigma_j$, $j \in \calI$
intersect it).  There is a maximal ascending chain $\calI = \calI_{N}\subset \dots \subset \calI_{1}$, and by
induction we have constructed the full series expansion for the approximate solution near each of the faces 
$\frakC_{\calI_\ell}$, $\ell \leq N-1$.  In particular, there is a complete series at the penultimate face
$\frakC_{\calI_{N-1}}$, which is a hemisphere blown up at $\tau$ points, where $\tau$ is the number of terminal
vertices emanating from the vertex $\calI_{N-1}$ in $T$.  This face carries a flat metric $g_{N-1}$, which has 
incomplete conic singularities at these $\tau$ interior boundaries as well as a complete conic structure at its outer boundary. 

We now choose the terms in the series expansion at $\frakC_{\calI_N}$.  This process is again almost identical to the 
one for $k=2$.   To simplify notation, drop subscripts, and let $r$ and $s$ denote the radial variable to 
$\frakC_{\calI_N} \cap \frakC_{\calI_{N-1}}$ in $\frakC_{\calI_{N-1}}$ and $\frakC_{\calI_N}$, respectively. No analogue of 
Lemma \ref{l:faceI} is necessary here since $g_{N-1} = dr^2 + \be^2_N r^2 d\theta^2$ near this boundary.  The 
approximate solution $u$ has an expansion $u \sim \sum s^j u_j'$ as in Proposition \ref{p:Mp}.  These coefficients
are then used to determine the boundary values for the coefficients in the expansion 
$u \sim \sum R^\alpha u_\alpha''$ as $s \to 0$; here $R$ is the radial variable to $\frakC_{\calI_N}$ (which restricts to 
$\calC_{\calI_{N-1}}$ as $r$).  
\begin{figure}[h]
\centering
\includegraphics[width = 0.7\textwidth]{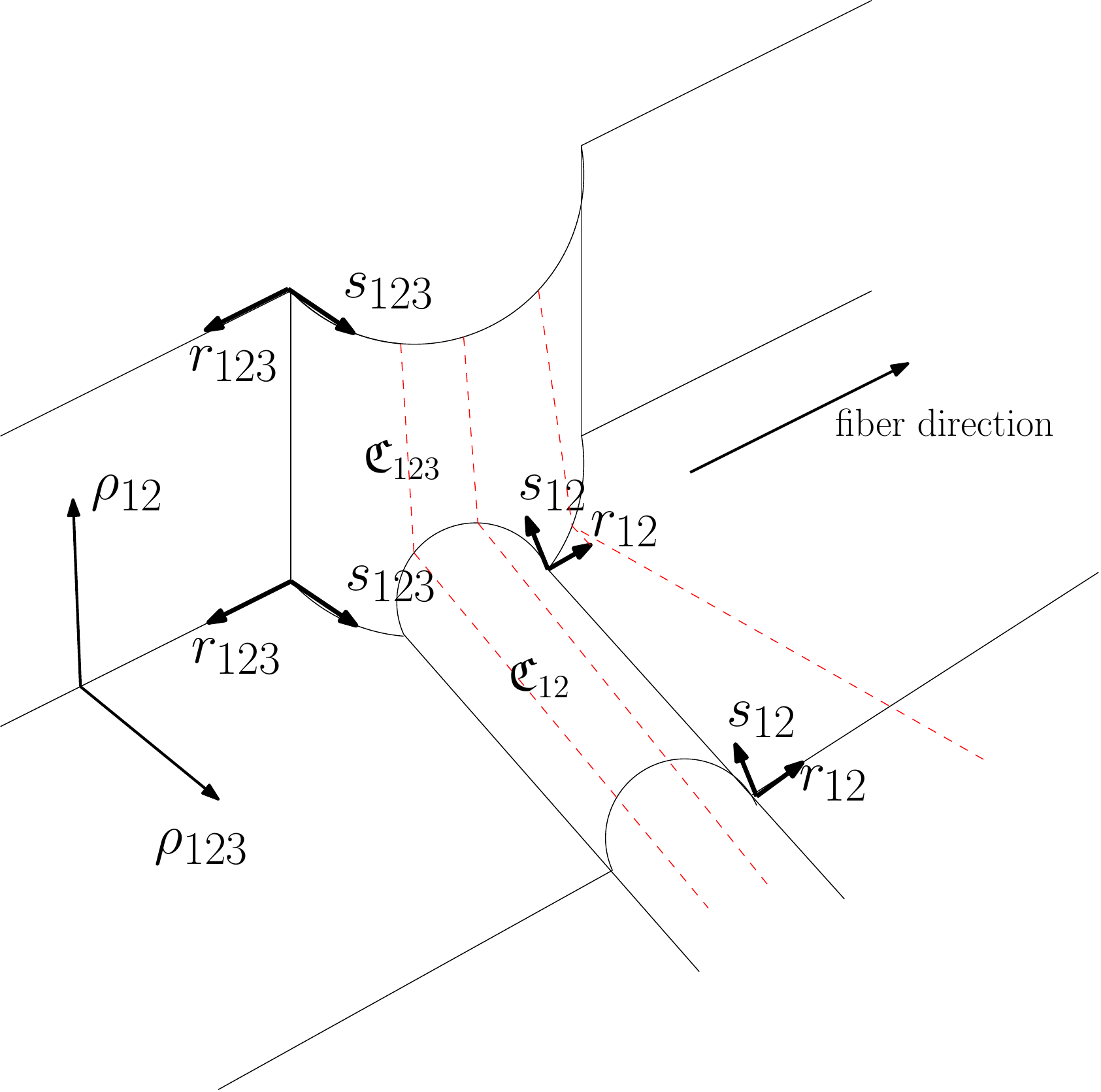}
\caption{Boundary faces and defining functions of $\calC_3$}
\label{doublestack}
\end{figure}
Rewrite this expansion as $u \sim \sum \rho^\alpha \tilde{u}_\alpha''$ (so we can commute 
the Laplacian past the powers of $\rho$).  In terms of the defining functions $R_j$ for the intermediate faces $\frakC_{\calI_j}$, 
where $R_N = R$, $R_{N-1} = s$ and $R_0$ equal to a defining function for $M_{\frakp'}$, we can 
take $\rho = R_0 \ldots R_N$, so that 
\[
u_\alpha'' = \tilde{u}_\alpha'' (R_0 \ldots R_{N-1})^\alpha.
\]
However, on $\frakC_{\calI_N}$, $R_0, \ldots, R_{N-2}$ are all constant, so more simply  $u_\alpha'' = \tilde{u}_\alpha'' s^\alpha$. 
This puts us in the same situation as before, where we must show that $\tilde{u}_\alpha'' = 0$ unless $\alpha \in \NN$,
and in that case, $\tilde{u}_\alpha''$ is a polynomial of order $\alpha$ so that $u_\alpha''$ is smooth
up to $\frakC_{\calI_{N-1}}$.   The calculations and arguments here are just as before. 

This paragraph does not conceal any delicate points; the admittedly intricate discussion carried out to
extend the expansion from $M_{\frakp'}$ to $\frakC_{12}$ adapts directly to the extension 
from $\frakC_{\calI_{N-1}}$ to $\frakC_{\calI_N}$. 

Now take a Borel sum of this multi-series at all the faces of $\calC_k$ to obtain a function $\tilde{u}$ for which
\[
\Delta_{g_{0,\frakp}} \tilde{u} + K_{0,\frakp} + e^{2\tilde{u}} = \calO(\rho^\ell) \ \ \mbox{for all}\ \ell \geq 0.
\]
Writing $\tilde{g}_0 = e^{2\tilde{u}} g_{0,\frakp}$, we  then define $v$ uniquely by 
\[
\Delta_{\tilde{g}_0} v + K_{\tilde{g}_0} + e^{2v} = 0
\]
with $|v| \leq C_\ell \rho^\ell$ for any $\ell$.  This shows that the solution $u$ is polyhomogeneous on 
$\calC_k$. 
\end{proof}

\section{Spherical metrics: cone angles less than $2\pi$}\label{s:sph}
We conclude this paper by extending Theorems \ref{flat} and \ref{t:hyp} to the spherical case.
As we have explained earlier, the existence theory for spherical cone metrics is completely understood only
when all cone angles are less than $2\pi$, so we restrict ourselves to that case here.  
In a sequel to this paper we investigate this problem for spherical metrics with large cone angles. The results
in that case are considerably more intricate and of a slightly different nature than the considerations here. 

First note that if $\vec \beta \in (0,1)^{k}$, then the cone angle produced by merging any subset of these must still have 
cone angle less than $2\pi$: 
$$
2\pi \left( \sum_{i\in\calI}(\beta_{i}-1)+1 \right) \in (0,2\pi), \forall \calI \subset \{1,\dots, k\}.
$$
We also recall that a (spherical) football is the spherical suspension of a circle of length $2\pi\beta$; it is a surface of genus zero 
with two antipodal conic points, each with angle $2\pi\beta$.  This angle may be any positive number. 

The main regularity theorem and its proof are very similar to those in the hyperbolic case.  Indeed, the only real
difference is that it is no longer immediately obvious that the linearized operator is invertible, but fortunately,
this is not the case. 
\begin{lemma}[\cite{MW}, Proposition 13]\label{l:invertible}
If $g$ is a spherical cone metric on the sphere with all cone angles less than $2\pi$, then the first nonzero eigenvalue
for the Friedrichs extension of its Laplacian is always strictly greater than $2$, unless $(S^2,g)$ is a football, in
which case this eigenvalue is exactly equal to $2$. 
\end{lemma}
In our second paper the primary focus will be on the extensions of this analysis when the spectrum. 

\begin{remark}
We shall restrict here to the case $k \geq 3$ and to the fibers of $\calC_k$ lying in the dense open subset $\calC_k'$ 
which do not lie above the preimage in $\calE_k$ of any of the following sets:
\begin{itemize}
\item[i)] the complete diagonal $\Delta_{1 \ldots k}$, or 
\item[ii)] the intersection of any two partial diagonals $\Delta_{\calI} \cap \Delta_{\calJ}$, $\calI \cap \calJ = \emptyset$, 
$\calI \cup \calJ = \{1, \ldots, k\}$ (we include the case where either $|\calI|=1$ or $|\calJ| = 1$), 
\end{itemize}
These two cases correspond to the degenerations where the $k$ points merge into either one or two points. 
\end{remark}

The basic existence and uniqueness result in this setting is well-known:
\begin{lemma}[\cite{Tr, LT}]\label{l:uniqueness}
If $k \geq 3$ and all cone angles are less than $2\pi$, then there exists a unique spherical cone metric on $M$ 
provided $\chi(M,\vec \beta)>0$, and when $M = \bS^2$, the cone angles also satisfy the Troyanov condition:
\[
\mbox{for each $j$,}\ \  \beta_j - 1 > \sum_{ i \neq j} (\beta_i-1).
\]
If $k=2$ and $M=\bS^{2}$, there is a spherical metric if and only if $\beta_1 = \beta_2$, and this metric
is unique up to conformal dilation. 
\end{lemma}

\begin{theorem}\label{t:sph}
Let $k \geq 3$, restrict to the open dense subset $\calC_k'$ of the configuration family and fix any $k$-tuple $\vec \beta$ of
cone angle parameters lying in the Troyanov region.  Then the family of fiberwise spherical metrics with these fixed cone angles 
is polyhomogeneous on $\calC_{k}'$.
\end{theorem}

\begin{proof}
As in Theorem \ref{t:hyp}, it suffices to work locally near any point $q$ in a corner $\cap_{i=1}^{\ell}\frakC_{\calI_{\ell}}$,
where $\calI_{1}\supset \dots \supset \calI_{\ell}$ are the nodes associated to $q$ in the associated tree.

Start with a model metric $g_{0,\frakp}$, defined as in ~\eqref{e:model}, which is obtained by gluing the family of (lifted)
flat conic metrics to the family of spherical metrics away from the coalescence locus: 
$$
g_{0,\frakp}=\chi g_{0,\frakp}^{fl} + (1-\chi) g^{sph}_{\frakp'}.
$$
Next construct approximate solutions $u_{N}$ using exactly the same argument as before. Note that we must use 
Lemma~\ref{l:uniqueness} when determining the expansion on the original surface,  and replace $\sinh \tilde r$ 
by $\sin \tilde r$ in~\eqref{e:sinh}. Otherwise, the steps are carried out in exactly the same way. 
Altogether, taking a Borel sum of the resulting series, we obtain an approximate solution $\tilde u$. 

We next obtain the correction term $v$, which solves
$$
\Delta_{\tilde g_{0}}v+e^{2v}+K_{\tilde g_{0}}=0,  \ \tilde g_{0}=e^{2\tilde u}g_{0,\frakp}.
$$
By assumption, the linearization of this equation, $\Delta_{\tilde g_{0}}-2$, is invertible on the singular fiber,
hence we may use Lemma~\ref{l:invertible}.  Using continuity of the eigenvalues, $\Delta_{\tilde g_{0}}-2$ remains invertible
on nearby fibers. Hence by the implicit function theorem there exists a solution $v$ which depends smoothly
on all parameters. 

The arguments of Proposition~\ref{p:maxPr} and Proposition~\ref{p:poly} now show that $|v| \leq C_{0}\rho^{\ell}$ 
for any $\ell\in \NN$, with similar estimates for all $b$-derivatives, so $v$ vanishes to infinite order at these boundaries.
This shows that $u = \tilde{u} + v$, and hence 
$
g_{\frakp}=e^{2u}g_{0,\frakp},
$
is polyhomogeneous.
\end{proof}

\appendix
\section{b-fibrations and polyhomogeneity}
In this appendix we recall some basic facts about manifold with corners, $b$-fibrations, and polyhomogeneous functions. For details 
we refer to~\cite{Me} or \cite[Appendix]{M-edge}. 
\begin{definition}
A space $X$ is called a \emph{manifold with corners} of dimension $n$ if, at each point there exists a nonnegative integer $k$ such 
that $X$ is modeled diffeomorphically near that point by a neighborhood of the origin in the product $(\RR^{+})^{k}\times \RR^{n-k}$.
The boundary faces and corners of $X$ correspond to the boundaries and corners in these local representations. We list the
boundary hypersurfaces of $X$ as $\{X_i\}_{i=1}^J$. The corners of codimension $\ell$ are the submanifolds 
$X_{i_1} \cap \ldots \cap X_{i_\ell}$. We require that all boundary faces and corners be embedded; this is merely a simplifying 
assumption to avoid talking about special cases.  However, with this assumption we can choose for each $i$ a smooth function
$\rho_i$ which is positive on $X \setminus X_i$ and which vanishes simply at $X_i$. This is called a boundary defining function
for that face. If $s = (s_1,\ldots, s_J)\in \CC^J$, then $\rho^s := \rho_1^{s_1} \ldots \rho_J^{s_J}$. 
\end{definition}
The spaces $\calE_{k}$ and $\calC_{k}$ constructed in~\S\ref{s:resolution} are manifolds with boundaries, with boundary 
{\it faces} $\{\frakC_{\calI}\}$ and $\{F_{\calI}\}$, where in each case the index $\calI$ ranges over all subsets of $\{1,\dots, k\}$.
We do not introduce a special notation for the corners. 

There is a whole ecosystem of geometric and analytic objects naturally defined on a manifold with corners. One key
object is the following.
\begin{definition}
The space of $b$-vector fields $\calV_b(X)$ on $X$ is the space of all smooth vector fields on $X$ which are tangent to all boundaries. 
The $b$-tangent bundle ${}^{b}TX$ is a canonical bundle whose full space of smooth sections is precisely $\calV_b(X)$. Its dual
is the $b$-cotangent bundle ${}^{b}T^{*}X$.
\end{definition}
The local coordinate expression of a general $b$-vector field is given in~\S\ref{ss:bvector}. 

We next come to the first of two most natural replacements for the space of smooth functions on $X$. 
\begin{definition}[Conormality] The bounded conormal functions on $X$ is the space
\begin{equation}
\calA^{0}(X): = \{u: V_{1} \dots V_{\ell} u \in L^{\infty}(X), \ \forall \, V_{i}\in \calV_{b},  \ \mbox{and}\ \ell\in \NN\}.
\end{equation}
For any pair of multi-indices $s \in \CC^J$, $p \in \NN_0^J$, we also define
\begin{equation}
\calA^{s,p}(X):=\rho^{s}(\log \rho)^{p}\calA^{0}(X).
\end{equation}
Finally, 
\begin{equation}
\calA^{*}(X):=\bigcup\limits_{s,p} \calA^{s,p}(X).
\end{equation}
\end{definition}
Note that any element of $\calA^*(X)$ is smooth in the interior of $X$. 

There is a subclass of $\calA^*(X)$ which is more useful in practice: the space of polyhomogeneous functions. These are 
associated to an index family:
\begin{definition} An \emph{index set} $\calE$ is a discrete subset $\{(s_j, p_j) \in \CC \times \NN_{0}$ such that
\begin{equation}
\mbox{if}\ (s_{j},p_{j})\in E, \ \mbox{and}\ |(s_{j},p_{j})|\rightarrow \infty \ \mbox{then}\ \Re(s_{j}) \rightarrow \infty.
\end{equation}

Now suppose that $\calE=\{E_{1}, \dots, E_{J}\}$ is a collection of index sets, one for each boundary face of $X$; we call
this an index family.  The space of \emph{polyhomogeneous} functions with index family $\calE$ (or more simply
$\calE$-smooth functions) on $X$, $\calA_{\phg}^{\calE}(X)$, consists of the elements of $\calA^{*}(X)$ with asymptotic
expansions with exponents given by the elements of $\calE$, i.e., $u \in \calA_{\phg}^{\calE}$ if
\[
u \sim \sum_{ (s,p) \in \calE} \rho^s (\log \rho)^p u_{s,p},
\]
where each $u_{s,p}$ is a smooth function the relevant corner of $X$.  This expansion is meant in the classical sense, 
and has a product type at the corners; it is a tractable replacement for a Taylor expansion at the faces and corners. 

We write $u \in \calA_{\phg}^*(X)$ when the index set is not specified (or is obvious from the context).
\end{definition}

It follows readily from this definition, and is useful in applications, to note that if $u \in \calA_{\phg}$, then for
each $i$, any coefficient $u^{(i)}_{s,p}$, which is simply the coefficient of $\rho_i^{s_i}(\log \rho_i)^{p_i}$ in the
expansion at $X_i$, is itself a polyhomogeneous function on $X_i$ with index family $\calE^{(i)}$ obtained
by omitting $\calE_i$.  


There is a convenient criterion for polyhomogeneity. 
\begin{proposition}\label{p:appendix}
If $u\in \calA^{*}(X)$ and for every $N > 0$, 
\begin{equation}
u-\sum_{\Re s<N, (s,p)\in \calE}\rho^{s}(\log \rho)^{p}u_{s,p}\in \rho^{N-1}\calA^{0}(X),
\end{equation}
then $u \in \calA_{\phg}^{\calE}(X)$. 
\end{proposition}

We also define a distinguished class of mappings between manifolds with corners. 
\begin{definition}
Let $X, Y$ be manifolds with corners, with corresponding sets of boundary defining functions $\{r_{i}\}$ and $\{\rho_{j}\}$,
respectively, associated to the enumerations of boundary faces $\{X_i\}_{i \in \calI}$, $\{Y_j\}_{j \in \calJ}$. 
A map $f:X\rightarrow Y$ is called a \emph{b-fibration} if the following conditions are satisfied:
\begin{enumerate}
\item[(b-map)] For any index $j \in \calJ$, the pullback of the corresponding boundary defining function $\rho_{j}$ is a smooth
nonvanishing multiple of the product of boundary defining functions of $X$, i.e., 
$$
f^{*}(\rho_{j})=h\prod_{i=1}^{I} r_{i}^{e(i,j)}, \ h>0, \ e(i,j)\in \NN.
$$ 
The exponent set $e(i,j)$ is called the lifting matrix of $f$.
\item[(b-submersion)] At each boundary point $p\in \partial X$, the map ${}^{b}f_{*}: T_{p}X\rightarrow T_{f(p)}Y$ is surjective.
\item[(b-fibration)] The lifting matrix $(e(i,j))$ has the property that for each $i$ there is at most one $j$ such that 
$e(i, j)\neq 0.$  In other words, no boundary hypersurfaces of $X$ is mapped into a corner of $Y$.
\end{enumerate}
\end{definition}


\begin{thebibliography}{99}

\bibitem{AlbinMelrose}
Pierre Albin, and Richard Melrose. 
\newblock Resolution of smooth group actions. 
\newblock {\em Contemp. Math} 535 (2011): 1-26.

\bibitem{CM1}
Alessandro Carlotto, and Andrea Malchiodi.
\newblock A class of existence results for the singular Liouville equation.
\newblock {\em Comptes Rendus Mathematique}, 349(3-4):161--166, 2011.


\bibitem{CM2}
Alessandro Carlotto, and Andrea Malchiodi.
\newblock Weighted barycentric sets and singular Liouville equations on compact
  surfaces.
\newblock {\em Journal of Functional Analysis}, 262(2):409--450, 2012.

\bibitem{BMM}
 Daniele Bartolucci,  Francesca De Marchis, and Andrea Malchiodi.
\newblock Supercritical conformal metrics on surfaces with conical singularities. 
\newblock {\em Int. Math. Res. Not.}, no. 24, 5625--5643, 2011.

\bibitem{CL1}
Chiun-Chuan Chen, and Chang-Shou Lin. 
\newblock Topological degree for a mean field equation on Riemann surfaces.
\newblock{\em Communications on pure and applied mathematics 56}, no. 12 (2003): 1667-1727.

\bibitem{CL2}
Chiun-Chuan Chen,  and Chang-Shou Lin. 
\newblock Mean field equation of Liouville type with singular data: topological degree. 
\newblock {\em Communications on Pure and Applied Mathematics 68}, no. 6 (2015): 887-947.

\bibitem{Dey}
Subhadip Dey.
\newblock Spherical metrics with conical singularities on 2-spheres. 
\newblock {\em Geometriae Dedicata} (2017): 1-9.


\bibitem{Ere2}
Alexandre Eremenko. 
\newblock Metrics of positive curvature with conic singularities on the sphere. 
\newblock {\em Proceedings of the American Mathematical Society},  132.11 (2004): 3349-3355.


\bibitem{Ere1}
Alexandre Eremenko.
\newblock Co-axial monodromy.
\newblock {\em arXiv preprint arXiv: 1706.04608}, 2017.

\bibitem{EGT}
Alexandre Eremenko, Andrei Gabrielov, and Vitaly Tarasov. 
\newblock Metrics with conic singularities and spherical polygons. 
\newblock {\em Illinois Journal of Mathematics}, 58.3 (2014): 739-755.


\bibitem{Ka}
Michael Kapovich.
{\em Branched covers between spheres and polygonal inequalities in simplicial trees.} 
2017.

\bibitem{KS} 
Chris Kottke, and Michael Singer.
\newblock Partial compactification of monopoles and metric asymptotics.
\newblock {\em arXiv preprint arXiv:1512.02979}, 2015.

\bibitem{LLYZ}
Youngae Lee, Chang-shou Lin, Wen Yang, and Lei Zhang. 
\newblock Degree counting for Toda system with simple singularity: one point blow up. 
\newblock{\em arXiv preprint arXiv:1707.07156}, 2017.

\bibitem{LT}
Feng Luo, and Gang Tian.
\newblock Liouville equation and spherical convex polytopes.
\newblock {\em Proceedings of the American Mathematical Society},
  116(4):1119--1129, 1992.

\bibitem{M-edge}
Rafe Mazzeo.
\newblock Elliptic theory of differential edge operators I.
\newblock {\em Communications in Partial Differential Equations},
16, no. 10 (1991): 1615-1664, 1991.


\bibitem{MW}
Rafe Mazzeo, and Hartmut Weiss.
\newblock Teichm\"uller theory for conic surfaces.
\newblock In {\em Geometry, Analysis and Probability, In Honor of Jean-Michel
  Bismut}, volume 310 of {\em Progress in Mathematics}, pages 127--164.
  Birkh\"auser Basel, 2017.
  
  
\bibitem{Mc}
Robert~C McOwen.
\newblock Point singularities and conformal metrics on Riemann surfaces.
\newblock {\em Proceedings of the American Mathematical Society},
  103(1):222--224, 1988.
  
\bibitem{Me}
Richard Melrose. 
\newblock Differential analysis on manifolds with corners, 
\newblock Book, In Preparation.
  
  
\bibitem{MP}
Gabriele Mondello, and Dmitri Panov.
\newblock Spherical metrics with conical singularities on a 2-sphere: angle
  constraints.
\newblock {\em International Mathematics Research Notices},
  2016(16):4937--4995, 2016.


\bibitem{SCLX}
Jijian Song, Yiran Chen, Bo Li, and Bin Xu. 
\newblock{Drawing cone spherical metrics via Strebel differentials.}
\newblock{\em arXiv preprint arXiv:1708.06535 (2017).}


\bibitem{Tr}
Marc Troyanov.
\newblock Prescribing curvature on compact surfaces with conical singularities.
\newblock {\em Transactions of the American Mathematical Society},
  324(2):793--821, 1991.  
  
  
\bibitem{Tr2}
Marc Troyanov.
\newblock Metrics of constant curvature on a sphere with two conical singularities.
\newblock {\em Lect. Notes Math}, 
1410 (1989): 296-308.

\bibitem{UY}
Masaaki Umehara, and Kotaro Yamada. 
\newblock Metrics of constant curvature 1 with three conical singularities on the 2-sphere. 
\newblock {\em Illinois Journal of Mathematics} 44.1 (2000): 72-94.

\end{thebibliography}
\end{document}